\documentclass[a4paper]{amsart}
\usepackage[utf8]{inputenc}
\usepackage[T1]{fontenc}
\usepackage{lmodern}
\usepackage{amsmath, amsthm, amssymb, amsfonts}
\usepackage[all]{xy}
\usepackage{nicefrac,mathtools}
\usepackage[shortlabels, inline]{enumitem}
\usepackage{microtype}
\usepackage{hyperref}
\usepackage{graphicx}
\usepackage[all]{xy}
\usepackage{xcolor}
\usepackage{tikz-cd}
\usepackage[lite]{amsrefs}
\usepackage{thmtools}

\numberwithin{equation}{section}
\theoremstyle{plain}
\newtheorem{theorem}[equation]{Theorem}

\newtheorem{lemma}[equation]{Lemma}
\newtheorem{proposition}[equation]{Proposition}

\newtheorem{corollary}[equation]{Corollary}

\theoremstyle{definition}
\newtheorem{definition}[equation]{Definition}
\newtheorem{notation}[equation]{Notation}

\theoremstyle{remark}
\newtheorem{remark}[equation]{Remark}

\newtheorem{example}[equation]{Example}

\newcommand*{\Cont}{\mathrm C}
\newcommand*{\Contc}{\mathrm{C_c}}
\newcommand{\Contz}{\mathrm{C_0}}
\newcommand{\base}[1][G]{{#1}^{(0)}}
\newcommand{\dd}{\mathrm{d}}
\newcommand{\7}{\backslash}

\newcommand{\tens}{\mathbin{\otimes}}

\newcommand{\Id}{\textup{Id}}
\newcommand{\supp}{{\textup{supp}}}

\newcommand{\Mlp}{\mathrm{m}}
\newcommand{\Dia}{\mathrm{dia}}
\newcommand{\Ist}{\text{Fix}}
\newcommand{\UCr}{\mathbb{S}^1}
\newcommand{\Bor}{\mathrm{Bor}}
\newcommand{\Ho}{\mathrm{Ho}}
\newcommand{\HoC}{\mathrm{Ho,Co}}
\newcommand{\Cov}[1][A]{\mathcal{#1}}
\DeclareMathOperator{\Img}{\textnormal{Im}}

\newcommand*{\Hilm}[1][H]{\mathcal #1}
\newcommand*{\Hils}[1][M]{\mathcal #1}
\newcommand*{\Mult}{\mathcal M}
\newcommand{\Bound}{\mathbb B}
\newcommand*{\Cst}{\textup C^*}
\newcommand{\Cred}{\textup{C}^*_{\textup{r}}}

\newcommand{\Comp}{\mathbb{K}}
\newcommand{\TComp}{\mathbb{K}_{\mathrm{T}}}
\newcommand*{\inpro}[2]{\langle#1, #2\rangle}
\newcommand{\abs}[1]{\lvert #1\rvert}
\newcommand{\norm}[1]{\lvert\lvert #1\rvert\rvert}

\newcommand{\KK}{\rm{KK}}
\newcommand{\C}{\mathbb{C}}
\newcommand{\R}{\mathbb{R}}
\newcommand{\Z}{\mathbb{Z}}

\newcommand{\N}{\mathbb{N}}

\newcommand{\homeo}{\approx}
\newcommand{\iso}{\simeq}
\newcommand{\LSec}{\mathcal{S}}
\newcommand{\Lone}{\mathcal{L}^1}
\newcommand{\Sph}{\mathbb{S}}

\newcommand*{\Star}{$^*$\nobreakdash-}
\newcommand*{\nb}{\nobreakdash}

\newcommand*{\defeq}{\mathrel{\vcentcolon=}}
\newcommand{\inverse}{^{-1}}
\newcommand{\RN}{\mathfrak{D}}
\newcommand{\mi}{\mathrm{i}}
\newcommand{\me}{\mathrm{e}}
\newcommand{\etale}{{\'e}tale}
\newcommand{\Pt}{\{*\}}
\newcommand{\Ext}{\mathcal{E}}
\newcommand{\ExtS}{\mathcal{E}^{\mathrm{Set}}}
\newcommand{\ExtB}{\mathcal{E}^{\mathrm{Borel}}}

\newcommand{\Res}{\mathcal{R}}

\title[Proper topological correspondences]{Locally free actions of groupoids and proper topological correspondences}
\author{Rohit Dilip Holkar}
\email{rohit.d.holkar@gamil.com}
\address{Department of Mathematics, Indian Institute of Science Education and Research, Dr.\,Homi Bhabha Road, NCL Colony, Pashan, Pune 411008, India.}

\thanks{\emph{Subject class.} 22D25, 22A22, 46H35, 47L30, 46L08, 19K35, 19L99.}

\begin{document}
\begin{abstract}
Let \((G,\alpha)\) and \((H,\beta)\) be locally compact Hausdorff groupoids with Haar systems, and let \((X,\lambda)\) be a topological correspondence from \((G,\alpha)\) to \((H,\beta)\) which induce the \(\Cst\)\nb-correspondence \(\Hilm(X)\colon \Cst(G,\alpha)\to \Cst(H,\beta)\). We give sufficient topological conditions which when satisfied the \(\Cst\)\nb-correspondence \(\Hilm(X)\) is proper, that is, the groupoid \(\Cst(G,\alpha)\) acts on the \(\Cst(H,\beta)\)\nb-Hilbert module \(\Hilm(X)\) via the comapct operators. Thus a proper topological correspondence produces an element in \(\KK(\Cst(G,\alpha),\Cst(H,\beta))\).
\end{abstract}

\maketitle
\tableofcontents

\section*{Introduction}
\label{sec:intro}

Given \(\Cst\)\nb-algebras  \(A\) and \(B\), a \(\Cst\)\nb-correspondence from \(A\) to \(B\) is a pair \((\Hils, \phi)\) where \(\Hils\) is a Hilbert \(B\)\nb-module, and \(\phi\colon A\to \Bound(\Hils)\) is a nondegenerate representation. We call the \(\Cst\)\nb-correspondence \((\Hils,\phi)\)  \emph{proper} if the representation \(\phi\colon A\to \Comp(\Hils)\). Proper correspondences are important and studied for various purposes. In this article, we shall denote a \(\Cst\)\nb-correspondence merely by the Hilbert module in its definition. We shall not come across an occasion where the representation needs to be explicitly spelled out.

In~\cite{Holkar2017Construction-of-Corr}, topological correspondences of locally compact groupoids equipped with Haar systems are defined. Let \((G,\alpha)\) and \((H,\beta)\) be locally compact groupoids equipped with Haar systems. A topological correspondence from \((G,\alpha)\) to \((H,\beta)\) is a pair \((X,\lambda)\) where \(X\) is a \(G\)\nb-\(H\)-bispace with the  \(H\)\nb-action  proper, \(\lambda\defeq\{\lambda_u\}_{u\in\base[H]}\) is a proper \(H\)\nb-invariant family of measures along the right momentum map \(s_X\colon X\to \base[H]\) and each measure \(\lambda_u\) in \(\lambda\) is  \((G,\alpha)\)\nb-quasi-invariant. The main theorem in~\cite{Holkar2017Construction-of-Corr}*{Theorem 2.10} asserts that a certain completion of \(\Contc(X)\) gives a \(\Cst\)\nb-correspondence \(\Hilm(X,\lambda)\) from \(\Cst(G,\alpha)\) to \(\Cst(H,\beta)\). 

For the above topological correspondence \((X,\lambda)\), let \(\base\xleftarrow{r_X} X\xrightarrow{s_X}\base[H]\) be the momentum maps. The \(\Cst\)\nb-algebra  \(\Comp(\Hilm((X,\lambda))\) of compact operators on the Hilbert \(\Cst(H,\beta)\)-module \(\Hilm(X,\lambda)\) can be described using a topological, namely, a hypergroupoid equipped with a Haar system. To do this, one observes that the right diagonal action of \(H\) on the fibre product \(X\times_{s_X,}X\) is proper. In~\cite{Renault2014Induced-rep-and-Hpgpd}, Renault shows that the quotient space \((X\times_{s_X,\base[H],s_X})/H\) is a hypergroupoid equipped with a Haar system \(\tilde{\lambda}\). The Haar system \(\tilde{\lambda}\) is induced by \(\lambda\). The \(\Cst\)\nb-algebra of this hypergroupoid, \(\Cst(X\times_{s_X,\base[H],s_X})/H), \tilde{\lambda}\) is the \(\Cst\)\nb-algebra of compact operators on \(\Hilm(X,\lambda)\). We revise this result in Section~\ref{comp-op}; this was written for locally compact, Hausdorff and second countable groupoids and spaces in~\cite{mythesis}.

The notion of topological correspondence in~\cite{Holkar2017Construction-of-Corr} generalises many existing notions of topological correspondences in the literature including Jean-Louis Tu's notion of \emph{locally proper generalised morphism} in~\cite{Tu2004NonHausdorff-gpd-proper-actions-and-K}. In~\cite{Tu2004NonHausdorff-gpd-proper-actions-and-K}, Tu defines \emph{a proper locally proper generalised morphism} which is a \emph{proper topological correspondences}. That is, he adds an extra conditions to his  topological correspondence so that the \(\Cst\)\nb-correspondence \(\Hilm(X,\lambda)\) it produces is proper. This work of Tu raises a question, when is a topological correspondence defined in~\cite{Holkar2017Construction-of-Corr} proper? To be precise, under what extra hypothesis on the topological correspondence in~\cite{Holkar2017Construction-of-Corr} is the corresponding \(\Cst\)\nb-correspondence proper?

Apart from Tu's above-mentioned work, this question has been discussed and answered for special cases. Marth-Stadler and O'uchi define a topological correspondence and a proper one in~\cite{Stadler-Ouchi1999Gpd-correspondences}. This work is a special case of the work in~\cite{Tu2004NonHausdorff-gpd-proper-actions-and-K}. In~\cite{Muhly-Tomforde-2005-Topological-quivers}, Muhly and Tomford define topological quivers, which are special types of topological correspondences (see~\cite{Holkar2017Construction-of-Corr}*{Example 3.3}). In the same article, they also characterise proper topological quivers, see~\cite{Muhly-Tomforde-2005-Topological-quivers}*{Theorem 3.11}.

Now we briefly elaborate the above works putting Tu's at the end. Let \(Y\) be a locally compact Hausdorff space. Slightly modifying~\cite{Muhly-Tomforde-2005-Topological-quivers}*{Definition 3.1}, we may say that a topological quiver from \(Y\) to itself is a quadruple \((X,b,f,\lambda)\) where \(b,f\colon X\to Y\) are continuous maps and \(\lambda\) is a continuous family of measures along \(f\).~\cite{Muhly-Tomforde-2005-Topological-quivers}*{Theorem 3.11} says that the quiver is proper if and only if \(f\) is a local homeomorphism and \(b\) is proper. 

Let \(Y\) and \(Z\) be locally compact Hausdorff spaces. Then as in~\cite{Buss-Holkar-Meyer2016Gpd-Uni-I}, in general, one defines that a quiver from \(Y\) to \(Z\) is a quadruple \((X,b,f,\lambda)\) where \(b\colon X\to Y\) and \(f\colon X\to Z\) are continuous maps, and \(\lambda\) is a continuous family of measures along \(f\). One may check that~\cite{Muhly-Tomforde-2005-Topological-quivers}*{Theorem 3.11} is valid for this slightly generalised notion of quivers also. If one thinks of the spaces \(Y\) and \(Z\) above as the trivial groupoids, then both the above of quivers are topological correspondences, see~\cite{Holkar2017Construction-of-Corr}*{Example 3.3}.

Tu~\cite{Tu2004NonHausdorff-gpd-proper-actions-and-K} works with locally compact groupoids equipped with Haar systems and the space involved in the (proper) topological correspondence is locally compact but not necessarily Hausdorff. Marta-Stadler and O'uchi's work~\cite{Stadler-Ouchi1999Gpd-correspondences} involves Hausdorff groupoids equipped with Haar systems and Hausdorff spaces; their proper correspondences is a special case of Tu's work, hence we focus on~\cite{Tu2004NonHausdorff-gpd-proper-actions-and-K}. Let \((G,\alpha)\) and \((H,\beta)\) be locally compact groupoids equipped with Haar systems. A topological correspondence from \((G,\alpha)\) to \((H,\beta)\) in the sense of Tu is a \(G\)\nb-\(H\)-bispace \(X\) such that
\begin{enumerate}[i), leftmargin=*]
\item both the actions are proper,
\item the action of \(G\) is free,
\item the right momentum map \(s_X\colon X\to\base[H]\) induces an isomorphism \([s_X]\colon X/H\to\base[H]\).\label{intro-Tu-1}
\end{enumerate}
Property~\ref{intro-Tu-1} above is equivalent to
\begin{enumerate}[resume, label=iii')]
\item the map \(\Mlp_{s_X}\colon G\times_{s_G,\base, r_X}X\to X\times_{s_X,\base[H],s_X}X\),
\[
\Mlp_{s_X}(\gamma,x)\mapsto (\gamma x,x),
\]
 is a homeomorphism\label{intro-Tu-2} where \(s_G\) is the source map of \(G\), \(r_X\colon X\to\base\) and \(s_X\colon X\to\base[H]\) are the momentum maps for the action of \(G\) and \(H\), respectively, and the domain and codomain of the function \(\Mlp_{s_X}\) are the obvious fibre products.
\end{enumerate}
\cite{Holkar2017Construction-of-Corr}*{Example 3.7} shows that this is a topological correspondence in our sense, that is, in the sense of~\cite{Holkar2017Construction-of-Corr}*{Definition 2.1}. \cite{Tu2004NonHausdorff-gpd-proper-actions-and-K}*{Definition 7.6} says that the above correspondence is proper if the map \([r_X]\colon X/H\to \base\) induced by the momentum map \(r_X\colon X\to \base\) for the action of \(G\) is proper. \cite{Tu2004NonHausdorff-gpd-proper-actions-and-K}*{Theorem 7.8} proves that if the above topological correspondence \(X\) is proper, then so is the \(\Cst\)\nb-correspondence \(\Hilm(X)\colon \Cred(G,\alpha)\to \Cred(H,\beta)\).

Now we discuss the three conditions in the definition of a proper correspondence which is proposed in this article. Let \((X,\lambda)\) be a topological correspondence from a locally compact groupoid equipped with a Haar system \((G,\alpha)\) to another one, say \((H,\beta)\). Let \(r_X\) and \(s_X\) be the momentum maps for the actions of \(G\) and \(H\), respectively, on \(X\). First of all, one may expect from preceding literature survey that the map \([r_X]\colon X/H\to \base\) should be proper. This is true; this is one of the conditions we need. However, this is not a sufficient condition while dealing a general topological correspondence. For example, if \(G=H=\Pt\), the trivial group, and \((X,\lambda)\) is a compact measure space, then \([r_X]\) is proper. In this case, \(\Hilm(X,\lambda)=\mathcal{L}^2(X,\lambda)\) and \(\Cst(G)=\Cst(H)=\C\). The action of \(\C\) on \(\Hilm(X)\) is by scalar multiplication which has the identity operator \(1\). Hence, in general, \(\Hilm(X,\lambda)\colon \C\to \C\) is not a proper correspondence. 

Secondly, since there are families measures involved one may expect that there should be condition(s) that relate the family of measures on the bispace and the Haar system on the left groupoid. This is a technical condition; this is the third condition in Definition~\ref{def:prop-corr}.

Thirdly and finally, an interesting and not-at-all-obvious condition is the rephrasing~\ref{intro-Tu-2} of the property~\ref{intro-Tu-1} appearing in the definition of Tu's proper correspondence above. It is a \emph{classical} condition that appears in the definition of a groupoid equivalence (\cite{Muhly-Renault-Williams1987Gpd-equivalence}) and may other notions of groupoid morphisms. This property gives a family of measures on \(X\) as observed in~\cite{Holkar2017Construction-of-Corr}*{Exmaple 3.7}. What is so interesting about this property? In its other form, namely~\ref{intro-Tu-2}, the property proves very useful.

With the same notations as in the last paragraph, let \(b\in\Contc(G)\) and \(\xi\in\Contc(X)\). Then the action of \(b\) on \(\xi\) that gives the representation of \(\Cst(G,\alpha)\) on \(\Hilm(X,\lambda)\) is given by
\[
b\xi(x)=\int_G b(\gamma)\xi(\gamma\inverse x)\Delta(\gamma,\gamma\inverse x)\,\dd\alpha^{r_X(x)}(\gamma)
\]
where \(\Delta\) is the adjoining function for the correspondence \((X,\lambda)\). Now, we aim to assign \(b\) an element \(\tilde{b}\) in \(\Contc((X\times_{s_X,\base[H],s_X}X)/H)\) such that \(\tilde{b}\xi=b\xi\). For this purpose, using the fact that \([r_X]\) is proper, we choose a\emph{dummy} function \(t\) in \(\Contc(X)\). Then \(b\otimes_{\base} t\in \Contc(G\times_{\base}X)\). If \(\Mlp_{s_X}\) were a homeomorphism, then \((b\otimes_{\base} t)\circ\Mlp_{s_X}\inverse\) is an element of \(\Contc(X\times_{s_X,\base[H],s_X}X)\). If one chooses the dummy function \(t\) carefully, then after averaging over \(H\), the image of \(b\otimes_{\base} t\) in \(\Contc((X\times_{s_X,\base[H],s_X}X)/H)\) serves as the required function \(\tilde{b}\).

Note that if \(\Mlp_{s_X}\) is a homeomorphism, then the action of \(G\) on \(X\) is free as well as proper. A bit of more work shows that the above argument goes through when \(\Mlp_{s_X}\) is a homeomorphism onto its image. An action of \(G\) on \(X\) for which \(\Mlp_{s_X}\) is a homeomorphism onto its image are called \emph{basic} in~\cite{Meyer-Zhu2014Groupoids-in-categories-with-pretopology}. However, the following simple example shows that for us even basic actions too much to ask for. Let \(\Sph^1\) denote the unit circle; consider this compact space as the trivial groupoid. Let the multiplicative group of order two, \(\Z/2\Z\defeq \{1,-1\}\), act on the space \(\Sph^1\) from left by \((-1)\cdot z =\bar{z}\) for \(-1\in \Z/2\Z\) and \(z\in \Sph^1\). Let \(\Sph^1\)  act on itself trivially from right. The momentum map for this action is the identity map \(\Id_{\Sph^1}\colon \Sph^1\to \Sph^1\). Let \(\tau_{\Id_{\Sph^1}}\) be the family of measures along \(\Id_{\Sph^1}\) which consists of point masses. Then \((\Sph^1, \tau_{\Id_{\Sph^1}})\) is a topological correspondence from \(\Z/2\Z\) to \(\Sph^1\); the adjoining function of this correspondence is the constant function \(1\). The topological  correspondence \((\Sph^1, \tau_{\Id_{\Sph^1}})\)  gives the \(\Cst\)\nb-correspondence \(\Cont(\Sph^1)\colon \Cst(\Z/2\Z) \to \Cont(\Sph^1)\). Note that here we make the \(\Cst\)-algebra \(\Cont(\Sph^1)\) a Hilbert module over itself in the obvious way. Since this \(\Cst\)\nb-algebra is unital, we have \(\Comp(\Cont(\Sph^1))=\Mult(\Cont(\Sph^1))=\Cont(\Sph^1)\). Therefore, the \(\Cst\)\nb-correspondence \(\Cont(\Sph^1)\colon \Cst(\Z)\to \Cont(\Sph^1)\) is proper.

In above example, \(\Mlp_{s_{\Sph^1}}\colon \Z/2\Z\times \Sph^1\to \Sph^1 \times \Sph^1\) is the map \((\pm1,z)\mapsto (z,z) \text{ or } (\bar{z},z)\) which is not a homeomorphism, even, onto its image as \((-1)\cdot(1+0i)=1\cdot(1+0i)\) where \(1, -1\in \Z/2\Z\) and \(1+0i\in \Sph^1\). However, on may check that this map is a local homeomorphism onto its image: for \(\pm1\in \Z/2\Z\), \(\Mlp_{{\Sph^1}}|_{\{\pm1\}\times {\Sph^1}}\) is a homeomorphism onto its mage; inverse of this restriction is \((z',z)\mapsto (\frac{z'}{z},z)\) where \((z',z)\in \Img(\Mlp_{{\Sph^1}}|_{\{\pm1\}\times {\Sph^1}})\).

This and similar examples motivate us to add the condition that \emph{\(\Mlp_{s_X}\) is a local homeomorphism onto its image}. And this serves our purpose well.
\medskip

Once the maps which are local homeomorphisms onto their images are in picture,  we have to study these local homeomorphisms: we have to study how does a map of this type behave
\begin{enumerate*}[(i)]
\item with respect extensions of continuous function along it, 
\item with respect to induction of measures along it.
\end{enumerate*}
We study these technical issues in Sections~\ref{sec:measure-th}.

One may put all these conditions and above study together and state the main result in this article. However, we investigate when is the map \(\Mlp_{s_X}\) above a local homeomorphism. This question breaks down into two pieces:
\begin{enumerate*}[(i)]
\item when is \(\Mlp_{s_X}\) locally one-to-one?
\item when is \(\Mlp_{s_X}\) open onto its image?
\end{enumerate*}
The second question has no concrete answer and examples show that \(\Mlp_{s_X}\) is open has to be a part of data. 
The first question leads us to the study of the locally free action of groupoids; this study is done in~Section~\ref{sec:locally-free-actions}. While studying the locally free actions, we also define transitive actions of groupoids in Section~\ref{sec:act-gpd} which facilitate rephrasing a classical condition more theoretically, see Remark~\ref{rem:relation-between-sX-quo-sX-Mlp}.

Since Muhly and Tomford~\cite{Muhly-Tomforde-2005-Topological-quivers} not only \emph{define} a proper topological correspondence but also \emph{characterise} it, we wish to see that up to what level can our definition of a proper correspondence reproduce their results, we study measures which are concentrated on certain sets (Section~\ref{sec:measure-th}).

What do we finally achieve? We define a proper topological correspondence (Definition~\ref{def:prop-corr}) and show that a proper topological correspondence of groupoids equipped with Haar systems produce a proper \(\Cst\)\nb-correspondence of full groupoid \(\Cst\)\nb-algebras (Theorem~\ref{thm:prop-corr-main-thm}). We study \emph{locally free} actions of groupoid: we define locally free and strongly locally free actions of groupoids and show that they do not mean the same (Example~\ref{exa:local-free-and-not-str-loc-free-action}). However, in the case of groups, these notions are same.
While proving this, the example, Example~\ref{exa:local-free-and-not-str-loc-free-action}, shows that a groupoid with discrete fibres need not be {\etale}. Moreover, while studying measures concentrated on a set, we prove Lemma~\ref{lem:conc-measu-equi} which gives different characterisations that when is a measure concentrated on a set; this lemma is not a deep work, but it comfortable to have it. In Section~\ref{sec:exa}, we show that all the definitions of various proper topological correspondences mentioned earlier fit in Definition~\ref{def:prop-corr}. Following is a detail discussion of results in this section.

We discuss the ({\etale}) proper correspondences of spaces and {\etale} groupoids in detail in Section~\ref{sec:ex-space-case} and~\ref{sec:ex-egpd-case}, respectively. We show that in a proper topological correspondence of spaces in the sense of Definition~\ref{def:prop-corr}, the family of measures on the middle space consists of atomic measures and has full support, see~\ref{lem:prop-quiver-labda-is-atomic}. Reader may compare this with ~\cite{Muhly-Tomforde-2005-Topological-quivers}*{Theorem 3.11}. This result generalises to {\etale} groupoids also, Lemma~\ref{lem:etale-corr-full-supp}.
 
Let \(X,Y\) and \(Z\) be spaces, and \(Y\xleftarrow{b} X\xrightarrow{f} Z\) be continuous maps where \(f\) is a local homeomorphism. Then \(X\) carries a continuous family of counting measures, \(\tau_{s_X}\),  along the {\etale} map \(f\) which makes \(X\) into a topological correspondence from \(Y\) to \(Z\). Definition~\ref{def:prop-corr} implies that \(X\) is proper if \(b\) is a proper map.

Section~\ref{sec:ex-egpd-case} discusses the case of {\etale} groupoids. In this section, we study {\etale} topological correspondence of {\etale} groupoids. Let \(X\) be a \(G\)\nb-\(H\)-bispace, where \(G\) and \(H\) are {\etale} groupoids. Assume that the right momentum map \(s_X\) is a local homeomorphism which gives \(X\) a continuous family of atomic measure, \(\tau_{s_X}\), along \(s_X\). If \((X, \tau_{s_X})\) is a topological correspondence from \(G\) to \(H\), then we call \(X\) an {\etale} correspondence. Example of such a morphism is the Hilsum-Skandalis morphism,  see~\cite{Hilsum-Skandalis1987Morphismes-K-orientes-deSpaces-deFeuilles-etFuncto-KK}. Proposition~\ref{prop:gen-etale-corr-iso} shows that any (proper) topological correspondence obtained from the above \(G\)\nb-\(H\)-bispace \(X\) by changing the family of measures on \(X\) is isomorphic to \((X,\tau_{s_X})\). Thus the \(\KK\)\nb-class in \(\KK(\Cst(G),\Cst(H))\) determined by such a bispace does not depend on the family of measures the bispace carries.
 We show that \(X\) is  a proper correspondence if the momentum map for the right action is a proper, see~Proposition~\ref{prop:prop-etale-cor-char}.  

What could we not achieve? The proof of Theorem~\ref{thm:prop-corr-main-thm} uses the cutoff function on a proper groupoid which needs that the space of units is Hausdorff. This forces that the bispace involved in a topological correspondence is Hausdorff. Thus, though we write results for locally compact groupoids when it comes to proving Theorem~\ref{thm:prop-corr-main-thm}, the fact that the groupoids are non-Hausdorff does not play a great role. We do not know examples of groupoid actions which is locally strictly locally free but not strictly locally free, and locally free but not locally locally free, see Figure~\ref{fig:LF-actions-rel} and the discussion below it.

Before proceeding to the summary of the article, we would recommend the reader to assume the function \(\RN\) and an adjoining function of a topological correspondence to be the constant function 1, especially in Section~\ref{sec:prop-corr} and the proof of Theorem~\ref{thm:prop-corr-main-thm} during the first reading. This would reduce the complexity in the proofs and ideas. 
Following is the sectionwise summary.

\noindent
\textbf{Section~\ref{sec:prelim}:}
 we fix some important notation in this section. This section has five subsections. The first one, namely, Section~\ref{sec:measure-th}, discusses extensions of functions along local homeomorphisms which are not necessarily surjective and then some measure theoretic preliminaries. The continuous extensions of a function along a local homeomorphism, which need not be surjective, are used to verbalise one of the conditions for a proper correspondence, Definition~\ref{def:prop-corr}{(iii)}.  In the measure theoretic preliminaries, firstly we discuss constructing measures using local data and restriction of measures. Then, in Lemma~\ref{lem:conc-measu-equi}, we discuss various equivalent ways of saying that a measure is concentrated on a subset. This is used to prove the next lemma, Lemma~\ref{lem:loc-concentrated-measure}, which is one of important result for this article. We use Lemma~\ref{lem:loc-concentrated-measure} to interpret a condition in the definition of a proper correspondence in certain cases. The last part of this section discusses absolute continuity of families of measures.

In the next subsection, Section~\ref{sec:act-gpd}, we discuss free, proper and transitive actions.

In Section~\ref{sec:locally-free-actions}, we introduce different notions of locally free actions of groupoids and examples of some of them. In Proposition~\ref{prop:etale-gpd-loc-free-act}, we show that any action of an {\etale} groupoid is strongly locally free.
 
Section~\ref{sec:prop-group-cutoff} is a revision of some well-known results about proper groupoids and cutoff functions.

Let \((H,\beta)\) be a locally compact groupoid with a Haar system. Let \(X\) be a proper \(H\)\nb-space and \(\lambda\) an \(H\)\nb-invariant family of measures on \(X\). In Section~\ref{comp-op}, we describe \(\TComp(X)\), the \(\Cst\)\nb-algebra of compact operators on the \(\Cst(H,\beta)\).
\smallskip

\noindent
\textbf{Section~\ref{sec:prop-corr}:}
This section starts with three examples which, we expect, may prove helpful to understand the definition of a proper correspondence. The examples are followed by the definition of a proper correspondence, Definition~\ref{def:prop-corr}. Then we make few remarks. The rest of the section is the proof our main theorem, namely, Theorem~\ref{thm:prop-corr-main-thm}. 
\smallskip

\noindent
\textbf{Section~\ref{sec:exa}:}
In this section, we give examples of proper correspondences and show that some well-known proper topological correspondences are proper in our sense. Example~\ref{exa:Tu} shows that a proper locally proper generalised morphism defined by Jean--Louis Tu in~\cite{Tu2004NonHausdorff-gpd-proper-actions-and-K} is a proper topological correspondence. Example~\ref{exa:gpd-equi} shows that an equivalence of groupoid, Example~\ref{exa:prop-quiver} shows that a proper quiver defined by Muhly and Tomford in~\cite{Muhly-Tomforde-2005-Topological-quivers} is a proper topological correspondence. This section discusses {\etale} topological correspondence of spaces and {\etale} groupoids. Proposition~\ref{prop:prop-etale-cor-char} shows that for an {\etale} correspondences, one of the groupoids can be reduced to a space.

\section{Preliminaries}
\label{sec:prelim}

\paragraph{Topological conventions}

We assume that reader is familiar with basic theory of locally compact groupoids (\cite{Renault1980Gpd-Cst-Alg}, \cite{Tu2004NonHausdorff-gpd-proper-actions-and-K}), continuous families of measures (\cite{Renault1985Representations-of-crossed-product-of-gpd-Cst-Alg}, \cite{mythesis}*{Section 2.5.2}) and topological correspondences of locally compact groupoids equipped with Haar systems (\cite{Holkar2017Construction-of-Corr}).

Let \(X\) be a topological space. A subset \(A \subseteq X\) is called quasi-compact if every open cover of \(A\) has a finite subcover, compact if it is quasi-compact and Hausdorff. The space \(X\) is called locally compact if every point has a compact neighbourhood. We call a groupoid locally compact if it is a locally compact topological space and its space of units is Hausdorff. For locally compact space \(X\) \(\Contc(X)_0\) denotes the set of functions \(f\) such that \(f\) vanishes outside a compact set \(V\) and \(f|_v\in\Contc(V)\). And \(\Contc(X)\) is defined as the linear span of \(\Contc(X)_0\). The main definition and theorem in this article are stated for locally compact groupoids and Hausdorff spaces, so reader may simply assume that groupoids are also Hausdorff.

For a function \(f\colon X\to Y\) of sets, \(\Img(f)\) denotes the image of \(f\). Assume that \(X\) and \(Y\) are spaces, and \(f\colon X\to Y\) a continuous map. We say \(f\) is open, if \(f(U)\) is open \emph{in} \(Y\) for every open \(U\subseteq X\). We call \(f\) open onto its image, if \(f(U)\) is open \emph{in} \(f(X)\) for very open \(U\subseteq X\).

 The function \(f\) above is called locally one-to-one if for each \(x\in X\) there is a neighbourhood \(U\subseteq X\) of \(x\) with \(f\vert_U\colon U\to f(U)\) is a one-to-one function. The function \(f\) is called a local homeomorphism if \(f\) is surjective and for each \(x\in X\) there is a neighbourhood \(U\subseteq X\) of \(x\) with \(f\vert_U\colon U\to f(U)\) is a homeomorphism. In this case, \(f\) is an open map.  We call \(f\) a local homeomorphism \emph{onto its image} if \(x\in X\) there is a neighbourhood \(U\subseteq X\) of \(x\) with \(f\vert_U\colon U\to f(U)\) is a homeomorphism; this is equivalent to saying that the function \(f'\colon X\to f(X)\) obtained from \(f\) by restricting its codomain is a local homeomorphism.
 
Let \(X\) be a space. If \(\sim\) is an equivalence relation on \(X\), then \([x]\in X/\!\sim\) denotes the equivalence class of \(x\in X\) under \(\sim\), and \(q_X\colon X\to X/\!\sim\) denotes the quotient map \(q_X(X)=[x]\). Let \(Y\) be another space and \(f\colon X\to Y\) a continuous map. If \(f\) preserves \(\sim\), that is, \(f(x)=f(y)\) for all \(x,y\in X\) with \(x\sim y\), then the universal property of quotient induces a map \(X/\!\sim \to Y\), \([x]\to f(x)\). We denote this map by \([f]\).

Let \(X,Y\) and \(Z\) be spaces, and let \(a\colon X\to Z\) and \(b\colon Y\to Z\) be functions. Then we denote the set \(\{(x,y)\in X\times Y: a(x)=b(y)\}\), which is the fibre product of \(X\) and \(Y\) over \(Z\) along \(a\) and \(b\), by \(X\times_{a,Z,b}Y\) most often. Sometimes, when the map \(a\) and \(b\) are clear, we write \(X\times_ZY\). For \(U\subseteq X\) and \(V\subseteq Y\), \(U\times_{a,Z,b}V\) denotes the set \((U\times V)\cap( X\times_ZY)\). Note that \(U\times_{a,Z,b}V\subseteq X\times_{a,Z,b}Y\) is open if and only if \(U\times V\subseteq X\times_{a,Z,b}Y\) is open, which is, in turn, open if and only if \(U\subseteq X\) and \(V\subseteq Y\) are open. 

Let \(X,Y,Z,a\) and \(b\) be as in the last paragraph. Let \(f\) be a function on \(X\) and \(g\) on \(Y\). Then \(f\otimes g\) is the function on \(X\times Y\) \((x,y)\mapsto f(x)g(x)\). The restriction of \(f\otimes g\) to the fibre product \(X\times_{a,Z,b}Y\) is written as \(f\otimes_Zg\).

Let \(X,Y\) and \(Z\) be spaces, and \(f\colon X\to Z\) and \(g\colon Y\to Z\) continuous maps. Then a basic open set for the product topology on \(X\times Y\) is a product \(U\times V\) where \(U\subseteq X\) and \(V\subseteq Y\) are open. Hence the sets of type \(U\times_{f,Z,g} V= (U\times V)\cap (X\times_{f,Z,g} Y)\), where \(U\subseteq X\) and \(V\subseteq Y\) are open, are basic open sets for the subspace topology of the fibre product \(X\times_{f,Z,g} Y\subseteq X\times Y\).

For a groupoid \(G\), \(r_G\) and \(s_G\) denote its range and source maps, respectively. For a left (right) \(G\)\nb-space, we always assume that the corresponding momentum map is \(r_X\) (respectively, \(s_X\)), with an exception of Example~\ref{exa:prop-quiver}. We denote the fibre product \(G\times_{s_G,\base,r_X}X\) (respectively, \(X\times_{s_X,\base, r_G}G\)) by \(G\times_{\base}X\) (respectively, \(X\times_{\base}G\)). We say \(\gamma\) and \(x\) are composable or \((\gamma, x)\) is a composable pair if \((\gamma,x)\in G\times_{\base}X\); similarly for the right action.

Let \(X\) be a left \(G\)\nb-space and right \(H\)\nb-space for groupoids \(G\) and \(H\). We call \(X\) a \(G\)\nb-\(H\)-bispace if the actions of \(G\) and \(H\) commute in the usual sense, that is, for every composable pairs \((\gamma, x)\in G\times_{\base}X\) and \((x,\eta)\in X\times_{\base[H]}H\) we have that
\begin{enumerate*}[(i)]
\item \((\gamma x, \eta)\in X\times_{\base[H]}H\)
  and \((\gamma, x\eta )\in G\times_{\base}X\),
\item \((\gamma x)\eta=\gamma(x\eta)\).
\end{enumerate*}

Let \((H,\beta)\) be a locally compact groupoid with a Haar system and \(X\) a right \(H\)\nb-space. Then there is continuous family of measures with full support \(\beta_X=\{\beta^{[x]}_X\}_{[x]\in X/H}\) along the quotient map \(X\to X/H\) which is given by
\begin{equation}
  \label{eq:measures-along-the-quotient}
\int_{X} f\;\dd\beta^{[x]}_X \defeq \int_{G} f(x\eta)\;\dd\beta^{s_X(x)}(\eta)
\end{equation}
 for \([x]\in X/H\) and \(f\in \Contc(X)\).
 See~\cite{mythesis}*{Proposition 1.3.21} for the proof.

Let \(\mathcal{A}\) be a category and let \(\mathcal{A}^0\) denote its class of objects. By \(x\in\!\in\mathcal{A}\) we mean that \(x\) is an object in \(\mathcal{A}\).

The symbols \(\R^*\) and \(\C^*\) denote, respectively, the set of nonzero real and complex numbers. And \(\R^-\) and \(\R^+\) denote is the set of negative and positive real numbers, respectively.
\subsection{Measures concentrated on sets}
\label{sec:measure-th}

Our reference for measure theory is Bourbaki~\cite{Bourbaki2004Integration-I-EN}. For a topological space \(X\), let \(\Bor(X)\) denote the set of Borel functions on \(X\). Given \(Y\subseteq X\), \(\chi_Y\) denotes the characteristic function of the set \(Y\). In this article, all the measures are assumed to be positive, Radon and \(\sigma\)\nb-finite. Now we introduce some notation.
\begin{notation}[Notation and remarks]
\label{not:1}
 \begin{enumerate}[leftmargin=*]
 \item Let \(X\) be a space and \(V,U\subseteq X\) with \(U\subseteq V\). Let \(f\) be a function on \(U\). Then we write \(0^V(f)\) for the function on \(V\) which equals \(f\) on \(U\) and zero on \(V-U\). 

Let \(X\) be locally compact space and \(U\) an open Hausdorff subset of it. Then define the set
\[
 \Contc(X, U)=\{f\in \Contc(X): \supp(f)\subseteq U \}.
\]
When \(X\) is Hausdorff one may identify \(\Contc(X,U)\) with \(\Contc(U)\) as follows: if \(g\) is in \(\Contc(X,U)\), then the restriction of \(g\) to \(U\) lies in \(\Contc(U)\). On the other hand, if \(f\) is in \(\Contc(U)\), then \(0^X(f)\) is in \(\Contc(X,U)\). These two process are inverses of each other. We shall frequently use this identification.

\item Let \(X\) and \(Y\) be locally compact Hausdorff spaces, and \(\phi\colon X\to Y\) a map which is open onto its image. For an open set \(U\subseteq X\), define
 \[
\ExtS_\phi(U)=\{U'\subseteq Y: U'\text{ is open in }Y\text{ and } U'\cap \phi(X)=\phi(U)\}
\]
(here an everywhere else, \(\Ext\) stands for `extension'). Note that for \(U\neq\emptyset\), \(\ExtS_\phi(U)\) is not empty: since \(\phi(U)\subseteq \phi(X)\) is open, there is an open set \(U'\subseteq Y\) with \(\phi(U)=U'\cap\phi(X)\). We call an element of \(\ExtS_\phi(U)\) an extension of \(U\) via \(\phi\), or simply an extension of \(U\).

\item  A set is called cocomapct if it has compact closure. Let \(X\) and \(Y\) be locally compact Hausdorff spaces, and \(\phi\colon X\to Y\) a local homeomorphism onto its image. Then \(\phi\) is an open map onto its image. Let
  \begin{align*}
\phi^{\Ho}&\defeq\{U\subseteq X : U \text{ is open and } \phi|_{U} \text{ is a homeomorphism}\}\\
\phi^{\HoC}&\defeq\{U\subseteq X : U  \text{ is cocompact, }U\in\phi^{\Ho} \text{ and } \phi|_{\overline{U}} \text{ is a homeomorphism}\}\\
  \end{align*}
Then \(\phi^{\Ho}\neq \emptyset\) is an open cover of \(X\) for \(X\neq \emptyset\). Given \(U\in\phi^{\Ho}\), we write \(\phi|_U\inverse\) for \((\phi|_U)\inverse\). Note that \(\phi|_U\inverse\colon \phi(U)\to U\) is a homeomorphism. The map \(\phi|_U\inverse\) induces an obvious isomorphism \(\Contc(U)\to \Contc(\phi(U))\) which sends a function \(f\in \Contc(U)\) to the function \(f\circ \phi|_{\phi(U)}\inverse\) in \(\Contc(\phi(U))\).

 Since \(X\) is locally compact, for any \(U\in\phi^{\Ho}\) and \(x\in U\), there is an open cocompact neighbourhood \(V\) of \(x\) with \(\overline{V}\subseteq U\). Hence \(\phi^{\HoC}\neq \emptyset\). Moreover, \(\phi^{\Ho}\) and \(\phi^{\HoC}\) are open covers of \(X\) which are closed under intersections of sets.
\item Let \(X,Y\) and \(\phi\) be as in (3) above. Let \(U\in \phi^{\Ho}\). For \(f\in \Contc(X;U)\) define the set
  \[
    \Ext_\phi(f) =\{f'\in \Contc(Y,U'): U'\in \ExtS_\phi(U)\text{ and } f'= f\circ\phi|_{U}\inverse \text{ on } \phi(U)\}.
\]
Since \(\phi|_U\) is a homeomorphism onto its image, we may write:
\[
    \Ext_\phi(f) =\{f'\in \Contc(Y,U'): U'\in \ExtS_\phi(U)\text{ and } f'\circ \phi|_U= f \text{ on } U\}.
\]
Lemma~\ref{lem:loc-homeo-fn-ext-restr} says that for \(U\in \phi^{\HoC}\) and \(f\in \Contc(X,U)\), \(\Ext_\phi(f)\) is nonempty.
We call an element of \(\Ext_\phi(f)\) \emph{an extension of \(f\) to \(Y\) via \(\phi\)} or, sometimes, simply an extension. Clearly, any two extensions of \(f\) agree on \(\phi(U)\) and equal \(f\circ \phi|_U\inverse\).
Since the intersection of any two sets in \(\ExtS_\phi\) with \(\phi(X)\) is \(\phi(U)\), we have 
\(f'|_{\phi(X)}=f''|_{\phi(X)}=0^{\phi(X)}(f\circ \phi|_U\inverse)\) for any two extensions \(f',f''\in \Ext_\phi(f)\).

For \(U\in \phi^{\Ho}\) , let \(U'\in \ExtS_\phi(U)\). For \(g\in \Contc(Y, U')\) we define the \emph{function} \(\Res^\phi(g)\) on \(X\) to be \(0^Y(g\circ \phi|_U)\), that is,
\[
\Res^\phi(g)= \begin{cases}  g\circ \phi|_U  &\mbox{on } U, \\
0 & \mbox{on } X-U \end{cases} 
\]
The function \(\Res^\phi(g)\) is continuous on \(U\). Lemma~\ref{lem:restr-is-Cc} says that \(\Res^\phi(g)\in \Contc(X;U)\) when \(U\in \phi^{\HoC}\). Here \(\Res\) stands for `restriction'.
\item Let \(X,Y\) and \(\phi\) be as in (3) above. Analogous to \(\Contc(X,U)\), for \(U\subseteq X\), we define \(\Bor(X,U)\) as the set of Borel measurable functions on \(X\) whose support lies in \(U\). For \(U\in \phi^{\Ho}\) and \(f\in \Bor(X,U)\), we define \(\ExtB_\phi(f)\) analogous to \(\Ext_\phi(f)\) in (4) above but with \(\Contc(Y,U')\) replaced by \(\Bor(Y,U')\). Finally, observe that \(\Res^\phi(g)\) makes sense for a Borel function \(g\in \Bor(Y,U')\) where \(U'\in\ExtS_\phi(U)\).
 \end{enumerate}
\end{notation}
\smallskip

Now we start preparing for Lemma~\ref{lem:restr-is-Cc}. Let \(X\) be a space, and \(Y\) and \(A\) its subspaces with \(A\subseteq Y\). The set \(A\) is (quasi-)compact in \(Y\) if and only if it is so in \(X\). To prove the \emph{if} part, let \(\{{U_j}'\}\) be an open cover of \(A\) in \(Y\). For each \({U_j}'\), there is an open set \({U_j}'\) of \(X\) with \({U_j}'={U_j}\cap Y\). Since \(\{{U_j}'\}\) cover \(A\), so do \(\{{U_j}\}\). Let \({{U_j}_1},\dots,{{U_j}_n}\) be a finite cover of \(A\) in \(X\), then \({{U_j}_1}',\dots,{{U_j}_1}'\) is a finite subcover of \(A\) in \(Y\).

Conversely, let \(\{{U_j}\}\) be an open cover of \(A\) in \(X\). Then \({U_j}'\defeq {U_j}\cap Y\) are open sets in \(Y\) which cover \(A\). Since \(A\) is compact in \(Y\), we may choose a finite subcover, say \({{U_j}_1}'\dots,{{U_j}_n}'\). Then \({{U_j}_1},\dots,{{U_j}_n}\) is the finite subcover of \(A\) in \(X\).
 We need \emph{if} part observation in the final argument of Lemma~\ref{lem:restr-is-Cc}. Even later this observation shall be used.

Let \(X\) be a space, \(Y\) and \(W\) its subspaces with \(Y\supseteq W\). In the next lemma, \(\overline{W}\) and \(\overline{W}^Y\) denote the closures of \(W\) in \(X\) and \(Y\), respectively. In general, \(\overline{W}\) denotes the closure of \(W\) in the biggest space in the discussion that contains \(W\). It is a basic result in point-set topology that \(\overline{W}^Y= \overline{W}\cap Y\).
\begin{lemma}
  \label{lem:restr-is-Cc}
Let \(X\) and \(Y\) be locally compact Hausdorff spaces, and \(\phi\colon X\to Y\) be a local homeomorphism onto its image. Let \(U\in \phi^{\HoC}\) and \(V\in \ExtS_\phi(U)\). Then for \(g\in \Contc(Y,V)\), \(\Res^\phi(g)\in \Contc(X;U)\). A similar claim holds for \(g\in \Bor(Y,V)\), that is, for \(g\in \Bor(Y,V)\), \(\Res^\phi(g)\in \Bor(X;U)\).
\end{lemma}
\begin{proof}
The Borel part of the lemma follows easily from the definition of \(\Res^\phi(g)\). We deal with the continuous case. Assume that \(g|_{\phi(U)}\in \Contc(\phi(U))\). Then \(\Res^\phi(g)\in \Contc(X;U)\). Hence we shall show that \(g|_{\phi(U)}\in \Contc(\phi(U))\). We write \(Z\) for \(\phi(X)\) in the discussion that follows.

We need the following observation: for any subset \(A\) of \(V\), \(A\cap \phi(\overline{U})=A\cap \phi(U)\).
The proof of the observation is as follows: it is obvious that \(A\cap \phi(\overline{U})\supseteq A\cap \phi(U)\). One the other hand, \(A\cap \phi(\overline{U})=(A\cap V)\cap(\phi(\overline{U})\cap Z)\subseteq A\cap (V\cap Z)=A\cap \phi(U)\).

Now the fact that \(g\inverse(\C^*)\subseteq \supp(g)\subseteq V\) and the observation in the last paragraph together give us that \(g\inverse(\C^*)\cap \phi(U)=g\inverse(\C^*)\cap \phi(\overline{U})\). But
\begin{align*}
  g\inverse(\C^*)\cap \phi(U)&= \{y\in \phi(U): g(y)\neq 0 \}=g|_{\phi(U)}\inverse(\C^*) \text{ and}\\
g\inverse(\C^*)\cap \phi(\overline{U})&= \{y\in \phi(\overline{U}): g(y)\neq 0 \}=g|_{\phi(\overline{U})}\inverse(\C^*).
\end{align*}
Thus \(g|_{\phi(U)}\inverse(\C^*)=g|_{\phi(\overline{U})}\inverse(\C^*)\); denote either of these sets by \(B\). Then \(\supp(g|_{\phi(U)})=\overline{B}^{\phi(U)}\) and \(\supp(g|_{\phi(\overline{U})})=\overline{B}^{\phi(\overline{U})}\). Furthermore, 
\begin{equation}
\label{eq:res-is-cc-1}
\supp(g|_{\phi(U)})=\overline{B}^{\phi(\overline{U})} \cap \phi(U).
\end{equation}
 Now we claim that \(\overline{B}^{\phi(\overline{U})} \subseteq V\).  This is because \(g|_{\phi(\overline{U})}\inverse(\C^*)\subseteq g\inverse(\C^*)\) and hence
\[\overline{B}^{\phi(\overline{U})}=\overline{g|_{\phi(\overline{U})}\inverse(\C^*)}\cap \phi(\overline{U})\subseteq \overline{g|_{\phi(\overline{U})}\inverse(\C^*)}\subseteq \overline{g\inverse(\C^*)}=\supp(g)\subseteq V.
\]
 Now Equation~\eqref{eq:res-is-cc-1} and the observation made at the beginning of the proof together yield that
\[
\supp(g|_{\phi(U)})=\overline{B}^{\phi(\overline{U})} \cap \phi(U)=\overline{B}^{\phi(\overline{U})} \cap \phi(\overline{U})=\overline{B}^{\phi(\overline{U})}=\supp(g|_{\phi(\overline{U})}).
\]
Note that since \(\overline{B}^{\phi(\overline{U})}\) is a closed subset of the compact set \(\phi(\overline{U})\), \(\overline{B}^{\phi(\overline{U})}\) is compact in \(\phi(\overline{U})\). This along with the observation prior to the statement of this lemma implies that \(\overline{B}^{\phi(\overline{U})}=\supp(g|_{\phi(U)})\) is compact in \(Y\) as well as \(\phi(U)\).
\end{proof}

\begin{lemma}
  \label{lem:loc-homeo-fn-ext-restr}
Let  \(X\) and \(Y\) be locally compact spaces with \(Y\) Hausdorff. Let \(\phi\colon X\to Y\) be a local homeomorphism onto its image. Let \(U\in\phi^{\HoC}\) and \(U'\in \ExtS_\phi(U)\). Then
\begin{enumerate}[i)]
\item  for \(f\in\Contc(X;U)\), \(\Ext_\phi(f)\) is nonempty;
\item  for \(f\in \Contc(X;U)\) and \(f'\in \Ext_\phi(f)\), \(\Res^\phi(f')=f\);
\item  for \(g\in \Contc(Y,U')\), \(g\in \Ext(\Res^\phi(g))\).
\end{enumerate}
Similar claims for hold the measurable case.
\end{lemma}
\begin{proof}
Consider the measurable case first. 
i): For \(f\in \Bor(X;U)\), \(0^Y(f\circ\phi|_U\inverse)\) lies in \(\ExtB_\phi(f)\). (ii) and (iii) can be proved on the similar lines as the continuous case. Now we prove the claims for the continuous case.

\noindent i): We know that \(\phi(\overline{U})\subseteq Y\) is compact, therefore,
\[
f\circ\phi|_{\phi(\overline{U})}\inverse\colon \phi(\overline{U})\to \C
\]
is a well-defined continuous function. We extend \(f\circ\phi|_{\phi(\overline{U})}\inverse\) to a function \(\tilde f\in\Contc(Y)\) with \(f=\tilde f\) on \(\phi(\overline{U})\) using Tietze's extension theorem.

Let \(V\in \ExtS_\phi(U)\). Then \(\phi(\supp(f))\subseteq V\) is compact, since \(\phi(\supp(f))\) is compact in the subspace \(\phi(U)\) of \(V\). Let \(w\in \Contc(Y)\) be a function which is \(1\) on the compact set \(\phi(\supp(f))\) and has support in \(V\). Then \(f'\defeq \tilde f w\in\Contc(Y)\) is an element of \(\Ext_\phi(f)\): it is plain that \(f'|_{\phi(U)}=\tilde f|_{\phi(U)}=f\circ {\phi|}_{U}\inverse\). Moreover,
\[
\left(\supp(f')\cap \phi(X)\right) \subseteq \left((\supp(\tilde f)\cap\supp(w))\cap \phi(X)\right)\subseteq\,\supp(w)\cap \phi(X)\subseteq \phi(U).
\]
Thus \(f'\in \Ext_\phi(f)\).

\noindent ii): If \(f'\in \Ext_\phi(f)\), then by definition \(f'\circ \phi|_U=f\) on \(\phi(U)\); hence \(\Res^\phi(f') = f'\circ \phi|_U=f\) on \(\phi(U)\). Since \(f\in \Contc(X;U)\), \(f=0\)  on \(X-U\), and so is \(\Res^\phi(f')\).

\noindent iii): Lemma~\ref{eq:res-is-cc-1} says that \(\Res^\phi(g)\in \Contc(X;U)\), hence we may define \(\Ext_\phi(\Res^\phi(g))\). The rest follows directly from the definitions of \(\Res^\phi(g)\) and \(\Ext_\phi(f)\).
\end{proof}

\paragraph{\textbf{Restriction of a measure, and construction of a measures using local data:}}

 Let \(X\) be a locally compact, Hausdorff space and \(U\) an open subset of \(X\). Identify \(\Contc(X,U)\) with \(\Contc(U)\) in the obvious way, see~\ref{not:1}(1). Recall from elementary measure theory that the measure \(\lambda\colon \Contc(X)\to\R\) on \(X\), when restricted to the subspace \(\Contc(X,U)\iso\Contc(U)\) gives a measure on \(U\) which is called the restriction of the \(\lambda\) to \(U\). We denote this restriction by \(\lambda|_U\).

In general, the measure \(\lambda\) can be restricted  not only to an open set but also to any Borel set. Let \((X,\lambda)\) be a Borel measure space and \(Y\) a Borel subset of \(X\). Then any Borel subset \(U\subseteq Y\) is of type \(U=U'\cap Y\) for a Borel \(U'\subseteq X\). For \(Y\subseteq X\) as above, define \(\lambda|_Y\), the restriction of \(\lambda\) to \(Y\), by \(\lambda|_Y(U)=\lambda(U'\cap Y)\) where \(U\subseteq Y\) and \(U'\subseteq X\) are  Borel sets with \(U=U'\cap Y\). We shall often use the following result from elementary measure theory: let \(f\) be an integrable function on \(X\) then
\[
\int_X f|_Y\,\dd\lambda\defeq\int_Y f\,\dd\lambda=\int_Y f\,\dd\lambda|_Y.
\]

Let \(X\) and \(\lambda\) be as in the last paragraph, and let \(U\) and \(V\) be Borel subsets of \(X\) with \(V\subseteq U\). It  is easy to check that the restricted measures satisfy the equality \(\lambda|_V=(\lambda|_U)|_V\). We will not need such a generalised notion of restricted measures, but we shall need the notion of restriction of a measure to a locally compact space, see~\cite{Bourbaki2004Integration-I-EN}*{Nr.\,7, {\S}5, IV.\,74}. 

 The following proposition gives a criterion to extend measures defined on elements of an open cover of \(X\) to whole of the space.
\begin{proposition}[Proposition~\cite{Bourbaki2004Integration-I-EN}*{Chpater III, \S 2, 1, Proposition 1}]
  \label{prop:bourbaki-pasting-measures}
Let \(\{Y_\alpha\}_{\alpha\in A}\) be an open cover of \(X\), \(X\) a locally compact Hausdorff space, and suppose given, on each subspace \(Y_\alpha\), a measure \(\mu_\alpha\), in such a way that for every pair \((\alpha,\beta)\), the restriction of \(\mu_\alpha\) and \(\mu_\beta\) to \(Y_\alpha\cap Y_\beta\) are identical. Under these conditions, there is one and only one measure \(\mu\) on \(X\) whose restriction to \(Y_\alpha\) is equal to \(\mu_\alpha\) for every index \(\alpha\).
\end{proposition}

\begin{lemma}
  \label{lem:pulling-and-pasting-measures-1}
Let \(X\) and \(Y\) be locally compact, Hausdorff spaces and let \(\lambda\) be a measure on \(Y\). If \(\phi\colon X\to Y\) be a local homeomorphism (onto) \(Y\), then \(\lambda\) induces a unique measure \(\phi^*(\lambda)\) on \(X\) such that for every \(f\in \Contc(X,U)\), where \(U\in\phi^{\Ho}\), we have
\begin{equation}
  \label{eq:pullback-measure-1}
\int_X f\,\phi^*(\lambda)=\int_Y 0^Y(f\circ \phi|_U\inverse) \;\dd\lambda. 
\end{equation}
\end{lemma}
\begin{proof}
We know that \(\phi^{\Ho}\) is an open cover of \(X\). For each \(U\in\phi^{\Ho}\) define a measure \(\phi^*(\lambda)_U\) on \(U\) by 
\[
\phi^*(\lambda)_U(f)=\int_{\phi(U)} f\circ \phi|_U\inverse \,\dd\lambda.
\]
 If \(U, V\in \phi^{\Ho}\), then \(U\cap V\in \phi^{\Ho}\). Furthermore, for \(f\in \Contc(U\cap V)\),
\[
\phi^*(\lambda)_{U}(f)=\phi^*(\lambda)_{U\cap V}(f)=\phi^*(\lambda)_{V}(f).
\]
Now Proposition~\ref{prop:bourbaki-pasting-measures} says that there is a unique measure on \(X\), which we denote by \(\phi^*(\lambda)\), which has the property that restriction of \(\phi^*(\lambda)\) to each \(U\) in \(\phi^{\Ho}\) equals \(\phi^*(\lambda)_U\). The result, now, follows from the observation that \(\lambda_{\phi(U)}(f\circ \phi|_U\inverse)=\lambda(0^Y(f\circ \phi|_{U}\inverse))\) for \(f\in\Contc(X,U)\) where \(U\in\phi^{\Ho}\).
\end{proof}
One may observe that Lemma~\ref{lem:pulling-and-pasting-measures-1} can be restated when the function \(f\) in Equation~\eqref{eq:pullback-measure-1} is an integrable function in \(\Bor(X,U)\) for \(U\in \phi^{\Ho}\). The proof is similar to the one above.

Let \(Y\) be a locally compact, Hausdorff space and \(\lambda\) a measure on it. Then~\cite{Bourbaki1966Topologoy-Part-1}*{Proposition 12, {\S}9.\,7, Chapter I} says that any locally compact Hausdorff subspace \(Z\) of \(Y\) is locally closed. Moreover,~\cite{Bourbaki1966Topologoy-Part-1}*{Proposition 5, {\S}3.\,3, Chapter I} implies that \(Z\) is an intersection of a closed and open subsets of \(Y\). Hence \(Z\) is a Borel subspace of \(Y\) and the restriction of \(\lambda\) to \(Z\) is defined.
\begin{corollary}
  \label{cor:pulling-and-pasting-measures-2}
Let \(X\) and \(Y\) be locally compact Hausdorff spaces, and \(\lambda\) a measure on \(Y\). If \(\phi\colon X\to Y\) is a local homeomorphism onto its image, then \(\lambda\) induces a unique measure \(\phi^*(\lambda)\) on \(X\) such that for every \(f\in \Contc(X,U)\) where \(U\in\phi^{\HoC}\) we have
\begin{equation}
  \label{eq:pullback-measure-2}
\int_X f\,\dd\phi^*(\lambda)=\int_{_{\phi(U)}} f'\,\dd\lambda=\int_{_{\phi(X)}} f'\,\dd\lambda
\end{equation}
for any extension \(f'\in\Ext_\phi(f)\).
 A similar statement holds for \(f\in \Bor(X,U)\) and \(f'\in\ExtB_\phi(f)\).
\end{corollary}
\begin{proof}
Let \(\phi'\colon X\to \phi(X)\) be the map obtained from \(\phi\) by restricting the codomain; here \(\phi(X)\subseteq Y\) is equipped with subspace topology. Then \(\phi'\) is a local homeomorphism and hence open. Since the image of a locally compact space under an open map is locally compact, \(\phi(X)\) is locally compact when equipped with the subspace topology from \(Y\). Thus \(\phi(X)\) is a locally compact subspace of \(Y\). Now we bestow \(\phi(X)\) with the restricted measure \(\lambda|_{\phi(X)}\) and apply Lemma~\ref{lem:pulling-and-pasting-measures-1} to \(\phi':X\to \phi(X)\). This gives us the measure \({\phi'}^*(\lambda|_{\phi(X)})\) which we denote by \(\phi^*(\lambda)\). Equation~\eqref{eq:pullback-measure-1} tells us that for \(f\in \Contc(X;U)\), where \(U\in\phi^{\Ho}\),
\begin{equation}
\label{eq:ext-1}
\int_X f\,\dd\phi^*(\lambda)=\int_{\phi(X)}0^{\phi(X)}(f\circ \phi|_{\phi(U)}\inverse)\;\,\dd\lambda.
\end{equation}
Since \(0^{\phi(X)}(f\circ \phi|_{\phi(U)}\inverse)\in \Contc(\phi(X),\phi(U))\), we can write
\[
\int_X f\,\dd\phi^*(\lambda)=\int_{\phi(X)}0^{\phi(X)}(f\circ \phi|_{\phi(U)}\inverse)\;\dd\lambda=\int_{\phi(U)}0^{\phi(X)}(f\circ \phi|_{\phi(U)}\inverse)\;\dd\lambda.
\]
Let \(f'\in\Ext_\phi(f)\). Then Notation~\ref{not:1}(4) tells us that \(f'|_{\phi(X)}=0^{\phi(X)}(f\circ\phi|_U\inverse)\) and \(f'|_{\phi(X)}=(f'|_{\phi(U)})_{\phi(X)}^0\). From this discussion and Equation~\eqref{eq:ext-1} we may write that
\begin{equation*}
 \int_X f\,\dd\phi^*(\lambda)
=\int_{\phi(U)}(f\circ\phi|_U\inverse)^0_{\phi(X)}\,\dd\lambda
=\int_Y f'|_{\phi(X)}\,\dd\lambda
=\int_Y f'|_{\phi(U)}\,\dd\lambda.
\end{equation*}

The claim of the lemma follows by observing that the last two terms in the above equation are \(\int_{\phi(X)}
  f'\,\dd\lambda\) and \(\int_{\phi(U)} f'\,\dd\lambda\) , respectively.

The proof for the measurable case can be written along the same lines as above.
\end{proof}

\begin{lemma}
  \label{lem:conc-measu-equi}
 Let \(Y\) be a locally compact, Hausdorff space, \(\lambda\) a measure on it and \(Z\) a locally compact subspace of \(Y\). Then the following are equivalent:
 \begin{enumerate}[i)]
 \item For every nonnegative \(f\in \Contc(Y)\), \(\int_Y f\,\dd\lambda = \int_Z f\,\dd\lambda\).
 \item For every \(f\in \Contc(Y)\), \(\int_Y f\,\dd\lambda = \int_Z f\,\dd\lambda\).
 \item For every \(f\in \Lone(Y)\), \(\int_Y f\,\dd\lambda = \int_Z f\,\dd\lambda\).
 \item For every nonnegative \(f\in \Lone(Y)\), \(\int_Y f\,\dd\lambda = \int_Z f\,\dd\lambda\).
 \item Every \(y\in Y\) has a neighbourhood \(V\subseteq Y\) such that \(\lambda(V\cap (Y-Z))=0\).
 \end{enumerate}
\end{lemma}
\begin{proof}
\noindent (i)\(\implies\)(ii): Let \(f\in \Contc(Y)\). Put \(f
^+ = \max\{f,0\}\) and \(f^-=-\min\{f,0\}\) where \(0\) is the constant function \(0\) on \(Y\). Then \(f^+\) and \(f^-\) are continuous, and they are compactly supported since \(\supp(f^+)\) and \(\supp(f^-)\) are contained in  \(\supp(f)\). Furthermore, \(f= f^+ - f^-\) . Hence
\[
\int_Y f\,\dd\lambda = \int_Y f^+\,\dd\lambda - \int_Y f^-\,\dd\lambda = \int_Z f^+\,\dd\lambda - \int_Z f^-\,\dd\lambda =\int_Z f\,\dd\lambda.
\]
\smallskip

\noindent (ii)\(\implies\)(iii): Let \(f\in \Lone(Y)\). Then \(f-f|_Z = f-f\chi_Z\) is also in \(\Lone(Y)\) and we may choose \(g\in \Contc(Y)\) such that \(\int_Y \abs{(f-f|_Z)-g}\,\dd\lambda < \epsilon/2\) for given \(\epsilon >0\). Now
\begin{multline}
 \label{eq:meas-conc-qui-1}
  \int_Y \abs{g}\,\dd\lambda=\int_Z \abs{g}\,\dd\lambda\leq \int_Z \abs{g-(f-f|_Z)}\,\dd\lambda + \int_Z (f-f|_Z)\,\dd\lambda\\
=\int_Z \abs{g-(f-f|_Z)}\,\dd\lambda\leq \int_Y \abs{g-(f-f|_Z)}\,\dd\lambda< \epsilon/2.
\end{multline}
In above computation, the first equality is due to the hypothesis, the second equality is due to the fact that \(\int_Z (f-f|_Z)\,\dd\lambda=\int_Y \chi_Z (f-f \chi_Z)\,\dd\lambda=0\) and the second last inequality is because of the fact that \(\abs{h\chi_Z}\leq \abs{h}\) for any integrable function \(h\in \Lone(Y)\). 
Now
\[
  0\leq \abs{\int_Y f\,\dd\lambda-\int_Y f|_Z\,\dd\lambda}\leq \int_Y(f-f|_Z)\,\dd\lambda\leq \int_Y\abs{(f-f|_Z)-g}\,\dd\lambda+ \int_Y \abs{g}\,\dd\lambda< \epsilon.
\]
The last inequality above is due to the choice of \(g\) and Equation~\eqref{eq:meas-conc-qui-1}. Since above equation holds for any \(\epsilon>0\), we have \(\int_Y f\,\dd\lambda=\int_Y f|_Z\,\dd\lambda\). 

\noindent (iii)\(\implies\)(iv): Clear.

\noindent (iv)\(\implies\)(v): For \(y\in Y\), let \(V\subseteq Y\) be a relatively compact neighbourhood. Then \(\lambda(\chi_V)=\lambda(V)\leq \lambda(\overline{V})< \infty\); thus \(\chi_V\in \Lone(Y)\) is a nonnegative function. Using the hypothesis we compute
\[
\lambda(V\cap (Y-Z))= \int_Y \chi_V\cdot\; \chi_{(Y-Z)}\;\dd\lambda=\int_Z \chi_V\cdot\; \chi_{(Y-Z)}\;\dd\lambda=\int_Y \chi_V\cdot\; \chi_Z\cdot\;\chi_{(Y-Z)}\;\dd\lambda=0, \]
as \(\chi_Z\cdot\;\chi_{(Y-Z)}=0\).

\noindent (v)\(\implies\)(i): Let \(y\in Y\) and \(V\)  a neighbourhood of \(y\) in \(Y\) such that \(\lambda(V\cap (Y-Z))=0\). Let \(g\in \Contc(Y,V)\) be a nonnegative function then
 \[\int_{Y-Z} g\,\dd\lambda=\int_Y g\cdot\, \chi_V\cdot\;\chi_{Y-Z}\;\dd\lambda\leq \norm{g}_\infty \int_Y \chi_V\cdot\;\chi_{Y-Z}\;\dd\lambda =\norm{g}_\infty\lambda(V\cap (Y-Z))=0.
\]
Now the equality \(\int_Y g\,\dd\lambda = \int_Z g\,\dd\lambda + \int_{Y-Z} g\,\dd\lambda\) gives us that \(\int_Y g\,\dd\lambda = \int_Z g\,\dd\lambda\).

Let \(f\) be a nonnegative function in \(\Contc(Y)\). For every \(y\in \supp(f)\), let \(V_y\) be a neighbourhood of \(y\) such that \(\lambda(V_y\cap (Y-Z))=0\). Because the support of \(f\) is compact and \(\{V_y\}_{y\in \supp(f)}\) cover \(\supp(f)\), we may choose finitely many neighbourhoods \(V_{y_1},\dots,V_{y_n}\) which cover \(\supp(f)\). Let \(u_1,\dots,u_n\) be a partition of unity on \(\supp(f)\) subordinated to this finite cover. Then for every \(i\in \{1,\dots,n\}\), \(f u_i\) is a nonnegative function in \(\Contc(Y,V_{x_i})\). Since \(f=\sum_{i=1}^n fu_i\) using the result proved in the previous paragraph we see that
\[
\int_Y f\,\dd\lambda= \int_Y\sum_{i=1}^n fu_i\,\dd\lambda= \sum_{i=1}^n \int_Y fu_i\,\dd\lambda =\sum_{i=1}^n \int_Z fu_i\,\dd\lambda=\int_Z f\,\dd\lambda. 
\]
\end{proof}
Let \(Y\) be a locally compact, Hausdorff space and \(\lambda\) a measure on it. Let \(Z\) be locally compact subspace of \(Y\). We say \(\lambda\) is \emph{concentrated} on \(Y\), if any one of the equivalent conditions in Lemma~\ref{lem:conc-measu-equi} holds. Condtion~{(v)} in Lemma~\ref{lem:conc-measu-equi} is~\cite{Bourbaki2004Integration-I-EN}*{Definition 4, Nr.\,7, {\S}5, V.\,55}.
\begin{lemma}
  \label{lem:loc-concentrated-measure}
Let \(Y,Z\) and \(\lambda\) be as in Lemma~\ref{lem:conc-measu-equi} above. Let \(\{U_i\}\) be an open cover of \(Y\). Then, for each \(i\), \(\lambda\) is concentrated on \(Z\) if and only if the restricted measure \(\lambda|_{U_i}\) is concentrated on \(U_i\cap Z\) for each \(i\).
\end{lemma}
\begin{proof}
  If \(\lambda\) is concentrated on \(Z\), then for every \(U_i\) and every \(f\in \Contc(Y,U_i)\iso \Contc(U_i)\) we have
  \begin{multline*}
\int_{U_i} f\,\dd\lambda= \int_Y f\,\chi_{U_i}\;\dd\lambda= \int_Y f\,\dd\lambda\\
=\int_Z f\,\dd\lambda= \int_Z f\,\chi_{U_i}\,\dd\lambda=\int_Y f\,\chi_{Z\cap U_i}\;\dd\lambda=\int_{Z\cap U_i} f\,\;\dd\lambda,
  \end{multline*}
which proves that \(\lambda|_{U_i}\) is concentrated on \(Z\cap U_i\).

Conversely, for given \(f\in \Contc(Y)\), cover its support by finitely many elements in \(\{U_i\}\), say \(U_{i_1},\dots,U_{i_n}\) where \(n\in \N\). Let \(u_1,\dots,u_n\) be the partition of unity on \(\supp(g)\) subordinated to the open cover \(\{U_{i_1},\dots,U_{i_n}\}\). Then an argument as in the proof of the part (v)\(\implies\)(i) in Lemma~\ref{lem:conc-measu-equi} shows that \(\int_Y f\,\dd\lambda=\int_Z f\,\dd\lambda\).
\end{proof}

\begin{notation}
  \label{not:cov-ext}
  Let \(X\), \(Y\) be spaces and \(\phi\colon X\to Y\) a continuous map which is also a local homeomorphism onto its image. Let \(\mathcal{A}\) be an open cover of \(X\). An extension of \(\mathcal{A}\) via \(\phi\) is a collection of open sets \(\mathcal{A}'\) of \(Y\) with the following properties
  \begin{enumerate}[(i)]
  \item if \(U'\in\Cov'\), then either there is
    \(U\in\mathcal{A}\) such that \(\phi(U)=\phi(X)\cap U'\), or
    \(\phi(X)\cap U'=\emptyset\);
  \item for each \(U\in\Cov\), \(\ExtS_\phi(U)\cap\Cov'\neq \emptyset\), that is, there is at least one set \(U'\in\Cov'\) such that \(U'\cap\phi(X)=\phi(U)\).
  \end{enumerate}
  That is, each set in \(\mathcal{A}'\) which intersects \(\Img(\phi)\) is an extension of a set in \(\mathcal{A}\) via \(\phi\), and for given \(U\in\Cov'\), \(\Cov'\) contains at least one extension of \(U\). Note that an extension via \(\phi\) always exits for any cover of \(X\).
The collection \(\Cov'\) is called \emph{exhaustive} if it covers \(Y\). If \(\phi\) is an inclusion map of a subspace, we shall simply say an  ``extension'' or ``exhaustive extension''.
\end{notation}

\begin{example}
\label{exa:ext-of-open-cov-closed-sb}
Let \(\phi\colon X\to Y\) be a local homeomorphism onto its image where \(X\) and \(Y\) are spaces. Assume that \(\Img(\phi)\subseteq Z\) is closed, then every open cover \(\{U_i\}_{i\in I}\) of \(X\) has an exhaustive extension in \(Y\) which is 
\[
\{U_i'\subseteq Y: i\in I, U_i' \text{ is open in Y, and } \phi(U_i)= U_i'\cap Z\}\cup \{Y-\Img(X)\}.
\]
\end{example}

As the next example shows, it is not true that any open cover of a subspace admits an exhaustive extension.
\begin{example}
  \label{exa:refined-ext-counterex}
We show that there is a cover of the open subspace \(\R-\{0\}\) of \(\R\) which does not admit an exhaustive extension to \(\R\). Let \(\{\R^-,\R^+\}\) be a cover of \(\R^*\). Then this cover does not have an exhaustive extension to \(\R\). Given any open cover \(\{U_i\}_{i\in I}\) of \(\R\), choose an index \(j\in I\) such that \(0\in U_j\). Then  \(U_j\cap \R^-\neq\emptyset\) and \(U_j\cap \R^+\neq\emptyset\). Hence \(U_j\cap \R^* \nsubseteq \R^-\) as well as \(U_j\cap \R^* \nsubseteq \R^+\).

In fact, one may show that no open cover of \(\R-\{0\}\) admits an exhaustive  extension to \(\R\); the proof uses an argument similar to the one above.
\end{example}

\begin{lemma}
  \label{lem:pulling-and-pasting-measures-3}
Let \(X\) and \(Y\) be locally compact Hausdorff spaces, and \(\lambda\) a measure on \(Y\). Let \(\phi\colon X\to Y\) be a local homeomorphism onto its image. Let \(\phi^*(\lambda)\) be as in Corollary~\ref{cor:pulling-and-pasting-measures-2}. Let \(\mathcal{A}\) be an open of \(X\) consisting of sets in \(\phi^{\HoC}\),  and let \(\mathcal{A}'\) be an extension of \(\mathcal{A}\) via \(\phi\). Assume that for every \(V\in\mathcal{A}\), \(f\in \Contc(X,V)\), every \(V'\in \mathcal{A}'\) which is an extension of \(V\) via \(\phi\), and any extension \(f'\in\Contc(Y,V')\) of \(f\) via \(\phi\), we have
\begin{equation}
  \label{eq:pullback-measure-3}
\phi^*(\lambda)(f)=\lambda(f').
\end{equation}
 Then the following statements are true:
\begin{enumerate}[i)]
\item For every \(V\in \mathcal{A}\) and its extension \(V'\in \mathcal{A}'\), the measure \(\lambda|_{V'}\) is concentrated on \(\phi(V)\).
\item If \(B\defeq \cup_{V'\in \mathcal{A}'}\; V'\), then \(B\subseteq Y\) is open and the measure \(\lambda|_{B}\) is concentrated on \(\phi(X)\).
\item If the open cover \(\mathcal{A}'\) is exhaustive, then \(\lambda\) is concentrated on \(\phi(X)\).
\end{enumerate}
\end{lemma}
\begin{proof}
  (i): Let \(g\in \Contc(V')\). Then Lemma~\ref{lem:loc-homeo-fn-ext-restr}(iii) says that \(g\) is in \(\Ext_\phi(\Res^\phi(g))\) and the hypothesis tells us that
\[
\int_X\Res^\phi(g)\, \dd\phi^*(\lambda)=\int_Y g \,\dd\lambda.
\]
Note that the last term equals \(\int_{V'} g\,\dd\lambda\) since \(g\in \Contc(Y,V')\). Using Equation~\eqref{eq:pullback-measure-2} we see that
\[
\int_X \Res^\phi(g)\;\dd\phi^*(\lambda)= \int_{\phi(V)} g\,\dd\lambda
\]

The above two equations imply that \(\int_{V'} g\,\dd\lambda = \int_{\phi(V)} g\,\dd\lambda\) for any \(g\in \Contc(Y,V')\iso \Contc(V')\). Hence using Lemma~\ref{lem:conc-measu-equi}(ii) we infer that \(\lambda|_{V'}\) is concentrated on \(\phi(V)\).
\smallskip

\noindent (ii): By definition, the space \(B\) is covered by \(V'\in \mathcal{A}'\) where \(V'\) is an extension of some \(V\in \phi^{\HoC}\). For every \(V'\) like this, \(V'\cap \phi(X)=\phi(V)\). Now (i) of this lemma and Lemma~\ref{lem:loc-concentrated-measure} together yield the desired result.
\smallskip

\noindent (iii): In this case, \(B=Y\) and the result follows from~{(ii)} above.
\end{proof}
\medskip

Let \(X\) be a space, and \(\lambda\) and \(\mu\) measures on \(X\). To say that \(\lambda\) is absolutely continuous with respect to \(\mu\), we write ``\(\lambda\ll\mu\)'' and to say that  \(\lambda\) and \(\mu\) are equivalent we write ``\(\lambda\sim\mu\)''.
\begin{lemma}
  \label{lem:RN-positive-makes-pullback-measure-equi}
Let \(X,Y,\phi\) and \(\lambda\) be as in Lemma~\ref{lem:pulling-and-pasting-measures-3}. Let \(\mu\) be another measure on \(Y\).
\begin{enumerate}[i)]
\item If \(\mu\ll\lambda\), then \(\phi^*(\mu)\ll\phi^*(\lambda)\). Moreover, \(\frac{\dd\phi^*(\mu)}{\dd\phi^*(\lambda)}=\frac{\dd\mu}{\dd\lambda}\circ \phi\).
\item If \(\mu\sim\lambda\), then \(\phi^*(\mu)\sim\phi^*(\lambda)\).
\end{enumerate}
\end{lemma}
\begin{proof}
(i): Due to Corollary~\ref{cor:pulling-and-pasting-measures-2}, for \(U\in \phi^{\HoC}\) and \(f\in \Contc(X;U)\), we have that
\begin{align}
 \int_X f(x)\,\dd\phi^*(\mu)(x)\notag
&= \int_{\phi(X)} f'(y)\,\dd\mu(y)\notag\\
&= \int_{\phi(X)} f'(y)\,\frac{\dd\mu}{\dd\lambda}(y)\;\dd\lambda(y)\label{eq:pulback-qui-lemma-1}
\end{align}
for \(f'\in \Ext_\phi(f)\). Let
\[
F= f\cdot \frac{\dd\mu}{\dd\lambda}\circ\phi
\]
be a function on \(X\). Then \(F\in \Bor(X;U)\) is integrable. Now note that \(f'\,\frac{\dd\mu}{\dd\lambda} \in \ExtB_\phi\left(F\right)\) and compute
\begin{align*}
 \int_X \left(f\, \frac{\dd\mu}{\dd\lambda}\circ \phi\right)\dd\phi^*(\lambda)(x)
&=  \int_X F\,\dd\phi^*(\lambda)(x)\\
&=\int_{\phi(X)} f'(y)\frac{\dd\mu}{\dd\lambda}(y)\,\dd\lambda(y)\\
&=\int_Xf(x)\,\dd\phi^*(\mu)(x) && \text{(due to Equation~\eqref{eq:pulback-qui-lemma-1}).}
\end{align*}
The second equality in the above computation uses the measurable version of Equation~\eqref{eq:pullback-measure-2}). This shows that \(\phi^*(\mu)\ll\phi^*(\lambda)\) and \(\frac{\dd\phi^*(\mu)}{\dd\phi^*(\lambda)}=\frac{\dd\mu}{\dd\lambda}\circ \phi\).

\noindent (ii): Proof of (i) above shows that \(\frac{\dd\phi^*(\mu)}{\dd\phi^*(\lambda)}=\frac{\dd\mu}{\dd\lambda}\circ \phi\). If \(\mu\sim\lambda\), then \(\frac{\dd\mu}{\dd\lambda}>0\). Hence \(\frac{\dd\mu}{\dd\lambda}\circ \phi>0\) which is equivalent to saying that \(\phi^*(\lambda)\sim\phi^*(\mu)\).
\end{proof}

\subsection{Free, proper and transitive actions of groupoids}
\label{sec:act-gpd}

Let \(G\) be a topological groupoid and \(X\) a left \(G\)\nb-space. Let \(Y\) be a space. We call a map \(f\colon X\to Y\) \(G\)\nb-invariant if \(f(\gamma x)=f(x)\) for all \((\gamma,x)\in G\times_{\base}X\).

A map \(f\colon X\to Y\) of locally compact spaces is called proper if the inverse image of a (quasi)-compact set in \(Y\) under \(f\) is quasi-compact. A proper map has various characterisations, for example, see~\cite{Tu2004NonHausdorff-gpd-proper-actions-and-K}*{Proposition 1.6}.

\begin{definition}
  \label{def:mult-fn}
For a groupoid \(G\), a left \(G\)\nb-space \(X\), and a \(G\)\nb-invariant map \(f\) from \(X\) to a space \(Y\), define the function \( \Mlp_f\colon G\times_{\base}X\to X\times_{f,Y,f}X\) by
 \[
  \Mlp_f(\gamma,x)=(\gamma x,x).
 \]
\end{definition}
\noindent A similar definition can be written for a right \(G\)\nb-space. 

Let \(G\) be a groupoid acting on a space \(X\). Then the map \((\gamma, x)\mapsto (\gamma x,x)\) from \(G\times_{\base}X\to X\times X\) can be realised as the one in Definition~\ref{def:mult-fn} by taking the space \(Y\) to be the singleton \(\Pt\) and \(f\) the constant map; in this case, we denote the map by \(\Mlp_0\).
We call the map in Definition~\ref{def:mult-fn} ``the map \(\Mlp_f\)''.

 Let \(G\) be a topological groupoid, \(X\) a left \(G\)\nb-space and \(\Mlp_0\) as in the above paragraphs. We call the action of \(G\) on \(X\)
 \begin{enumerate*}[(i)]
 \item \emph{free} if the map \(\Mlp_0\) is one-to-one;
 \item \emph{proper} if \(\Mlp_0\) is proper;
 \item \emph{basic} if \(\Mlp_0\) is a homeomorphism onto its image (see~\cite{Meyer-Zhu2014Groupoids-in-categories-with-pretopology}).
 \end{enumerate*}
\noindent It can be checked that the action of \(G\) on \(X\) is free if and only if for any two composable pairs \((\gamma,x)\) and \((\gamma',x)\) in \(G\times_{\base}X\), \(\gamma x=\gamma' x\) implies \(\gamma=\gamma'\).

\begin{remark}
\label{rem:prop-act-im-closed}
  Talking about a proper action, a proper map is closed,
  see~\cite{Bourbaki1966Topologoy-Part-1}*{Proposition I, \S 10.1,
    Chapter I}. Thus for a groupoid \(G\) and a \(G\)\nb-space \(X\)
  \begin{enumerate*}[(i)]
  \item the image of \(G\times_{\base}X\)
    under \(\Mlp_0\)
    is a closed subset of \(X\times X\)
    if the action of \(G\) on \(X\) is proper;
  \item the action of \(G\)
    on \(X\)
    is free as well as proper if and only if the image of \(\Mlp_0\)
    is a closed subset of \(X\times X\),
    and \(\Mlp_0\) is a homeomorphisms onto its image; the action is basic.
  \end{enumerate*}
  The inclusion of the open interval \((0,1)\hookrightarrow [0,1]\)
  is a map which is homeomorphism onto its image but is not proper.
  in \((0,1)\) is not compact.
\end{remark}

Let \(G\) be a topological groupoid and \(X\) a left \(G\)\nb-space. Let \(Y\) be a space, and let \(f\colon X\to Y\) be a \(G\)\nb-invariant map. Let \([f]\colon G\7 X\to Y\) be the continuous map induced by \(f\); thus \([f]([x])=f(x)\) for all \(x\in X\). We always bestow \(G\7X\) with the quotient topology.

\begin{lemma}
  \label{lem:m-and-quotient-of-sx}
Let \(G\) be a topological groupoid and \(X\) a left \(G\)\nb-space. Let \(Y\) be a space and \(f\colon X\to Y\) a \(G\)\nb-invariant map. Then the following statements hold:
\begin{enumerate}[i)]
\item The action of \(G\) on \(X\) is free if and only if the map \(\Mlp_f\colon G\times_{\base}X\to X\times_{f,Y,f}X\) is one-to-one.
\item The map \([f]\colon G\7 X\to Y\) is surjective if and only if \(f\colon X\to Y\) is surjective. If \(f\) is open, then so is \([f]\).
\item \([f]\colon G\7 X\to Y\) is one-to-one if and only in the map \(\Mlp_f\) is surjective.
\end{enumerate}
\end{lemma}
\begin{proof}
  \noindent (i): Note that the images of \(\Mlp_f\) and \(\Mlp_0\) are the same subsets of \(X\times X\), since \(f\) is invariant. Let \(\iota\colon X\times_{f,Y,f}X\hookrightarrow X\times X\) be the inclusion map. Then \(\Mlp_f=\Mlp_0\circ \iota\). Since \(\iota\) is one-to-one, \(\Mlp_f\) is one-to-one if and only if \(\Mlp_0\) is one-to-one.

\noindent (ii): The first claim follows from the fact that \(f\) and \([f]\) have the same images. The second claim follows from the universal property of the quotient map of topological spaces.

\noindent (iii): Firstly, assume that \([f]\) is one-to-one. Let \((x,z)\in X\times_{f,Y,f}X\) be given, then \(f(x)=f(z)\). Now the fact that \([f]([x])=f(x)=f(z)=[f]([z])\) along with the injectivity of \([f]\) implies that \([x]=[z]\), that is, \(x=\gamma z\) for some \(\gamma\in G\). Thus \((x,z)=(\gamma z, z)=\Mlp_f(\gamma,z)\).

Conversly, assume that \(\Mlp_f\) is surjective. Let \([x],[z]\in G\7 X\) be such that \([f]([x])=[f]([z])\). Then \(f(x)=f(z)\) and we see that \((x,z)\in X\times_{f,Y,f}X\). Now the surjectivity of \(\Mlp_f\) gives \(\gamma\in G\) such that \((\gamma z,z)=\Mlp_f(\gamma,z)=(x,z)\), thus \([x]=[z]\).
\end{proof}

\begin{definition}
  \label{def:trans-act}
Let \(G\) be a groupoid, \(X\) a \(G\)\nb-space, and \(Y\) a space. For a \(G\)\nb-invariant map \(f\colon X\to Y\) we say that the action of \(G\) on \(X\) is \emph{transitive over \(f\)} if one of the following equivalent conditions hold:
\begin{enumerate}[i), leftmargin=*]
\item the map \(\Mlp_f\) is surjective, that is, for all \((x,x')\in X\times_{f,Y,f}X\), there is \(\gamma\in G\) with \(\gamma x'=x\);
\item the map \([f]\colon G\7 X\to Y\) induced by \(f\) is one-to-one.
\end{enumerate}
\end{definition}
\noindent Lemma~\ref{lem:m-and-quotient-of-sx}~{(iii)} shows that the two conditions in Definition~\ref{def:trans-act} are equivalent. We say that an action of a groupoid \(G\) on a space \(X\) is transitive if it is transitive over the constant map \(X\to{\Pt}\). A groupoid is transitive if and only if the obvious action of the groupoid on its space of units is transitive.
\smallskip

Let \(G\) and \(H\) be groupoids, and \(X\) a \(G\)-\(H\)\nb-bispace. The conditions that \(\Mlp_{s_X}\colon G\times_{\base}X\to X\times_{s_X,\base[H],s_X}X\) is a homeomorphism, or the map \([s_X]\colon G/X\to \base[H]\) is one-to-one (or a homeomorphism) appear in many works concerning equivalences, generalised morphisms, and actions of groupoids. We revisit these conditons in the following remark.
\begin{remark}
  \label{rem:relation-between-sX-quo-sX-Mlp}
Let \(X\) be a \(G\)\nb-\(H\)-bispace for groupoids \(G\) and \(H\). Let \(s_X\colon X\to \base[H]\) be an open surection. Recall the following one-sided classical conditions in for a groupoid equivalence in~\cite{Muhly-Renault-Williams1987Gpd-equivalence}, namely,
\begin{enumerate}[i)]
\item The action of \(G\) on \(X\) is free.
\item The action of \(G\) on \(X\) is proper.
\item The map \(s_X\) induces an isomorphism \([s_X]\colon G/X\to \base[H]\).
\end{enumerate}
  
\noindent 1) Since \(s_X\) is an open surjection, due to Lemma~\ref{lem:m-and-quotient-of-sx}{(ii)}. Hence Condition~{(iii)} above is equivalent to, and can be replaced by, the one that the action of \(G\) on \(X\) is transitive over \(s_X\).

\noindent 2) As the action of \(G\) on \(X\) is free and proper, Remark~\ref{rem:prop-act-im-closed} implies that \(\Mlp_f\) is an isomorphism onto its image \(\Img(G\times_{\base}X)\). But the action of \(G\) on \(X\) is transitive over \(s_X\), hence \(\Img(G\times_{\base}X)= X\times_{s_X,\base[H],s_X}\), that is, \(\Mlp_{s_X}\) is an isomorphism--- this is one of the classical conditions that is given as an alternate to the third one above, and as we see, in present situation, this condition also is equivalent to that the action of \(G\) on \(X\) is transitive over \(s_X\).
\end{remark}

\subsection{Locally free actions of groupoids}
\label{sec:locally-free-actions}

Let \(X\) be a left \(G\)\nb-space where \(G\) is a groupoid. For \(x\in X\), the isotropy group of \(x\), which we denote by \(\Ist(x)_G\), is the group \(\{\gamma\in G: \gamma x=x\}\). The element \(r_X(x)\)  is the unit in \(\Ist(x)_G\).

Let \(G\) and \(X\) be as above, and let \(\gamma\in G\). We say that the action of \(G\) is \emph{locally free at \(\gamma\)} if there is a neighbourhood \(U\subseteq G\) of \(\gamma\) with the property that for all \(\eta\in U\) and \(x\in X\), \(\gamma x=\eta x\) implies \(\gamma=\eta\). In this case, we say that ` \(\gamma\) acts freely on \(X\) in the neighbourhood \(U\)'. Caution: this nomenclature does not mean that for\emph{ any} two elements \(\eta,\delta\in U\) and any \(x\in X\) with \(\eta x= \delta x\implies \eta =\delta\); the statement is true if and only if one of \(\eta\) or \(\eta'\) is \(\gamma\).
 \begin{lemma}
  \label{lem:loc-free-act}
Let \(G\) be a groupoid and \(X\) a \(G\)\nb-space. Assume that the momentum map \(r_X\) is surjective. Then the following statements are equivalent.
\begin{enumerate}[i)]
\item The action of \(G\) on \(X\) is locally free at every \(u\in \base\).
\item The action of \(G\) on \(X\) is locally free at every \(\gamma\in G\).
\item The isotropy group of every \(x\in X\) is discrete.
\end{enumerate}
\end{lemma}
\begin{proof}
We prove that \((i)\mathrel{\iff}(ii)\) and \((i)\iff (iii)\).

\noindent \((i)\iff (ii)\): The implication \((ii)\mathrel{\implies}(i)\) is clear. To prove the converse, let \(\gamma\in G\) be given. Let \(Z\subseteq G\) be an open neighbourhood of \(s_G(\gamma)\) in which \(s_G(\gamma)\) acts freely on \(X\). Let \(\textup{mult}_G\colon G\times_{s_G,\base,r_G}G\to G\) denote the multiplication map, \(\textup{mult}_G(\eta,\eta')= \eta\eta'\). Then \( \textup{mult}_G(\gamma\inverse, \gamma)=s_G(\gamma)\). Using the continuity of \(\textup{mult}_G\), choose open neighbourhoods \(U_1\subseteq G\) of \(\gamma\inverse\), and \(U_2\subseteq G\) of \(\gamma\) with \(\textup{mult}_G(U_1\times_{\base}U_2)\subseteq Z\).
 Let \(\textup{inv}_G\colon G\to G\) denote the inverting function \(\eta\mapsto\eta\inverse\). Then \(U=\textup{inv}_G(U_1)\cap U_2\) is an open neighbourhood containing \(\gamma\), to see this, we observe that \(\gamma=(\gamma\inverse)\inverse\in \textup{inv}_G(U_1)\) and \(\gamma \in U_2\), hence \(\gamma\in U\). And \(U\) is open because \(\textup{inv}_G\) is a homeomorphism.

We claim that \(\gamma\) acts freely on \(X\) in \(U\). Let \(\eta\in U\) and \(x\in X\) be such that \(\gamma x=\eta x\), that is,  \(\eta\inverse\gamma x=\textup{mult}_G(\eta\inverse,\gamma)x=x\). Since \(\eta\in U = \textup{inv}_G(U_1)\cap U_2\), we have that \(\eta\inverse \in U_1\). Then \(\textup{mult}_G(\eta\inverse, \gamma)\in Z\). But \(Z\) acts freely on \(X\). Hence \(\eta\inverse\gamma x= x\) implies \(\eta\inverse\gamma = r_X(x)\), that is, \(\gamma=\eta\).
\smallskip

\noindent \((i)\iff (iii)\):
 Assume that the action of \(G\) is locally free at every \(\gamma\in G\). Then, for each \(x\in X\), we produce a neighbourhood of the unit \(r_X(x)\) in \(\Ist(x)_G\) which contains only the unit. Let \(x\in X\) be given. Let \(U\subseteq G\) be a neighbourhood of the unit \(r_X(x)\in \Ist(x)_G\) in which \(r_X(x)\) acts freely on \(X\). Then \(U\cap\Ist(x)_G=\{r_X(x)\}\) because for every \(\eta\in U\), \(\eta x = x\) is equivalent to \(\eta x = r_X(x)x = x\) which implies that \(\eta=r_X(x)\).

Conversely, assume that \(\Ist(x)_G\) is discrete for all \(x\) in \(X\). Let \(u\in \base\) be given. Then \(u=r_X(x)\) for some \(x\in X\). Let \(U\subseteq G\) be an open neighbourhood with \(U\cap \Ist(x)_G=\{r_X(x)\}\). Then \(u\) acts freely on \(X\) in the  neighbourhood \(U\). Because if \(\gamma\in U\) with \(\gamma y =  y\), then \(\gamma\in\Ist(x)_G\). Thus \(\gamma\in (U\cap\, \Ist(x)_G)=\{r_X(x)\}\). Hence \(\gamma = r_X(x)\).
\end{proof}

Assume that \(G\) and \(X\) are as in Lemma~\ref{lem:loc-free-act}. Then the proof of (i)\(\implies\)(ii) above shows that for \(x\in X\) the isotropy group \(\Ist(x)_G\) is discrete if and only if the action of \(G\) is locally free at \(r_X(x)\in\base\), or equivalently, for \(x\in X\) the group \(\Ist(x)_G\) is discrete if and only if the action of \(G\) is locally free at \(\gamma\in G\) for some \(\gamma\in G_{r_X(x)}\).

\begin{definition}
  \label{def:loc-free-act}
  \begin{enumerate}[i),leftmargin=*]
  \item An action of a groupoid \(G\) on a space \(X\) is called \emph{locally free} if any of \((i)\)---\((iii)\) in Lemma~\ref{lem:loc-free-act} holds.
  \item An action of a groupoid \(G\) on a space \(X\) is called \emph{locally locally free} if every \(\gamma\in G\) has a neighbourhood \(U\subseteq G\) and every \(y\in X^{s_G(\gamma)}\) has a neighbourhood \(V\subseteq X\) such that for all \(\eta\in U\) and \(x\in V\), \(\gamma x=\eta x\) implies \(\gamma=\eta\).
  \end{enumerate}
\end{definition}
Clearly, a locally free action is locally locally free. For a group, the converse holds. To see it, one may use an argument as in the proof of Lemma~\ref{lem:loc-free-act}{(i)\(\implies\)(iii)}, and observe that all isotropy groups are discrete.
\begin{example}
  \label{exa:R-S-1}
The action of \(\R\) on the unit circle \(\UCr\), \(r\cdot\me^{2\pi\mi\theta}=\me^{2\pi\mi(\theta+r)}\) is not free; here \(r\in \R\) and \(\me^{2\pi\mi\theta}\in \UCr\). But this action is locally free since \(\Ist(\me^{2\pi\mi\theta})_{\R}\iso \Z\) is discrete for all \(\me^{2\pi\mi\theta} \in \UCr\).
\end{example}

Let \(G\) be a groupoid acting on a space \(X\). For \(\gamma\in G\), we say that the action of \(G\) is \emph{strongly locally free at \(\gamma\)} if there is a neighbourhood \(U\subseteq G\) of \(\gamma\) with the property that for all \(\eta,\delta\in U\) and \(x\in X\), \(\eta x= \delta x\) implies that \(\eta = \delta\). We say that `the neighbourhood \(U\) (of \(\gamma\))' acts freely on \(X\).

 \begin{lemma}
  \label{lem:str-loc-free-act}
Let \(G\) be a groupoid and \(X\) a \(G\)\nb-space. Assume that the momentum map \(r_X\) is surjective. Let \(Y\) be a space and \(f\colon X\to Y\) a \(G\)\nb-invariant map. Then the following statements are equivalent.
\begin{enumerate}[i)]
\item The action of \(G\) on \(X\) is strongly locally free at every \(u\in \base\).
\item Every unit \(u\in \base\) has a neighbourhood \(U\) in \(G\) such that such that the restriction of the map \(\Mlp_f\) to \(U\times_{\base}X\) is one-to-one.
\item The action of \(G\) on \(X\) is strongly locally free at every \(\gamma\in G\).
\item Every element \(\gamma\in G\) has a neighbourhood \(U\) in \(G\) such that the restriction of the map \(\Mlp_f\) to \(U\times_{\base}X\) is one-to-one.
\end{enumerate}
\end{lemma}
\begin{proof}
It can be readily seen that (i) and (ii) are different phrasing of the same fact, and so are (iii) and (iv).
The implication (iii)\(\implies\)(i) is obvious and the proof of its converse can be written on exactly the same lines as that of the proof of (ii)\(\implies\)(i) of Lemma~\ref{lem:loc-free-act}. We sketch the proof roughly: let \(\gamma \in G\). Let \(Z\subseteq G\) be an open neighbourhood of \(s_X(\gamma)\) that acts freely on \(X\). Let \(U_1\) and \(U_2\) be as in the proof of Lemma~\ref{lem:loc-free-act}. Then as shown in the proof of Lemma~\ref{lem:loc-free-act}, \(U=\textup{inv}_G(U_1)\cap U_2\) is an open neighbourhood of \(\gamma\) in \(G\). We claim that \(U\) acts freely on \(X\). Let \(\eta,\delta\in U\), then \(\eta x=\delta x\iff \delta\inverse \eta x = x\). As in the same proof above, we observe that \(\delta\inverse \eta \in Z\). And since \(Z\) acts freely on \(X\), we must have \(\delta\inverse\eta=r_G(\eta)\), that is, \(\eta=\delta\).
\end{proof}

\begin{definition}
  \label{def:str-loc-free-act}
  \begin{enumerate}[i),leftmargin=*]
  \item An action of a groupoid \(G\) on a space \(X\) is called \emph{strongly locally free} if any one of {(i)}-{(iv)} in Lemma~\ref{lem:str-loc-free-act} above holds.
  \item  An action of a groupoid \(G\) on a space \(X\) is called \emph{locally strongly locally free} if every \(\gamma\in G\) has a neighbourhood \(U\subseteq G\) and every \(y\in X^{s_G(\gamma)}\) has a neighbourhood \(V\subseteq X\) such that for all \(\delta,\eta\in U\) and \(x\in V\), \(\eta x=\delta x\) implies \(\eta=\delta\). 
  \end{enumerate}
\end{definition}
Clearly, a strongly locally free action is locally free as well as locally strongly locally free. As discussed on page~\pageref{page:gp-LF-is-SLF}, for groups, all these four notions of free action coincide.
\begin{example}
  \label{exa:local-free-and-not-str-loc-free-action}
Let
\begin{align*}
G&=\{(0,0)\}\cup\{(1/n,i/n^2)\in \R^2: n\in\N \text{ and } 0\leq i\leq n\},\\
\base&\defeq  \{(0,0)\}\cup \{(1/n,0)\in \R^2: n\in\N\}
\end{align*}
and equip both sets with the subspace topology coming from \(\R^2\); then \(\base\) is a subspace of \(G\). We think of \(G\) as a group bundle over the space \(\base\): for \(n\in\N\), the fibre over \((1/n,0)\) is
\[
\{(1/n,i/n^2): 0\leq i\leq n\}=\{(1/n,0),(1/n,1/n^2),\dots,(1/n,n/n^2)\}
\]
which we identify with the finite cyclic group \(\Z/(n+1)\Z\) of order \(n\) by the map \(i/n^2\mapsto [i]\in \Z/(n+1)\Z\). The fibre over the origin \((0,0)\) is the trivial group. One may see that this group bundle is a continuous group bundle. We make this bundle into a groupoid in the standard way: \(\base\) is a the space of units, the projection onto the first factor, \(G\to \base\), \((x,y)\mapsto (x,0)\), is the source as well as range map for \(G\). The composite of \((1/n,i/n^2),(1/n,j/n^2)\in G\) is \((1/n,[i+j]/n^2)\) and the inverse of \((1/n,i/n^2)\) is \((1/n,(n+1-i)/n)\). For the fibre over zero, the operations are defined in the obvious way.

Let \(X=\base\) and \(s_X\colon \base\to G\) the inclusion map. Then we claim that the trivial left action of \(G\) on \(X\) is locally free but not strongly locally free. The action is locally free because for every \((x,0)\in\base\) the isotropy group \(\Ist_G((x,0))\) is discrete, to be precise, for \(n\in\N\)  \(\Ist_G((1/n,0))\iso\Z/(n+1)\Z\), and \(\Ist_G((0,0))\) is the trivial group. Now Lemma~\ref{lem:loc-free-act} tells us that action of \(G\) is locally free. We claim that given no neighbourhood of \((0,0)\) in \(G\) acts freely on \(X\).

Let \(U\) be an open neighbourhood of \((0,0)\) and let \(U'\subseteq \R^2\) be an open set such that \(U=U'\cap G\). Let \(\epsilon>0\) be sufficiently small so that, \(W'\defeq (-\epsilon,\epsilon)\times(-\epsilon,\epsilon)\), an open square centred at the origin is contained in \(U'\). Put \(W=W'\cap G\), then \(W\subseteq U\) is an open set containing \((0,0)\). Since the sequences  \(\{(1/n,1/n)\}_{n\in\N}\) and \(\{(1/n,0)\}_{n\in\N}\) converge to \((0,0)\) in \(G\) as well as \(\R^2\), for a sufficiently large \(n\) we have \((1/n,0),(1/n, 1/n)\in W\subseteq U\). Thus for \((1/n,0),(1/n, 1/n)\in G\) and for \(x\defeq(1/n,0)\in X\), we have
\[
x=(1/n,0)\cdot x=(1/n,1/n)\cdot x;
\]
here \(\cdot\) is the trivial action of \(G\) on \(X\). Thus \(U\) does not act strongly freely on \(X\).

Furthermore, the contrapositive of Proposition~\ref{prop:etale-gpd-loc-free-act} says that \(G\) is not {\etale}. Thus \(G\) is a groupoid in which the fibre over every unit is discrete but the groupoid is not {\etale}.
\end{example}

In this article, locally strongly locally free actions are what interest us the most.
 \begin{lemma}
  \label{lem:loc-str-loc-free-act}
Let \(G\) be a groupoid and \(X\) a \(G\)\nb-space. Assume that the momentum map \(r_X\) is surjective. Let \(f\colon X\to Y\) be a \(G\)\nb-invariant map to a space \(Y\). Then the following statements are equivalent.
\begin{enumerate}[i), leftmargin=*]
\item The action of \(G\) on \(X\) is locally strongly locally free at every \(\gamma\in G\).
\item The map \(\Mlp_f\) is locally one-to-one.
\end{enumerate}
 Additionally, if \(\Mlp_f\) is an open map onto its image, then the following statements are equivalent.
\begin{enumerate}[resume, leftmargin=*, label=\roman*)]
\item The action of \(G\) on \(X\) is locally strongly locally free at every \(\gamma\in G\).
\item The map \(\Mlp_f\) is a local homeomorphism onto its image.
\end{enumerate}
\end{lemma}
\begin{proof}
  Assume that the action of \(G\) is locally strongly locally free at every element \(\gamma\). Given a point \((\gamma,x)\in G\times_{\base}X\), we need to find a neighbourhood of it such that the restriction of \(\Mlp_f\) to this neighbourhood is one-to-one. Let \(U\) be a neighbourhood of \(\gamma\) and \(V\subseteq X\) be the one of \(x\) with the property that for all \(\delta,\eta\in U\) and \(x\in V\), \(\eta x=\delta x\implies\eta=\delta\). Then \(U\times_{\base}V\) is the neighbourhood to which the restriction \(\Mlp_f|_{U\times_{\base}X}\) is one-to-one.

Conversely, let \(\gamma\in G\) is given. Choose \(x\) in \(X\) which is composable with \(\gamma\); this can be done since \(r_X\) is assumed to be surjective. Let \(W\) be an open neighbourhood of \((\gamma,x)\) in \(G\times_{\base}X\) such that \(\Mlp_f|_{W}\) is one-to-one. Let \(U\subseteq G\) and \(V\subseteq X\) be open neighbourhoods of \(\gamma\) and \(x\), respectively, such that \(U\times_{\base}V\) is a basic open neighbourhood of \((\gamma,x)\) and \(U\times_{\base}B\subseteq V\). Then \(U\) and \(V\) are the required neighbourhoods which prove the claim.

The rest of the proof is an easy exercise.
\end{proof}

One can see that a free action is locally free, locally locally free, strictly locally free and locally strictly locally free.

It is well-known that action of a group on a space is called locally free if the isotropy group of each point of the space is discrete, for example, see~\cite{Candel-Conlon2000Foliations-I}*{Definition 11.3.7} which discusses the special case of locally free action of a group on a foliation. It can be shown that for groups a locally free action as in Definition~\ref{def:loc-free-act} and a strictly locally free one as in Definition~\ref{def:str-loc-free-act} are equivalent:\label{page:gp-LF-is-SLF} assume that \(G\) is a group which acts on a space \(X\) and the action is locally free. Thus for each point in \(X\) the corresponding isotropy group is discrete. Let \(\gamma\in G\) be given. Choose \(x\in X\) and let \(U'\subseteq G\) be a neighbourhood of \(e\), the unit of \(G\), that intersects \(\Ist(x)_G\) only at \(e\). Let \(U\subseteq U'\) be a neighbourhood such that \(U\inverse U\subseteq U'\). Then \(\gamma U\) is the neighbourhood of \(\gamma\) which acts freely on \(X\). This is because, if \(\eta,\delta\in \gamma U\), then \(\eta'\defeq \gamma\inverse \eta\), \(\delta'\defeq\gamma\inverse \delta\) and \({\delta'}\inverse\eta'\) are in \(U'\). And then \(\eta' x = \delta' x \iff {\delta'}\inverse \eta' x=x\). Which implies \(\eta'=\delta'\) and hence \(\eta=\delta\). The neighbourhood \(U\) of \(U'\) above can be constructed, for example, see~\cite{Folland1995Harmonic-analysis-book}*{Proposition 2.1 b}.

The proof above also shows that, for groups, a locally locally free action is locally strongly locally free. Let \(G\) and \(X\) be as above, and assume that the action of \(G\) on \(X\) is locally locally free. Let \(U'\times V\subseteq G\times X\) a basic neighbourhood of \((e,x)\in G\times X\) such that \(\Mlp_0\) is a one-to-one when restricted to \(U\times \). Now the above proof goes through with \(X\) replaced by \(V\).

Thus for groups, the notion of locally free action is equivalent to a locally locally free action (see page~\pageref{def:loc-free-act}) and a strongly locally free action, and a locally locally free action is same as a locally strongly free action. This allows us to conclude that for groups, a strongly locally free action is same as a locally strongly free action. This shows that the four notions of locally free action coincide for groups.

The proofs above, however, do not work for groupoids. The reason is that there may be more than a single unit in a groupoid. As mentioned earlier, locally strictly locally free actions interest us the most. Since, in the case of groups, every locally free action is a locally strictly free action, we pocket a large class of examples of locally strictly locally free actions.

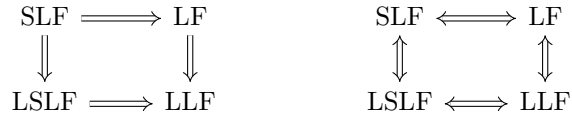
\begin{figure}[htb]
  \centering
\[
\begin{tikzcd}[scale=2]
\text{SLF}\arrow[Rightarrow]{r}\arrow[Rightarrow]{d}
 & \text{LF} \arrow[Rightarrow]{d}& &\text{SLF}\arrow[Leftrightarrow]{r}\arrow[Leftrightarrow]{d}& \text{LF} \arrow[Leftrightarrow]{d}\\
\text{LSLF} \arrow[Rightarrow]{r} & \text{LLF}& &\text{LSLF} \arrow[Leftrightarrow]{r} & \text{LLF}
 \end{tikzcd}
\]
\caption{The different types of free actions and their interrelations; the first square shows the case of groupoids, and the second one shows the case of groups; here S stands for \emph{strognly}, L for \emph{locally} and F for \emph{free}.}
\label{fig:LF-actions-rel}
\end{figure}
Speaking of the square for groupoid in Figure~\ref{fig:LF-actions-rel}, Example~\ref{exa:local-free-and-not-str-loc-free-action} shows that a locally free action need not be strongly locally free. But as-of-now, we do not know if a locally locally free should be locally free, or if a locally strictly locally free actions should be strictly locally free.
\medskip

Now we turn our attention to a special class of groupoids, namely, {\etale} groupoids. Here is an example first:
\begin{example}
  \label{exa:disc-grp}
Let \(\Z/2\Z=\{-1,1\}\) act on \([-1,1]\) as \(\pm 1\cdot t=\pm 1\) for all \(t\in [-1,1]\). This action is not free since \(\pm 1\cdot 0 =0\). But \(\Ist(t)_{\Z/2\Z}\) is discrete for each \(t\in [-1,1]\). Hence, this action is locally free.

In general, if \(G\) is a discrete group, then any \(G\) action is locally free. Because if \(X\) is a \(G\) space, then for each \(x\in X\), \(\Ist(x)_G\subseteq G\) is discrete. This example generalises very well to {\etale} groupoids. A groupoid is called {\etale} if its range map is a local homeomorphism.
\end{example}

Recall from~\cite{Renault1980Gpd-Cst-Alg}*{Definition 1.2.6} that a locally compact groupoid is called \(r\)\nb-discrete if its space of units is an open subset. A groupoid is {\etale} if and only if it is \(r\)\nb-discrete and has a Haar system (Proposition~\cite{Renault1980Gpd-Cst-Alg}*{1.2.8}).

\begin{proposition}
  \label{prop:etale-gpd-loc-free-act}
Let \(G\) be an \(r\)\nb-discrete groupoid. Then the action of \(G\) on \(X\) is strongly locally free.
\end{proposition}
\begin{proof}
We claim that the restriction of \(\Mlp_f\) to \(\base\times_{\base}X\) is a homeomorphism onto its image. Because the map
\[
    (r_X,\Id)\colon X\to \base\times_{\base}X, \quad x\mapsto (r_X(x),x),
\]
is a homeomorphism, and so is the diagonal embedding
\[
\Dia_X\colon X\hookrightarrow X\times_{f,Y,f} X,\quad x\mapsto (x,x).
\]
Now observe that \(\Mlp_f\) is the composite \(\Dia_X\circ(r_X,\Id)\inverse\); this proves that claim. In this case, Lemma~\ref{lem:str-loc-free-act} implies that the action of \(G\) on \(X\) is strongly locally free.
\end{proof}
\medskip

We conclude this section by introducing a notation. Let \(G\) be a locally compact groupoid, \(X\) a \(G\)\nb-space, and let the \(G\) action be locally strongly locally free. Let \(Y\) be a space and \(f\colon X\to Y\) a \(G\)\nb-invariant map. Assume that \(\Mlp_f\) is open onto its image. Then every \((\gamma,x)\in G\times_{\base}X\) has a basic neighbourhood \(A\times_{\base}B \subseteq G\times_{\base}X\) restricted to which \(\Mlp_f\) is one-to-one. Thus \(G\times_{\base}X\) has an open cover consisting of basic open sets restricted to which \(\Mlp_f\) is a homeomorphism onto its image. One may choose an open cover where each basic open set is cocompact. Note that since \(X\) is Hausdorff, the fibre product \(G\times_{\base}X\) is Hausdorff.

\begin{notation}[Notation and discussion]
  \label{not:loc-sec}
Let \(G\) be a locally compact groupoid 
and \(X\) a left \(G\)\nb-space. Let \(Y\) be a space and \(f\colon X\to Y\) be a \(G\)\nb-invariant map. Assume that \(\Mlp_f\colon G\times_{\base}X\to X\times_{f,Y,f} X\), \((\gamma, x)\mapsto (\gamma x,y)\), is a local homeomorphism onto its image , that is, the action of \(G\) on \(X\) is locally strongly locally free and \(\Mlp_0\) is open onto its image.

 Let \(U\subseteq G\) and \(V\subseteq X\) be open, Hausdorff sets with \(\Mlp|_{U\times_{\base}V}\) a homeomorphism onto its image. Then for each \((x,y)\in \Mlp(U\times_{\base}V)\), there is a unique \(\gamma\in U\) with \(\gamma y=x\). We define the function 
\[
\LSec|_{\Mlp(U\times_{\base}V)}\colon \Mlp(U\times_{\base}V)\to G
\]
which sends \((x,y)\in \Mlp(U\times_{\base}V)\) to the unique \(\gamma\in U\) such that \(x=\gamma y\). Note that if \(pr_1\colon G\times_{\base}X\to G\) is the projection on the first factor, then \(\LSec|_{\Mlp(U\times_{\base}V)}=pr_1\circ\Mlp|_{\Mlp(U\times_{\base}V)}\inverse\). This shows that \(\LSec_{\Mlp(U\times_{\base}V)}\) is continuous. 

We call \(\LSec|_{\Mlp(U\times_{\base}V)}\) \emph{ the local section} at \(U\times_{\base}V\). Using \(\LSec|_{\Mlp(U\times_{\base}V)}\) we may write \(\Mlp|_{\Mlp(U\times_{\base}V)}\inverse\colon \Mlp(U\times_{\base}V)\to U\times_{\base}V\) as
\begin{align}
\Mlp|_{\Mlp(U\times_{\base}V)}\inverse(x,y)&=(\LSec|_{\Mlp(U\times_{\base}V)}(x,y),y)\notag\\
&=(\LSec|_{\Mlp(U\times_{\base}V)}(x,y),\LSec|_{\Mlp(U\times_{\base}V)} (x,y)\inverse x).\notag
\end{align}
If one writes \(\gamma_{(x,y)}\) for \(\LSec|_{\Mlp(U\times_{\base}V)}(x,y)\), then the above formula looks better:
\begin{align}
\Mlp|_{\Mlp(U\times_{\base}V)}\inverse(x,y)&=(\gamma_{(x,y)},\,y) \label{eq:loc-sec-inv-y}\\
&=(\gamma_{(x,y)},\, \gamma_{(x,y)}\inverse x).\label{eq:loc-sec-inv-x}
\end{align}
Equation~\eqref{eq:loc-sec-inv-y} expresses the local section \(\LSec|_{\Mlp(U\times_{\base}V)}\) of \(U\times_{\base}V\) in terms of \(y\) whereas Equation~\eqref{eq:loc-sec-inv-x} expreses it in terms of \(x\).
\end{notation}

\subsection{Proper groupoids and cutoff functions}
\label{sec:prop-group-cutoff}
This subsection is based on Section~{1.2} of~\cite{Holkar2017Composition-of-Corr}.

Let \(G\) be a groupoid. We call \(G\) \emph{proper} if the map \(s_G\times r_G\colon G\to \base\times\base\), \(s_G\times r_G(\gamma)=(s_G(\gamma), r_G(\gamma))\) is proper. An action of \(G\) on a space \(X\) is proper if the transformation groupoid \(G\ltimes X\) is proper; this definition is equivalent to the one given after Definition~\ref{def:mult-fn} on page~\pageref{def:mult-fn}. It is well-known that if the action of \(G\) on \(X\) is proper, then \(G/X\) inherits many \emph{good} topological properties from \(X\), that is, if \(X\) is locally compact (Hausdorff or paracompact or  second-countable), then \(G/X\) is also locally compact (respectively, Hausdorff or paracompact or second-countable).

For a groupoid \(G\), the space of units \(\base\) carries an action of \(G\) for which \(s_G\) is the momentum map, and the action is that for \(u\in\base\) and \(\gamma\in G_u\) \(\gamma u=r_G(\gamma)\). It can be checked that \(G\) is proper if and only if this action of \(G\) on \(\base\) is proper. We call this action \emph{the} left action of \(G\) on its space of units. \emph{The} right action of \(G\) on its space of units is defined similarly.

Let \(G\) be a groupoid. For nonempty subsets \(V,W\subseteq\base\), we denote the set \(\{\gamma\in G: r_G(\gamma)\in V\text{ and } s_G(\gamma)\in W\}\) by \(G^V_W\). If \(W=V\), then \(G^V_V\) is a groupoid called the \emph{restriction groupoid} for the subset \(V\subseteq \base\). The space of units of the restriction groupoid \(\base[G_V^V]\) is \(V\). \(G^V_V\) is a topological subgroupoid of \(G\) when bestowed with the subspace topology of \(G\).
\begin{lemma}
  \label{lem:rest-prop-gpd}
Let \(G\) be a proper groupoid and \(V\) a subset of \(\base\).
\begin{enumerate}[i)]
\item If \(G\) is proper the so is the restriction groupoid \(G_V^V\).
\item The inclusion map \(V\hookrightarrow \base\) induces a topological embedding \(G/V\hookrightarrow G/\base\).
\end{enumerate}
\end{lemma}
\begin{proof}
(i) Firstly, note that \(G^V_V=(r_G\times s_G)\inverse(V\times V)\). Now the claim of the lemma follows from~\cite{Bourbaki1966Topologoy-Part-1}*{Chapter I, \S 10 No.\,1, Proposition 3(a)}.

\noindent
(ii) First of all observe that for every \(v\in V\) the equivalence class of \(v\) in \(G/V\) and \(G/\base\) are the same sets. Thus the inclusion \(V\hookrightarrow \base\) does induce a well-defined one-to-one map \(G/V\hookrightarrow G/\base\). This map is continuous because of the universal property of the quotient space, and is open onto its image because \(V\hookrightarrow\base\) is so.
\end{proof}

Let \(G\) and \(V\) be as in Lemma~\ref{lem:rest-prop-gpd}. Assume that \(r_G\) is an open map. If \(V\) is an open subset of \(\base\), then {(ii)}~of Lemma~\ref{lem:rest-prop-gpd} says that \(G/V\) is homeomorphic to an open subset of \(G/\base\) since the quotient map \(\base\to G/\base\) is open in this case. Thus if \(G/\base\) is paracompact, then \(G/V\) is paracompact.
\begin{lemma}
  \label{lem:cutoff-function}
Let \(G\) be a locally compact proper groupoid with the range map open. Assume that \(G/\base\) is paracompact for the action of \(G\) on its space of units \(\base\). Then there is a continuous real-valued function \(F\) on \(\base\) which has the following properties
\begin{enumerate}[i)]
  \item $F$ is not identically zero on any equivalence class for the action of \(G\) on \(\base\);
  \item for every compact subset $K$ of $G/\base$, the intersection of $r_G\inverse(K)$ with $\supp(F)$ is compact.
  \end{enumerate}
\end{lemma}
\begin{proof}
  This proof follows from that of~\cite{Holkar2017Composition-of-Corr}*{Lemma 1.7}.
\end{proof}
Let \((G,\alpha)\) be a locally compact proper groupoid with a Haar system. Then \(r_G\) is open. Let \(F\) be a function on \(\base\) as in Lemma~\ref{lem:cutoff-function}. Now the proof of \cite{Holkar2017Composition-of-Corr}*{Lemma 1.7} shows that the function \(\bar F\) on \(\base\) defined as
\[
\bar F(u)=\int_G F\circ s_G(\gamma) \;\dd\alpha^u(\gamma)
\]
is continuous and \(0<\bar F(u)<\infty\) for all \(u\in\base\). The discussion in the same proof also shows that the function
\[
c_G\defeq \frac{F}{\bar F}\circ s_G
\]
on \(G\) is continuous has the property that 
\[
\int_{G^u} c_G(\gamma) \;\dd\alpha^u(\gamma) =1
\]
for all \(u\in\base\). We call \(c_G\) a cutoff function for \((G,\alpha)\). 
\begin{example}
  \label{exa:cutoff-for-H-space}
  Let \((H,\beta)\) be a locally compact groupoid with a Haar system and \(X\) a proper rigt \(H\)\nb-space with \(X/H\) paracompact. Then \(X\rtimes H\) is a locally compact proper groupoid. Recall that for \((x,\eta)\in X\rtimes H\), \[
(x,\eta)\inverse=(x\eta,\eta\inverse),\quad r_{X\rtimes H}(x,\eta)= x,\quad  \text{and}\quad s_{X\rtimes H}(x,\eta)=x\eta.
\] It is well known that the Haar system \(\beta\) induces a Haar system \(\tilde\beta\) for \(X\rtimes H\): for \(x\in X=\base[(X\rtimes H)]\) and \(f\in \Contc(X\rtimes H)\)
\[
\int_{X\rtimes H} f\,\dd\tilde\beta^u\defeq \int_{H} f(x,\eta)\,\dd\beta^{s_X(x)}(\eta).
\]
Let \(c\) be a function on \(X\) as in Lemma~\ref{lem:cutoff-function}. Then
\[
\bar c(u)=\int_{X\rtimes H} c\circ s_{X\rtimes H}(x,\eta) \;\dd\tilde\beta^u(\eta)= \int_{H} c(x\eta) \;\dd\beta^u(\eta),
\]
where \(u\in X\), is the function similar to \(\bar F\) above and 
\[
c_{X\rtimes H}= \frac{c}{\bar c}\circ s_{X\rtimes H}
\]
is a cutoff function like \(c_G\) above. Note that \(\bar c\) is \(H\) and \(X\rtimes H\)\nb-invariant. In this example, we call the function \(c\) a \emph{pre-cutoff} function and \(c/\bar c\) \emph{the normalised pre-cutoff function} corresponding to \(c\). Thus, the cutoff function \(c_{X\rtimes H}\) is the composite of the \emph{the normalised pre-cutoff function} corresponding to \(c\) and the source map of the groupoid \(X\rtimes H\).  We shall use this example and notation in (4) of the proof of the main theorem, Theorem~\ref{thm:prop-corr-main-thm}.
\end{example}
\subsection{Compact operators}
\label{comp-op}

Let \((H,\beta)\) be a locally compact groupoid with a Haar system. Let \((X,\lambda)\) be a pair consisting of a proper right \(H\)\nb-space \(X\) and an \(H\)\nb-invariant family of measures on \(\lambda\) along the momentum map \(s_X\). For $\zeta,\xi \in C_c(X)$ and $g \in C_c(H)$ define
\begin{equation}\label{def:r-act-inerpro}
 \left\{\begin{aligned}
(\zeta\cdot g )(x) &\defeq \int_{H^{s_X(x)}} \zeta(x\eta)\, g({\eta}\inverse) \; \dd\beta^{s_X(x)}(\eta) \text{ and }\\
\langle \zeta, \xi \rangle (\eta) &\defeq \int_{X_{r_H(\eta)}}\, \overline{\zeta(x)} \xi(x\eta) \; \dd\lambda_{r_H(\eta)}(x).
  \end{aligned}\right.
\end{equation}
Using the above equations one may make \(\Contc(X)\) into an inner product \(\Contc(H)\)\nb-module which completes to a Hilbert \(\Cst(H,\beta)\)\nb-module; see~\cite{Renault1985Representations-of-crossed-product-of-gpd-Cst-Alg}*{Corollaire 5.2} for details. We denote this Hilbert \(\Cst(H,\beta)\)\nb-module by \(\Hilm(X)\).

Let \((X,\lambda)\) and \((H,\beta)\) be as above. Then \(X\times_{s_X,\base[H],s_X}X\) carries the diagonal action of \(H\) which is proper; for \((x,y)\in X\times_{s_X,\base[H],s_X}X\) and \(\eta\in H^{s_X(x)}\) the action is \((x,y)\eta=(x\eta,y\eta)\). Let \(G\) denote the quotient space \((X\times_{s_X,\base[H],s_X}X)/H\). For \(f,g\in \Contc(G)\) define
\begin{align}
 f*g([x,z])&=\int_X f([x,y])g([y,z])\,\dd\lambda_{s_X(z)}(y)\label{eq:hgpd-con}\\
  f^*([x,y])&=\overline {f[y,x]}\label{eq:hgpd-inv}.
\end{align}
Then \(f^*\) is in \(\Contc(G)\) and~\cite{Holkar-Renault2013Hypergpd-Cst-Alg} shows that \(f*g\) is well-defined and lies in \(\Contc(G)\). When equipped with the convolution in Equation~\eqref{eq:hgpd-con} and the involution~\eqref{eq:hgpd-inv}, \(\Contc(G)\) becomes a {\Star}algebra.~\cite{Holkar-Renault2013Hypergpd-Cst-Alg}*{Theorem 2.2} shows the representations of the {\Star}algebra \(\Contc(H)\) induces  representations of \(\Contc(G)\). Using these induced representations we complete \(\Contc(G)\)  to a \(\Cst\)\nb-algebra which we denote by \(\Cst(G)\); see~\cite{Holkar-Renault2013Hypergpd-Cst-Alg} for details.

In~\cite{Renault2014Induced-rep-and-Hpgpd}, Renault shows that the quotient space \(G\) and the above \(\Cst\)\nb-algebra corresponding to it carry a well-defined mathematical structure. He shows that \(G\) is the \emph{spatial hypergroupoid} associated with the proper \(H\)\nb-space \(X\). The \(H\)\nb-invariant family of measures \(\lambda\) induces a Haar system \(\lambda_G\) on \(G\) and the algebra \(\Cst(G)\) is the \(\Cst\)\nb-algebra \(\Cst(G,\lambda_G)\) of the hypergroupoid \(G\) equipped with the Haar system \(\lambda_G\).

\begin{remark}
  \label{rem:cst-cat}[A long remark that explains the origins of Equations~\eqref{eq:hgpd-con},\eqref{eq:hgpd-inv} and~\eqref{eq:hgpd-l-act-inpro}.]

We advise reader to have look at Proposition~\ref{prop:implications-of-Cst-cat-2} and its proof until Equation~\eqref{eq:hgpd-l-act-inpro} below before reading this remark further. In this remark, we discuss a {\Star}category studied in~\cite{Holkar-Renault2013Hypergpd-Cst-Alg} and its relation to Equations~\eqref{eq:hgpd-con},\eqref{eq:hgpd-inv} and~\eqref{eq:hgpd-l-act-inpro}. For a locally compact groupoid \(H\) equipped with a Haar system \(\beta\), one may form a {\Star}category \(\mathfrak{C}_c(H,\beta)\) as follows: the objects of \(\mathfrak{C}_c(H,\beta)\) are the vector spaces \(\Contc(X,\lambda)\) where \((X,\lambda)\) is a proper \(H\)\nb-space with an \(H\)\nb-invariant family of measures.
The arrows from \(\Contc(X,\lambda)\) to \(\Contc(Y,\mu)\), where \(\Contc(X,\lambda), \Contc(Y,\mu)\in\!\in \mathfrak{C}_c(H,\beta)^0\), are functions in \(\Contc((X\times_{s_X,\base[H],s_Y}Y)/H)\). The involution of \(f\in\Contc((X\times_{s_X,\base[H],s_Y}Y)/H)\) is \(f^*([x,y])=\overline{f([y,x])}\).

Let \(\Contc(Z,\nu)\) be one more object in \(\mathfrak{C}_c(H,\beta)\). If \(f\colon \Contc(X)\to \Contc(Y)\) and \(g\colon \Contc(Y)\to \Contc(Z)\) are arrows in \(\mathfrak{C}_c(H,\beta)\), then their composite arrow, which goes from \((X,\lambda)\) to \((Z,\nu)\), is denoted by \(f*g\) and defined as
\[f*g[x,z]=\int_Y f[x,y]g[y,z]\,\dd\mu_{s_z(z)}(\psi).\]
When equipped with the involution in Equation~\eqref{eq:hgpd-inv}, \(\mathfrak{C}_c(H)\) becomes a {\Star}category. See~\cite{Holkar-Renault2013Hypergpd-Cst-Alg} for details.

\cite{Holkar-Renault2013Hypergpd-Cst-Alg}*{Theorem 2.2} shows that that the representations of \((H,\beta)\) induce the ones of \(\mathfrak{C}_c(H)\) which makes \(\mathfrak{C}_c(H)\) a normed {\Star}category. Using these induced representations, we complete \(\mathfrak{C}_c(H)\) to a \(\Cst\)\nb-category \(\mathfrak{C}(H,\beta)\).

Let \((X,\lambda)\in\!\in\mathfrak{C}_c(H,\beta)\). Note that \(\Contc(H,\beta\inverse)\) is also an object in \(\mathfrak{C}_c(H,\beta)\). Now it can be checked that Equation~\eqref{eq:hgpd-con} gives the composition of two arrows from \(\Contc(X,\lambda)\) to itself and Equation~\eqref{eq:hgpd-inv} is an involution of an arrow from \(\Contc(X,\lambda)\) to itself. Furthermore,~\cite{mythesis}*{Table 3.2} shows that Equation~\eqref{eq:hgpd-l-act-inpro} below is equivalent to the composition of certain arrows in \(\mathfrak{C}_c(H)\). This explains the origins and well-definedness of Equations~\eqref{eq:hgpd-con},\eqref{eq:hgpd-inv} and~\eqref{eq:hgpd-l-act-inpro}.
\end{remark}
\begin{proposition}
 \label{prop:implications-of-Cst-cat-2} 
Let \((X,\lambda)\) be a proper \(H\)\nb-space equipped with an \(H\)\nb-invariant family of measures. Let \(\Cst(G)\) be the \(\Cst\)\nb-algebra obtained by completing the {\Star}algebra \(C_c(G)\) as in~\cite{Holkar-Renault2013Hypergpd-Cst-Alg}. Then $\Cst(G)\iso\Comp(\Hilm(X))$.
\end{proposition}
   \begin{proof}
We show that \(\Cst(G)\) and \(\Cst(H,\beta)\) are Morita equivalent and \(\Hilm(X)\) is the imprimitivity bimodule; this implies that \(\Cst(G)\iso\Comp(\Hilm(X))\). The isomorphism \(\Cst(G)\iso\Comp(\Hilm(X))\) is given by the representation of \(\Cst(G)\) on \(\Hilm(X)\).

 For \(f\in \Contc(G)\), \(\xi, \zeta \in \Contc(X)\) and \(g\in \Contc(H)\) let \(\zeta a\) and \(\inpro{\zeta}{\xi}\) be as in Equation~\ref{def:r-act-inerpro}. Furthermore, define
\begin{align}
 \left\{\begin{aligned}\label{eq:hgpd-l-act-inpro}
  f\zeta(x)=\int_X f[x,y]\,\zeta(y)\,\dd\lambda_{s_X(x)}(y) \;\text{ and}\\
_G\inpro{\zeta}{\xi}=\int_H \zeta(x\eta)\,\overline{\xi(y\eta)}\,\dd\beta_{s_X(x)}(\eta).
  \end{aligned}\right.
\end{align}
The integrals in Equation~\ref{eq:hgpd-l-act-inpro} are well-defined, in addition to which, \(f\zeta \in \Contc(X)\) and \({_G\inpro{\zeta}{\xi}}\in \Contc(G)\), see Remark~\ref{rem:cst-cat} above. Equation~\eqref{eq:hgpd-l-act-inpro} defines a representation of the {\Star}algebra \(\Contc(G)\) on \(\Contc(X)\) and a \(\Contc(G)\)\nb-valued inner produce on \(\Contc(X)\), respectively. Completing the {\Star}category \(\mathfrak{C}_c(H,\beta)\) to the \(\Cst\)\nb-category, \(\mathfrak{C}(H,\beta)\), we get a representation of \(\Cst(G)\) on \(\Hilm(X)\) and the \(\Contc(G)\)\nb-valued inner product makes \(\Hilm(X)\) into a left Hilbert \(\Cst(G)\)\nb-module.

It is a standard computation to check that the actions of \(\Contc(G)\) and \(\Contc(H)\) on \(\Contc(X)\) commute in the usual sense, that is \((f\xi)g=f(\xi g)\) for all \(f\in\Contc(G), \xi\in\Contc(X)\) and \(g\in\Contc(H)\). Thus \(\Hilm(X)\) is a \(\Cst(G)\)\nb-\(\ Cst(H,\beta)\)-bimodule which carries \(\Cst(G)\) and \(\Cst(H,\beta)\)\nb-valued inner products \(\inpro{}{}\) and \({_G\inpro{}{}}\), respectively.

Since \(\lambda\) is a proper family of measures along \(s_X\), the set \(\{\inpro{\zeta}{\xi}|\zeta,\xi\in \Contc(X)\}\) is dense in \(\Contc(H,\beta)\), see~\cite{Renault1985Representations-of-crossed-product-of-gpd-Cst-Alg}*{Lemme 1.1}.

 The set \(\{ _G\inpro{\zeta}{\xi}:\zeta,\xi\in\Contc(X)\}\), that is, the image of \(\Contc(X)\tens_{\C}\Contc(X)\) in \(\Contc(G)\) under the map \((\zeta,\xi)\mapsto\, _G\inpro{\zeta}{\xi}\) is dense: this follows from the a standard argument that uses the family of measures in Equation~\eqref{eq:measures-along-the-quotient}. The the argument is as follows: The restriction map mapping \(\zeta\otimes\xi\in \Contc(X)\otimes_{\C}\Contc(X)\) to \(\zeta\otimes_{s_X,\base[H],s_X}\xi\in \Contc(X\times_{s_X,\base[H],s_X}X)\) has a dense image due to the Stone-Weirstra{\ss} theorem. The Haar system \(\beta\) induces the family of measures \(\beta_{X\times_{s_X,\base[H],s_X}X}\) along the quotient map \(X\times_{s_X,\base[H],s_X}X\to G\) similar to the one in Equation~\eqref{eq:measures-along-the-quotient}. Hence the image of the integration map \(B_{X\times_{s_X,\base[H],s_X}X}\colon \Contc(X\times_{s_X,\base[H],s_X}X)\to \Contc(G)\), induced by the family of measures \(\beta_{X\times_{s_X,\base[H],s_X}X}\), is dense.

Let \(\iota\) be the bijection from \(\Contc(X)\otimes_{\C}\Contc(X)\) itself given by
\[
\iota \colon \zeta\otimes_{\C}\xi \mapsto \zeta\otimes_{\C}\overline{\xi}.
\]
 Then we observe that \(\{ _G\inpro{\zeta}{\xi}:\zeta,\xi\in\Contc(X)\}\subseteq \Contc(G)\) is the image of the composite \(B_{X\times_{s_X,\base[H],s_X}X}\circ \iota\). Hence \(\{ _G\inpro{\zeta}{\xi}:\zeta,\xi\in\Contc(X)\}\subseteq \Contc(G)\) is dense.

 To prove that \(\Hilm(X)\) implements Morita equivalence between the two \(\Cst\)\nb-algebras, we need to check that the equality \( {_G\inpro{\zeta}{\xi}\cdot\theta}= \zeta\cdot\inpro{\xi}{\theta}\) holds for every \(\zeta,\xi\) and \(\theta\) in \(\Hilm(X)\). To check this let \(\zeta,\xi,\theta\in\Contc(X)\), then for any \(x\in X\)
\begin{align}
  {_G\inpro{\zeta}{\xi}\cdot\theta}\,(x)
&=\int_X {_G\inpro{\zeta}{\xi}}[x,y]\,\theta(y)\;\dd\lambda_{s_X(x)}(y)\notag\\
&=\int_X\int_G \zeta(x\eta)\,\overline{\xi(y\eta)}\,\theta(y)\;\dd\beta^{s_X(x)}(\eta)\,\dd\lambda_{s_X(x)}(y),\label{eq:implications-of-Cst-cat-2-1} && \text{ and}
\end{align}
\begin{align*}
  \zeta\cdot\inpro{\xi}{\theta}\,(x)
&=\int_G \zeta(x\eta)\,\inpro{\xi}{\theta}(\eta\inverse)\;\dd\beta^{s_X(x)}(\eta)\\
&=\int_G\int_X \zeta(x\eta)\,\overline{\xi(y)}\,\theta(y\eta\inverse)\;\dd\lambda_{s_X(x)}(y)\,\dd\beta^{s_X(x)}(\eta).
\end{align*}
Change the variable \(y\mapsto y\eta\) in the last term above and use the right invariance of the family of measures \(\lambda\) to see that the equation transforms to
\begin{equation}
  \label{eq:implications-of-Cst-cat-2-2}
  \zeta\cdot\inpro{\xi}{\theta}\,(x)=  \int_G\int_X \zeta(x\eta)\,\overline{\xi(y\eta)}\,\theta(y)\;\dd\lambda_{s_X(x)}(y)\,\dd\beta^{s_X(x)}(\eta).
\end{equation}
Now use Fubini's theorem in the above equation to interchange the integrals to see that Equations~\eqref{eq:implications-of-Cst-cat-2-1} and~\eqref{eq:implications-of-Cst-cat-2-2} are same. Thus 
\[  {_G\inpro{\zeta}{\xi}\cdot\theta}= \zeta\cdot\inpro{\xi}{\theta}\]
for all \(\zeta,\xi\) and \(\theta\in\Contc(X)\). One may see that the \(\Cst\)\nb-category in~\cite{Holkar2017Construction-of-Corr} allows to extend above equality to \(\Cst(G), \Hilm(X)\) and \(\Cst(H,\beta)\) which proves that \(\Hilm(X)\) is the imprimitivity bimodule between \(\Cst(G)\) and \(\Cst(H,\beta)\).
   \end{proof}
   \begin{notation}
     \label{not:Hgpd-of-comp-operators}
Let \(X,H,\beta\) and \(\lambda\) as in Proposition~\ref{prop:implications-of-Cst-cat-2}. Then we denote the hypergroupoid \(G=(X\times_{s_X,\base[H],s_X}X)/H\) in the proof of the proposition by \(\TComp(X)\); the suffix `T' stands for \emph{topological}. Hence we change \(\sigma_{G}\) and \( _{G}\inpro{}{}\) change to \(\sigma_{\TComp(X)}\) and \( _{\TComp(X)}\inpro{}{}\), respectively, in the further writing.
   \end{notation}

\section{Proper topological correspondences}
\label{sec:prop-corr}

In this section, we define a proper topological correspondence, see Definition~\ref{def:prop-corr}. Section~\ref{sec:some-examples}  discusses some details that may proof helpful to understand the definition of a proper topological correspondence. Without these details, the definition may look artificial or technical. We hope that Example~\ref{exa:concrete-exa-of-G-X-H} in Section~\ref{sec:some-examples} should setup the background for Definition~\ref{def:prop-corr}.

\subsection{Some examples}
\label{sec:some-examples}

 Example~\ref{exa:transf-gpd} revises some well-known facts about transformation groupoids; Example~\ref{exa:equi-rel-gpd} discusses the groupoid associated with a continuous equivalence relation and Example~\ref{exa:concrete-exa-of-G-X-H} puts these two examples together to elaborate the background in the definition of a proper correspondence.

\begin{example}[Transformation groupoid]
  \label{exa:transf-gpd}
  Let \(G\) be a locally compact groupoid and \(X\) a left \(G\)\nb-space. For \((\gamma, x)\in G\ltimes X\), \(r_{G\ltimes X}(\gamma,x)=(r_{G}(\gamma), \gamma x)\) and \(s_{G\ltimes X}(\gamma,x)=(s_G(\gamma),x)\). We often identify the space of units \(\base\times_{\Id_{\base},\base,r_X}X\) of this transformation groupoid with \(X\) in which case, the range and source of the above element are \(r_{G\ltimes X}(\gamma,x)=\gamma x\) and \(s_{G\ltimes X}(\gamma,x)=x\). Given a unit \((r_X(x),x)\in \base[G\ltimes X]\), \(r_{G\ltimes X}\inverse(r_X(x),x)=\{(\gamma, \gamma\inverse x): r_G(\gamma)=r_X(x)\}\).
  It is well known that a Haar system on \(G\) induces one on the transformation groupoid \(G\ltimes X\); if \(\alpha\) is a Haar system on \(G\), then the Haar system on \(G\ltimes X\), which we denote by \({\alpha_2}\), is given by
\[
\int_{G\ltimes X} f\,\dd{\alpha_2}^{x}\defeq \int_G f(\gamma,\gamma\inverse x)\,\dd\alpha^{r_X(x)}(\gamma)
\]
for \(f\in \Contc(G\ltimes X)\) and \(x\in X\homeo\base[(G\ltimes X)]\). 

 Let \(H\) be another groupoid. Assume that \(H\) acts on \(X\) from right. Then \((\gamma,x)\eta\mapsto (\gamma,x\eta)\), for \((\gamma,x)\in G\ltimes X\) and \((x,\eta)\in X\times_{\base[H]}H\), is a right action of groupoid \(H\) on groupoid \(G\ltimes X\) by invertible functors, see~\cite{Holkar2017Construction-of-Corr}*{Definition 1.8 and Example 1.11}.
\end{example}

\begin{example}[The groupoid of an equivalence relation]
\label{exa:equi-rel-gpd}
  Let \(X\) and \(Y\) be spaces, and \(\phi\colon X\to Y\) a continuous map. Then the fibre product  \(X\times_{\phi,Y,\phi}X\defeq\{(x,x'): \phi(x)=\phi(x')\}\) carries a groupoid structure: \((x,x')\) and \((x'',x''')\)
  in \(X\times_{\phi,Y,\phi}X\) are composable if and only if \(x'=x''\) and the composite \((x,x')(x', x''')=(x,x''')\). The space of units of \(X_\phi\) is the diagonal \(\{(x,x)\in X\times_{\phi,Y,\phi}X:x\in X\}\) in \(X\times_{\phi,Y,\phi}X\) which can be identified with \(X\). For the element \((x,x')\) as above,  its
  \begin{enumerate*}[(i)]
  \item source is \(x'\),
  \item range is \(x\) and
  \item inverse \((x,x')\inverse= (x',x)\).
  \end{enumerate*}
 We denote this groupoid by \(X_\phi\); this is the groupoid of the equivalence relation \(x\sim x'\) if and only if \(\phi(x)=\phi(x')\). 

In addition to the above data, let \(\lambda=\{\lambda_y\}_{y\in Y}\) be a continuous family of measures with full support along \(\phi\). Then \(\phi\) induces a Haar system \({\lambda_1}\) on \(X_\phi\); for \(f\in \Contc(X_\phi)\) and \(z\in X\iso \base[X_\phi]\),
\[
  \int_{X_\phi} f\,\dd{\lambda_1}^z \defeq \int_X
  f(z,x)\,\dd\lambda_{\phi(z)}(x).
\] 
Note that \(\lambda_1^z=\delta_z\otimes\lambda_{\phi(z)}\) where \(\delta_x\) is the point mass at \(x\in X\). The family of measures \({\lambda_1}\) is a continuous family of measures due to~\cite{Renault1985Representations-of-crossed-product-of-gpd-Cst-Alg}*{Lemme 1.2} and has full support because \(\lambda\) has full support. It is easy to see that \({\lambda_1}\) is left invariant: for \(f\in \Contc(X_\phi)\) and \((v,w)\in X_\phi\)
  \begin{multline*}
    \int_{X_\phi} f((v,w)(w,x))\,\dd{\lambda_1}^{r_{X_\phi}(v,w)}(w,x)= \int_{X_\phi} f((v,w)(w,x))\,\dd{\lambda_1}^{w}(w,x)\\
=\int_{X} f(v,x)\,\dd\lambda_{\phi(w)}(x)=\int_{X} f(v,x)\,\dd\lambda_{\phi(v)}(x)= \int_{X_\phi} f(v,x)\,\dd{\lambda_1}^{r_{X_\phi}(v,w)}(v,x).
  \end{multline*}
  In the above computation, the third equality due to the fact that \(\phi(v)=\phi(w)\). If \(\lambda\) does not have full support, then \({\lambda_1}\) is a continuous family of left invariant measures and not a Haar system.

  Additionally, if we assume that \(X\) and \(Y\) are right \(H\)\nb-spaces for a groupoid \(H\), and \(\phi\) is an \(H\)\nb-equivariant map, then \((x,y)\eta\mapsto (x\eta,y\eta)\), for \((x,y)\in X\times_{\phi,Y,\phi}X\) and \(\eta\in X\times_{\base[H]}H\), is an action of \(H\) on the groupoid \(X_\phi\) in the sense of~\cite{Holkar2017Construction-of-Corr}*{Definition 1.8}. The momentum map \(r_{X_\phi,H}\colon X_\phi\to \base[H]\) of this action is \(r_{X_\phi,H}(x,y)=s_X(x)\).

If \(\lambda\) is an \(H\)\nb-invariant family of measures, then a routine computation shows that the family of measures \(\lambda_1\) is also \(H\)\nb-invariant.
The example ends here.
\end{example}

We need the following remark: it is well-known that the right action \(r_H(\eta)\eta=s_H(\eta)\) of a groupoid \(H\) on its space of units \(\base[H]\) is well-defined continuous action. Moreover, for any right \(H\)\nb-space \(X\), the momentum map \(s_X\colon X\to \base[H]\) is \(H\)\nb-equivariant for this action.

\begin{example}
  \label{exa:concrete-exa-of-G-X-H}
Let \(G\) and \(H\) be locally compact groupoids, and \(X\) a \(G\)\nb-\(H\)-bispace. Let \(\alpha\) be a Haar system on \(G\) and \(\lambda\) an \(H\)\nb-invariant continuous family of measures on \(X\) along \(s_X\).
  \begin{enumerate}[leftmargin=*]
  \item Then \(s_X\colon X\to \base[H]\) is a \(G\)\nb-invariant map. Let \(X_{s_X}\) be the groupoid associated with the continuous map \(s_X\colon X\to \base[H]\) as in Example~\ref{exa:equi-rel-gpd}. It is a routine calculation to check that the map
 \[
\Mlp_{s_X}\colon G\ltimes X \to X_{s_X},\quad (\gamma, x)\mapsto (\gamma x,x),
 \]
 is a homomorphism of groupoids.  Furthermore, since \(X\) is a \(G\)\nb-\(H\)-bispace, both groupoids, \(G\ltimes X\) and \(X_{s_X}\), carry actions of \(H\) described in Examples~\ref{exa:equi-rel-gpd} and~\ref{exa:transf-gpd} and the map \(\Mlp_{s_X}\) is an \(H\)\nb-equivariant homomorphisms of groupoids. Note that \(\Mlp_{s_X}\) is the restriction of \(r_{G\ltimes X}\times s_{G\ltimes X}\colon G\ltimes X\to X\times X\)  to the closed subspace \(X\times_{s_X,\base[H],s_X}X\subseteq X\times X\).

 Identify \(X\) with the space of units of \(G\ltimes X\) as well as \(X_{s_X}\). Then
 \begin{enumerate}
 \item \(\Mlp_{s_X}|_{\base[G\ltimes X]}=\Id_X\colon X\to X\),
 \item \(r_{G\ltimes X}=r_{X_{s_X}}\circ \Mlp_{s_X}\).
\end{enumerate}
The last equality above implies that \(\Mlp_{s_X}\inverse\circ  r_{X_{s_X}}\inverse (x)=r_{G\ltimes X}\inverse(x)\) for all \(x\in X\).
   \item The Haar system \(\alpha\) on \(G\) induces the Haar system \({\alpha_2}\) on \(G\ltimes X\) (see Example~\ref{exa:transf-gpd}). As in Example~\ref{exa:equi-rel-gpd}, \(\lambda\) induces a continuous invariant family of measures \({\lambda_1}\) on the groupoid \(X_{s_X}\).  Since \(\lambda\) is \(H\)\nb-invariant, so is \(\lambda_1\). If \(\lambda\) has full support, then \({\lambda_1}\) is a Haar system for \(X_{s_X}\).
   \item In addition to the current hypotheses, assume that the action of \(G\) on \(X\) is locally strongly locally free and \(\Mlp_{s_X}\) is open onto its image, that is, \(\Mlp_{s_X}\) is a local homeomorphism onto its image (see Lemma~\ref{lem:loc-str-loc-free-act}). Let \(\RN\colon X\times_{s_X,\base[H],s_X}X\to \R\) be a nonnegative \(H\)\nb-invariant 
     continuous function. Then, for each \(x\in X\homeo\base[X]_{s_X}\), \(\RN(x,\_){\lambda_1}^x\) is an absolutely continuous measure with respect to \({\lambda}_1^x\) on \(X_{s_X}^x\). Furthermore, the Radon-Nikodym derivative \(\frac{\dd\RN(x,\_){\lambda}_1^x}{\dd{\lambda}_1^x}(x,y)=\RN(x,y)\). Let \(\RN\lambda_1\defeq \{\RN(x,\_)\lambda_1^x\}_{x\in X}\). We abuse the notation: for each \(x\in X\homeo\base[X]_{s_X}\), let \(\RN(x,\_){\lambda}_1^x\) denote the measure on \(X_{s_X}\) which equals  \(\RN(x,\_){\lambda}_1^x\) on the fibre \(X_{s_X}^x\) and zero outside it. Now Corollary~\ref{cor:pulling-and-pasting-measures-2} says that for each \(x\in X\homeo\base[X]_{s_X}\), \(\RN(x,\_){\lambda_1}^x\) induces a measure \(\Mlp_{s_X}^*(\RN(x,\_)\lambda_1^x)\) on \(G\ltimes X\), to be precise, on \(r_{G\ltimes X}\inverse(x)\subseteq G\ltimes X\).

     Since \(\RN\) is  \(H\)\nb-equivariant,  so is the family of measures \(\RN\lambda_1\). Recall from (1) above that \(\Mlp_{s_X}\) is an \(H\)\nb-equivariant map. The \(H\)\nb-equivariance of \(\Mlp_{s_X}\) and  \(\RN\lambda_1\) implies that the family of measures
     \[
       \Mlp_{s_X}^*(\RN\lambda_1)\defeq\{\Mlp_{s_X}^*(\RN(x,\_){\lambda}_1^x)\}_{x\in    X}
     \]
    is also \(H\)\nb-invariant--- this claim follows from a direct computation.
   Lemma~\ref{lem:loc-str-loc-free-act} describes this family of measures: for \(U\in \Mlp_{s_X}^{\HoC}\) and \(f\in \Contc(G\ltimes X,U)\),
    \begin{equation}
      \label{eq:RN-lambda-induced-measure}
    \Mlp^*_{s_X}(\RN(x,\_){\lambda_1^x})(f)= (\RN(x,\_){\lambda_1^x})(f')=\int_X f'(x,y) \RN(x,y)\,\dd\lambda_{s_X(x)}(y)  
    \end{equation}
     for any extension \(f'\in \Ext_{\Mlp_{s_X}}(f)\). 
   \item 
     In Definition~\ref{def:prop-corr}, one of the conditions demands that
     \[
       (\Mlp_{s_X}^* (\RN(x,\_){\lambda_1^x}) = {\alpha_2^x}
     \]
     for each \(x\in X\).
     
Recall from (3) above that \(\RN(x,\_){\lambda_1^x}\ll {\lambda_1^x}\) and the Radon-Nikodym derivative \(\frac{\dd(\RN(x,\_){\lambda_1^x})}{\dd{\lambda_1^x}}(x,\_)=\RN(x,\_)\). Lemma~\ref{lem:RN-positive-makes-pullback-measure-equi} says that \({\alpha_2}_x=\Mlp_{s_X}^*(\RN(x,\_){\lambda_1^x})\ll \Mlp_{s_X}^*({\lambda_1^x})\), and, the Radon-Nikodym derivative 
\[
\frac{\dd{\alpha_2^x}}{\dd \Mlp_{s_X}^*({\lambda_1^x})}=\RN\circ\Mlp_{s_X}.
\]
 
  \item Note that, for  \(U\in \Mlp_{s_X}^{\Ho}\) and \(f\in \Contc(G\ltimes X,U)\), we may write
    \[
    f\circ \Mlp_{s_X}|_U\inverse(x,y)= f(\gamma_{(x,y)}\inverse, y)
    \]
using Equation~\ref{eq:loc-sec-inv-y} or
    \[
    f\circ \Mlp_{s_X}|_U\inverse(x,y)=
    f(\gamma_{(x,y)},\gamma_{(x,y)}\inverse x)
    \]
    using Equation~\ref{eq:loc-sec-inv-x}. We use the second form of \(f\circ \Mlp_{s_X}|_U\inverse\) in Definition~\ref{def:prop-corr}.
  \end{enumerate}
\end{example}

\subsection{Proper correspondences}
\label{sec:prop-corr-1}

As advised in the Introduction, reader may eliminate the functions \(\RN\) and the adjoinig function \(\Delta\) in this section during the first reading--- one may consider \(\RN\) and \(\Delta\) are the constant function \(1\), and skip the remarks discussing these functions.
\begin{definition}
  \label{def:prop-corr}
Let \((G,\alpha)\) and \((H,\beta)\) be locally compact groupoids with Haar systems, and \((X,\lambda)\) a topological correspondence from \((G,\alpha)\) to \((H,\beta)\). Let \(X\) be Hausdorff and \(X/H\) paracompact. We call \((X,\lambda)\) \emph{proper} if there is an \(H\)\nb-invariant nonnegative continuous function \(\RN\colon X\times_{s_X,\base[H],s_X}X\to \R\) and following holds:
\begin{enumerate}[i)]
\item the left action of \(G\) on \(X\) is locally strongly locally free, and the map \(\Mlp_{s_X}\) is open onto its image,
\item the map \([r_X]\colon X/H\to \base\) induced by the left momentum map is proper,
\item there is an open cover \(\mathcal{A}\) of \(G\times_{\base} X\) consisting of sets in \(\Mlp_{s_X}^{\HoC}\) and there is an extension \(\Cov'\) of \(\Cov\) via \(\Mlp_{s_X}\) with the property that for all \(U\in \mathcal{A}\), any extension \(U'\in \Cov'\) of it via \(\Mlp_{s_X}\), given \(f\in \Contc(G\ltimes X, U)\), and given any extension \(f'\in \Contc(X\times_{s_X,\base[H],s_X}, U')\)  of \(f\), the following equality holds
 \begin{equation}
    \label{eq:prop-corr}
    \int_{X} f' (x,y)\;\RN(x,y)\;\dd\lambda_{s_X(x)}(y)= \int_{G} f(\gamma_{(x,y)},\gamma_{(x,y)}\inverse x)\,\dd\alpha^{r_X(x)}(\gamma)
 \end{equation}
for every \(x\in X\).
\end{enumerate}
\end{definition}

We make few remarks. In the following remarks, we denote restriction of a measure \(\mu\) to a set \(U\) by \(\mu|_U\) instead of \(\mu_U\). We adopt this convention only for the following discussion.

Lemma~\ref{lem:loc-str-loc-free-act} implies that the first condition in Definition~\ref{def:prop-corr} is equivalent to \(\Mlp_{s_X}\) is a local homeomorphism. Thus (iii) in Definition~\ref{def:prop-corr} makes sense, see Example~\ref{exa:concrete-exa-of-G-X-H} for details. Equation~\eqref{eq:RN-lambda-induced-measure} in that example says that the left side of Equation~\eqref{eq:prop-corr} is \(\Mlp^*_{s_X}(\RN(x,\_){\lambda_1^x})(f)\), and thus Equation~\eqref{eq:prop-corr} says that
\begin{equation}
   \label{equ:meaning-eq-prop-corr}
\Mlp^*_{s_X}(\RN(x,\_){\lambda_1^x})(f)(x)=\alpha_2^x(f)(x)
\end{equation}
for all \(f\in \Contc(G\ltimes X,U)\) where \(U\in \mathcal{A}\subseteq {s_X}^{\HoC}\) and \(x\in X\homeo\base[G\ltimes X]\homeo \base[X]_{s_X}\). Now Equation~\eqref{equ:meaning-eq-prop-corr} along with Lemma~\ref{lem:pulling-and-pasting-measures-3} tell us that for \(U'\in \ExtS_{s_X}(U)\cap\Cov'\), the measure \({\RN(x,\_){\lambda_1^x}}|_{U'}\) is concentrated on \(\Mlp_{s_X}(U)\); Lemma~\ref{lem:closed-image-of-Mlp-measure-conc} discusses when is the  measure  \(\RN(x,\_){\lambda_1^x}\) is concentrated on \(\Mlp_{s_X}(G\times_{\base} \{x\})\).
Furthermore, as discussed in Example~\ref{exa:concrete-exa-of-G-X-H}, the measure \(\alpha_2^x=\Mlp^*_{s_X}(\RN(x,\_){\lambda_1^x})\ll\Mlp^*_{s_X}({\lambda_1^x})\) and the Radon-Nikodym derivative
\[
\frac{\dd\Mlp^*_{s_X}({\lambda_1^x})}{\dd{\alpha_2^x}} = \RN(x,\_)\circ \Mlp_{s_X}\circ \textup{inv}_{G\ltimes X},
\]
see the composite map in~{(4)} of Example~\ref{exa:concrete-exa-of-G-X-H}.
\smallskip

In Definition~\ref{def:prop-corr}, if the open cover \(\{\Mlp_{s_X}(U)\}_{U\in \Mlp_{s_X}^{\HoC}}\) of \(\Mlp_{s_X}(G\ltimes X)\) admits an extension to \(X\times_{s_X,\base[H],s_X}X\), then, due to Lemma~\ref{lem:pulling-and-pasting-measures-3}, we may conclude that the measure \(\RN(x,\_){\lambda_1^x}\) is concentrated on the image of \(\Mlp_{s_X}\).

\begin{lemma}
  \label{lem:closed-image-of-Mlp-measure-conc}
In Definition~\ref{def:prop-corr}, if the image of \(\Mlp_{s_X}\) is a closed subset of \(X\times_{s_X,\base[H],s_X}X\), then the measure \(\RN(x,\_){\lambda_1^x}\) is concentrated on the image of \(\Mlp_{s_X}\).
\end{lemma}
\begin{proof}
 This follows from  Example~\ref{exa:ext-of-open-cov-closed-sb} and the remark immediately above this lemma.
\end{proof}
A particular instance when the hypothesis of Lemma~\ref{lem:closed-image-of-Mlp-measure-conc} is fulfilled is when \(\Mlp_{s_X}\) is a proper map. We encounter examples of this type when \(\Mlp_{s_X}\) is a homeomorphism or homeomorphism onto a closed subspace of \(X\times_{s_X,\base[H],s_X}X\).

 \begin{remark}
   \label{rem:gp-case-cond-iii}
 Let \((X,\lambda)\) be a topological correspondence from a groupoid with a Haar system \((G,\alpha)\) to another one \((H,\beta)\). If \(X\) is \(H\)\nb-compact, that is \(X/H\) is compact, then the map \([r_X]:X/H\to\base\) induced by the right momentum map \(r_X\) is proper. On the other hand, assume that the space of the units of the groupoid \(G\) is compact, for example, \(G\) is a group. If \([r_X]\) is a proper map, then \(X/H=[r_X]\inverse(\base)\) is a compact space.  Thus when \(\base\) is a compact space, Condition~{(ii)} in Definition~\ref{def:prop-corr} is equivalent to that the quotient space \(X/H\) is compact.
 \end{remark}
\medskip

Now we prepare to state the main theorem. Let \((G,\alpha)\) and \((H,\beta)\) be locally compact groupoids with Haar systems, and \((X,\lambda)\) a topological correspondence from \((G,\alpha)\) to \((H,\beta)\) with \(\Delta\) as the adjoining function. This gives a \(\Cst\)\nb-correspondence \(\Hilm(X)\colon \Cst(G,\alpha)\to \Cst(H,\beta)\) which is a certain completion of \(\Contc(X)\).  From~\cite{Holkar2017Construction-of-Corr} , recall the formulae for the actions of \(\Contc(G)\) and \(\Contc(H)\) on \(\Contc(X)\), and the \(\Contc(H)\)\nb-valued inner product on \(\Contc(X)\) which gives this \(\Cst\)\nb-correspondence: for \(b \in C_c(G)\), \(\zeta,\xi \in C_c(X)\) and \(a \in C_c(H)\)
\begin{equation}\label{def:lr-act-inerpro}
 \left\{\begin{aligned}
(b\cdot \zeta)(x) &\defeq \int_{G^{r_X(x)}} b(\gamma)\, \zeta(\gamma\inverse x) \,
\Delta^{1/2}(\gamma, \gamma\inverse x) \; \dd \alpha^{r_X(x)}(\gamma),\\
(\zeta\cdot a )(x) &\defeq \int_{H^{s_X(x)}} \zeta(x\eta)\, a({\eta}\inverse) \; \dd\beta^{s_X(x)}(\eta) \text{ and }\\
\langle \zeta, \xi \rangle (\eta) &\defeq \int_{X_{r_H(\eta)}}\, \overline{\zeta(x)} \xi(x\eta) \; \dd\lambda_{r_H(\eta)}(x).
  \end{aligned}\right.
\end{equation}
  Denote the {\Star}representation of the pre-\(\Cst\)\nb-algebra \(\Contc(G)\), defined in Equation~\eqref{def:lr-act-inerpro}, on the dense subspace \(\Contc(X)\) of the Hilbert \(\Cst(H,\beta)\)\nb-module \(\Hilm(X)\) by \(\sigma_G\), that is, \(\sigma_G(b)(\zeta)=b\cdot \zeta\). Similarly, let \(\sigma_H\) denote the representation of the pre-\(\Cst\)\nb-algebra \(\Contc(H)\) on the pre-Hilbert \(\Cst(H,\beta)\)\nb-module \(\Contc(X)\). 
\begin{theorem}
  \label{thm:prop-corr-main-thm}
Let \((G,\alpha)\) and \((H,\beta)\) be locally compact groupoids with Haar systems. If \((X,\lambda)\) is a proper topological correspondence from \((G,\alpha)\) to \((H,\beta)\), then the \(\Cst\)\nb-correspondence \(\Hilm(X)\colon \Cst(G,\alpha)\to\Cst(H,\beta)\) is proper.
\end{theorem}
The rest of this section is devoted to prove Theorem~\ref{thm:prop-corr-main-thm}. The proof is broken into three steps:
\begin{enumerate*}[(i)]
\item first we show that if \(h\in \Contc(G)\) and
  \(\xi\in\Contc(X)\), then the restriction \((h\otimes\xi)|_{G\times_{\base}X}\in\Contc(G\times_{\base}X)\); this is a well know fact that we revise. 
\item For  given \(b\in\Contc(G)\), we find suitable functions
  \(\epsilon_i\in\Contc(G)\) and \(\delta_j\in \Contc(X)\), where
  \(1\leq i\leq m\) and \(1\leq j\leq n\), and \(m\) and \(n\) are
  natural numbers, such that
  \(f_{ij}\defeq b\epsilon_i\otimes \frac{c}{\bar c}\,\delta_j\) is in
  \(\Contc(G\times_{\base}X,D_{ij})\) for some
  \(D_{ij}\in \Cov\), here \(c\) and \(\bar c\) are fixed
  pre-cutoff and cutoff functions. Recall from Definition~\ref{def:prop-corr} that \(\Cov\) is a fixed cover of \(G\ltimes X\) by sets in \(\Mlp_{s_X}^{\HoC}\).
\item Finally, let \(F_{ij}'\in\Contc(X\times_{s_X,\base[H],s_X}X)\) be 
 an extension of \(f_{ij}\) via \(\Mlp_{s_X}\) which is supported in \(\Cov'\). We show that the operator \(\sigma_G(b)\) on \(\Hilm(X)\) is same as the operator \(\sigma_{\TComp(X)}\left(\sum_{i=1,j=1}^{i=n,j=m}   B_{X\times_{s_X,\base[H],s_X}X}(F_{ij}')\right)\). See the proof of Proposition~\ref{prop:implications-of-Cst-cat-2} for the meaning of \(B_{X\times_{s_X,\base[H],s_X}X}(F_{ij}')\).
\end{enumerate*}
We start by proving follwing lemmas.
\begin{lemma}
  \label{lem:top-1}
  Let \(X\) be a topological space and \(L, K\subseteq X\) subspaces. If \(L\) is closed in \(X\) and \(K\) is (quasi-)compact in \(X\), then \(L\cap K\) is (quasi-)compact in \(L\) as well as in \(X\).
\end{lemma}
\begin{proof}
  Let \(\{U_i\}_{i\in I}\) be an open cover of \(K\cap L\) in \(L\) where \(I\) is an index set. For each \(i\in I\) let \(U'_i\subseteq X\) be an open set such that \(U_i=U_i'\cap K\). Then \(\{U_i'\}_{i\in I}\cup \{X-L\}\) is an open cover of \(K\) in \(X\). Since \(K\subseteq X\) is (quasi-)compact, choose and open subcover of this cover which covers \(K\); thus there is a finite natural number \(n\) such that \(U_{i_1}',\dots,U_{i_n}'\) and \(X-L\) cover \(K\).  Now one may see that  \(U_{i_1},\dots,U_{i_n}\), is an open cover \(K\cap L\).

Moreover, since \(L\cap K\) is (quasi-)compact in \(L\), it is (quasi-)compact in any subspace of \(X\) containing \(L\), in particular, \(X\) itself.
\end{proof} 

\begin{lemma}
  \label{lem:main-thm-1}
Let \(X\) be a locally compact left \(G\)\nb-space for a locally compact groupoid \(G\). Let \(K\subseteq G\) and \(Q\subseteq X\) be nonempty subsets with \(s_G(K)\cap r_X(Q)\neq \emptyset\). If \(K\) and \(Q\) are quasi-compact (or compact) in \(G\) and \(X\), respectively, then so is \(K\times_{\base}Q\subseteq G\times_{\base}X\).
\end{lemma}
\begin{proof}
Since all the sets \(K\), \(Q\) and \(s_G(K)\cap r_X(Q)\) are nonempty, so is \(K\times_{\base}Q\). 

Assume that \(K\subseteq G\) and \(Q\subseteq X\) are quasi-compact. We observe that \(G\times_{\base} X\subseteq G\times X\) is closed; this is proved in the last paragraph below. Then \(K\times_{\base} Q\defeq (K\times Q)\cap (G\times_{\base} X\)) is quasi-compact in \(G\times_{\base} X\) due to Lemma~\ref{lem:top-1}. This proves the assertion about quasi-compactness of \(K\times_{\base} Q\). Similar is the argument for the  claim regarding compactness.

A proof that \(G\times_{\base} X\subseteq G\times X\) is closed: recall from basic topology that a space \(Y\) is Hausdorff if and only if the diagonal \(\Dia(Y)\defeq \{(y,y):y\in Y\}\subseteq Y\times Y\) is closed. By hypothesis, \(\base\), the space of units of the groupoid \(G\), is Hausdorff hence \(\Dia(\base)\subseteq \base\times\base\) is closed. Now the function \(s_G\times r_X\colon G\times X\to \base\times\base\), \((\gamma, x)\mapsto (s_G(\gamma), r_X(x))\), is continuous, and \(G\times_{\base} X= (s_G\times r_X)\inverse(\Dia(\base))\). Being the inverse image of a closed set under a continuous function \(G\times_{\base} X\subseteq G\times X\) is closed.
\end{proof}

\begin{lemma}
  \label{lem:main-thm-3}
Let \(Y\) and \(X\) be locally compact spaces, and let \(K\subseteq Y\) and \(Q\subseteq X\) be quasi-compact spaces. Let \(\{W_d\}_{d\in D}\) be an open cover of \(K\times Q\) where \(D\) is an index set and each \(W_d\) is open \(Y\times X\). Then there are finite open covers \(\{U_1,\dots,U_m\}\) of \(\{K\}\), and \(\{V_i,\dots,V_n\}\) of \(\{Q\}\) such that for each \((i,j)\in\{1,\dots,m\}\times\{1,\dots,n\}\) there is an index \(d\in D\) with \(U_i\times V_j\subseteq W_d\), and 
 \(U_i\subseteq Y\) and \(V_j\subseteq X\) are open.
\end{lemma}
\begin{proof}
The proof of this lemma is exactly same as the one of Lemma~\cite{Tu2004NonHausdorff-gpd-proper-actions-and-K}*{7.15} except that one replaces all the open covers in \(K\) and \(Q\) by open covers in \(Y\) and \(X\), respectively.
\end{proof}

Now we write the proof of Theorem~\ref{thm:prop-corr-main-thm}; during the first reading of this proof, reader may assume that \(\RN\) and the adjoining function \(\Delta\) are the constant function \(1\).

\begin{proof}[Proof of Theorem~\ref{thm:prop-corr-main-thm}] 
\noindent 1) Let \((G,\alpha)\), \((H,\beta)\) and \((X,\lambda)\) be as in the statement of the theorem. Let \(\Delta\) be the adjoining function of the topological correspondence. Let \(\sigma_G\) and \(\sigma_{\TComp}\) have the meaning as the discussion following Equation~\eqref{def:lr-act-inerpro} and Notation~\ref{not:Hgpd-of-comp-operators}, respectively.

Recall from the proof of Proposition~\ref{prop:implications-of-Cst-cat-2} that the Haar system \(\beta\) of \(H\) induces a family of measures \(\beta_{X\times_{s_X,\base[H],s_X}X}\) along the quotient map \(X\times_{s_X,\base[H],s_X}X\to (X\times_{s_X,\base[H],s_X}X)/H\) similar to the one in Equation~\eqref{eq:measures-along-the-quotient}. The image of the corresponding integration map \(B_{X\times_{s_X,\base[H],s_X}X}\colon \Contc(X\times_{s_X,\base[H],s_X}X)\to \Contc((X\times_{s_X,\base[H],s_X}X)/H)\) is dense.
\medskip

\noindent 2) Let \(b\in\Contc(G)\) be a given nonzero function. In (6) of this proof, we show that there are natural numbers \(m\) and \(n\), and a function \(F_{ij}\in \Contc(X\times_{s_X,\base[H],s_X}X)\) for \(0\leq i \leq m\) and \(0\leq j\leq n\) with
\[
\sigma_G(b)\xi=\sum_{i,j=1}^{i=m,j=n}\sigma_{\TComp(X)}(B_{X\times_{s_X,\base[H],s_X}X}(F_{ij}))\xi
\]
for all \(\xi\in\Contc(X)\).

Since every function in \(\Contc(G)\) is a finite linear combination of functions in \({\Contc}_0(G)\), we may assume that \(b\in {\Contc}_0(G)\), see~\cite{Khoshkam-Skandalis2002Reg-Rep-of-Gpd-Cst-Alg-and-Applications}*{Lemma 1.3} for details. Then \(K=\supp(b)\subseteq T\) where \(T\subseteq G\) is open and Hausdorff. Consequently, \(s_G(K)\subseteq\base\) is compact. Use Condition~{(ii)} in Definition~\ref{def:prop-corr} to see that \([r_X]\inverse(s_G(K))\subseteq X/G\) is compact. Since the quotient map \(q\colon X\to X/G\) is open, we choose a compact set \(Q\subseteq X\) with \(q(Q)=[r_X]\inverse(K)\). Then \(s_G(K)=r_X(X)\neq\emptyset\). Now apply Lemma~\ref{lem:main-thm-1} to the sets \(K\subseteq G\) and \(Q\subseteq X\) to see that \(K\times_{\base}Q\subseteq G\times_{\base}X\) is compact.

\medskip

\noindent 3) Let
\(
  \tilde{\Cov}= \{\tilde{A}\subseteq G\times X: \tilde{A}\text{ is open in } G\times X, \text{ and } \tilde{A}\cap (G\times_{\base} X)\in \Cov \}\);
  this is a collection of open sets in \(G\times K\) which covers \(G\times_{\base}X\). Since \(K\times Q=\supp(b)\times Q\) and \(G\times_{\base}X\) are closed in \(G\times X\), so is \(K\times_{\base} Q\defeq (K\times Q)\cap (G\times_{\base}X)\). Therefore, \(\tilde{\Cov}\cup \{(G\times X)-(K\times_{\base} Q)\}\) is an open cover of \(K\times Q\) consisting of open sets in \(G\times X\). Apply Lemma~\ref{lem:main-thm-3} to this cover to get finitely many open sets \(U_1,\dots,U_m\) in \(G\) and  finitely many open sets \(V_1,\dots,V_m\) in \(X\) which cover \(K\) and \(Q\), respectively, and for each \((\gamma,x)\in K\times_{\base}Q\), there are indices \(i\in\{1,\dots,n\}\) and \(j\in\{1,\dots,m\}\) such that \((\gamma,x)\in U_i\times V_j\subseteq A\) for some \(A\in\tilde{\Cov}\).

  Therefore, for \(1\leq i \leq m\) and \(1\leq j\leq n\) and any functions \(f\in \Contc(G,U_i)\) and \(g\in \Contc(X, V_j)\), the restriction \(f\otimes_{\base}g\) is supported in a set in \(\Cov\).

  Let \(\{\epsilon_i\}_{i=1}^n\) be a partition of unity of \(K\) subordinate to the open cover \(\{U_i\}_{i=1}^n\) and \(\{\delta_j\}_{j=1}^m\) be the one of \(Q\) subordinate to \(\{V_j\}_{j=1}^m\).
\medskip

\noindent 4) Recall Example~\ref{exa:cutoff-for-H-space}. As in this example, let \(c: X\to \bar\R^+\) be a pre-cutoff function, that is, a function similar to the function \(F\) in Lemma~\ref{lem:cutoff-function}. Let \(\bar c\) be same as in the example; \(\bar c\) is obtained by averaging \(c\) over \(H\):
\[
\bar c(x)= \int_{H} c(x\eta) \;\dd\beta^{s_X(x)}(\eta) \quad \text{ for } u\in X.
\]
Therefore, as discussed in the example, \(\bar c\) is an \(H\)\nb-invariant function.

Fix \((i,j)\in\{1,\dots,m\}\times\{1,\dots,n\}\). Then the product \(b\epsilon_i \in \Contc(G, U_i)\). Similarly, \((c/\bar c)\,\delta_j\in \Contc(X,V_j)\) and the function
\begin{equation}
  \label{eq:def-l-ij}
l_{ij}\defeq \left(b\epsilon_i\,\otimes_{G\times_{\base}X}\; \frac{c}{\bar c}\, \delta_j\right) \;\cdot\,\Delta
\end{equation}
is in \(\Contc(G\times_{\base}X, U_i\times_{\base[H]} V_j)\), due to Lemma~\ref{lem:main-thm-1}.

Recall from (3) above that each \(l_{ij} \) is supported in a set in \(\Cov\subseteq \Mlp_{s_X}^{\HoC}\). Fix an extension \(F_{ij}\in \Ext_{\Mlp_{s_X}}(l_{ij})\) which is supported in a set in \(\Cov'\). Using the local sections in Equation~\eqref{eq:loc-sec-inv-x} we see that for \((x,y)\in\Mlp_{s_X}(U_i\times_{\base} V_j)\),
\begin{align*}
  F_{ij}(x,y)&=  l_{ij}\circ \Mlp_{s_X}|_{U_i\times_{\base}V_j}\inverse(x,y) \\
  &= b(\gamma_{(x,y)})\;\epsilon(\gamma_{(x,y)})\;\delta(\gamma_{(x,y)}\inverse\, x)\, \frac{c(\gamma_{(x,y)}\inverse\, x)}{\bar{c}(\gamma_{(x,y)}\inverse\, x)}\; \Delta(\gamma_{(x,y)},\gamma_{(x,y)}\inverse\, x).
\end{align*}
\medskip

\noindent 5)
 Let 
\(
F_{ij}'=F_{ij}\RN
\)
where \(\RN\colon X\times_{s_X,\base[H],s_X}X\to \R^*\) is the \(H\)\nb-invariant nonnegative 
continuous function associated to the proper correspondence \((X,\lambda)\). Since \(F_{ij}\in \Ext_{\Mlp_{s_X}}(l_{ij})\) is continuous function with compact support, and \(\RN\) is continuous,  \(F_{ij}'\) is also continuous and compactly supported.

Let \(\xi\in \Contc(X)\) and recall from (1) above that \(B_{X\times_{s_X,\base[H],s_X}X}\) is the family of measures along the quotient map \(X\times_{s_X,\base,s_X}X\to (X\times_{s_X,\base,s_X}X )/H\) as in Equation~\eqref{eq:measures-along-the-quotient}. Then
\begin{align*}
&\sigma_{\TComp(X)}(B_{X\times_{s_X,\base[H],s_X}X}(F_{ij}'))\xi(x)\\
=&\int_XB_{X\times_{s_X,\base[H],s_X}X}(F_{ij}')[x,y]\xi(y)\;\dd\lambda_{s_X(x)}(y)\\
=&\int_X\int_H F_{ij}'(x\eta,y\eta)\xi(y)\;\dd\beta^{s_X(x)}(\eta)\,\dd\lambda_{s_X(x)}(y) \\
 =&\int_X\int_H F_{ij}(x\eta,y\eta)\;\RN(x\eta,y\eta)\;\xi(y)\;\dd\beta^{s_X(x)}(\eta)\,\dd\lambda_{s_X(x)}(y).
\end{align*}
Now use the \(H\)\nb-invariance of \(\RN\) which is built in its definition (see Definition~\ref{def:prop-corr}) to see that the above term equals
\[
\int_X\int_H F_{ij}(x\eta,y\eta)\;\RN(x,y)\;\xi(y)\;\dd\beta^{s_X(x)}(\eta)\,\dd\lambda_{s_X(x)}(y).
\]
Using Fubini's theorem we write the above term as
\[
\int_H\left(\int_X F_{ij}(x\eta,y\eta)\;\RN(x,y)\;\dd\lambda_{s_X(x)}(y)\right)\,\dd\beta^{s_X(x)}(\eta).
\]
Recall from (4) above that \(F_{ij}\in \Ext_{\Mlp_{s_X}}(l_{ij})\) and \(F_{ij}\) is supported in an element of \(\Cov'\). Therefore, Equation~\eqref{eq:prop-corr} in Definition~\ref{def:prop-corr} allows us to change the above integral on \(X\) to the one on \(G\) as
\[
\int_H\int_G l_{ij} (\gamma_{(x\eta,y\eta)}, \gamma_{(x\eta,y\eta)}\inverse x\eta)\;\xi(y)\;\dd\alpha^{r_X(x)}(\gamma)\,\dd\beta^{s_X(x)}(\eta).
\]
Now substitute the value of \(l_{ij}\) in above equation from Equation~\eqref{eq:def-l-ij}; we write \(\gamma\) instead of \(\gamma_{(x\eta,y\eta)}\) for simplicity and then the previous term looks
\[
  \int_H\int_G  b(\gamma)\epsilon_i(\gamma)\, \delta_j(\gamma\inverse x\eta)\xi(\gamma\inverse x\eta)\frac{c(\gamma\inverse x\eta)}{\bar c(\gamma\inverse x\eta)}\;\Delta(\gamma,\gamma\inverse x\eta)
  \;\,\dd\alpha^{r_X(x)}(\gamma)\,\dd\beta^{s_X(x)}(\eta).
\]
\smallskip

\noindent 6)
Recall from Example~\ref{exa:cutoff-for-H-space} that \(\bar c\) is \(H\)\nb-invariant, and recall from~\cite{Holkar2017Construction-of-Corr} that \(\Delta\) is also \(H\)\nb-invariant. We incorporate these changes and compute:
\begin{align*}
&\left(\sum_{i,j=1}^{i=m,j=n}\sigma_{\TComp(X)}\left(B_{X\times_{s_X,\base[H],s_X}X}(F_{ij}')\right)\right)\xi(x)\\
&=\sum_{i,j=1}^{i=m,j=n}\int_H\int_G b(\gamma)\,\epsilon_i(\gamma)\,\delta_j(\gamma\inverse x \eta)\, \frac{c(\gamma\inverse x \eta)}{\bar{c}(\gamma\inverse x)}\,\xi(\gamma\inverse x )\, \Delta(\gamma,\gamma\inverse x)
  \;\,\dd\alpha^{r_X(x)}(y)\dd\beta^{s_X(x)}(\eta)\\
&=\sum_{i=1}^{i=m}\int_G b(\gamma)\,\epsilon_i(\gamma)\,\xi(\gamma\inverse x )  \Delta(\gamma,\gamma\inverse x) \left(\sum_{j=1}^n\int_H\delta_j(\gamma\inverse x \eta)\, \frac{c(\gamma\inverse x \eta)}{\bar{c}(\gamma\inverse x)}
\dd\beta^{s_X(x)}(\eta)\right)\,\dd\alpha^{r_X(x)}(y)\\
&=\sum_{i=1}^{i=m}\int_G b(\gamma)\,\epsilon_i(\gamma)\,\xi(\gamma\inverse x )  \Delta(\gamma,\gamma\inverse x)\! \left(\int_H\left(\sum_{j=1}^n\delta_j(\gamma\inverse x \eta)\right) \frac{c(\gamma\inverse x \eta)}{\bar{c}(\gamma\inverse x)}\dd\beta^{s_X(x)}(\eta)\right)\!\dd\alpha^{r_X(x)}(y).
\end{align*}
We use Fubini's theorem during the second step above. Now use the fact that \(\{\delta_j\}_{j=1}^{n}\) is a partition of unity subordinate to the cover \(\{V_j\}_{j=1}^n\) of \(Q\). Then the last term equals
\[
  \sum_{i=1}^{i=m}\int_G b(\gamma)\,\epsilon_i(\gamma)\,\xi(\gamma\inverse x ) \Delta(\gamma,\gamma\inverse x) \left(\int_H \frac{c(\gamma\inverse x \eta)}{\bar{c}(\gamma\inverse x)}     \dd\beta^{s_X(x)}(\eta)\right)\;\,\dd\alpha^{r_X(x)}(y).
\]
Note that the integrant in the second integral, \(\dd\beta^{s_X(x)}\), is the cutoff function \((c/\bar c )\circ s_{X\rtimes H}\), see Example~\ref{exa:cutoff-for-H-space}. With this observation we continue the computing the last term:
\begin{align*}
&\sum_{i=1}^{i=m}\int_G b(\gamma)\,\epsilon_i(\gamma)\,\xi(\gamma\inverse x ) \Delta(\gamma,\gamma\inverse x) \left( \frac{1}{\bar{c}(\gamma\inverse x)}\int_H c(\gamma\inverse x \eta)\,\dd\beta^{s_X(x)}(\eta)\right)\;\,\dd\alpha^{r_X(x)}(y)\\
&= \sum_{i=1}^{i=m}\int_G b(\gamma)\epsilon_i(\gamma)\,\xi(\gamma\inverse x ) \Delta(\gamma,\gamma\inverse x)\,\dd\alpha^{r_X(x)}(y)\\
&= \int_G b(\gamma)\left(\sum_{i=1}^{i=m} \epsilon_i(\gamma)\right)\xi(\gamma\inverse x ) \Delta(\gamma,\gamma\inverse x)\,\dd\alpha^{r_X(x)}(y)=\sigma_G(b)\xi(x).
\end{align*}
To get the last equality we use the fact that \(\{\epsilon_i\}_{i=1}^m\) is a partition of unity subodinate the cover \(\{U_i\}_{i=1}^m\) of \(K\).

Thus, for given \(b\in\Contc(G)\), we have found finite natural numbers \(m\) and \(n\), and function \(F_{ij}'\in \Contc(X\times_{s_X,\base,s_X}X)\), where \(0\leq i \leq m\) and \(0\leq j\leq n\), such that
\[
\sigma_G(b)=\sum_{i,j=1}^{i=m,j=n}\sigma_{\TComp(X)}(B_{X\times_{s_X,\base[H],s_X}X}(F_{ij}'))
\]
where each \(\sigma_{\TComp(X)}(B_{X\times_{s_X,\base[H],s_X}X}(F_{ij}))\), as per Proposition~\ref{prop:implications-of-Cst-cat-2}, is a compact operator on \(\Hilm(X)\) .
\end{proof}

\begin{corollary}
  \label{cor:prop-corr-gives-kk-cycle}
  Let \((G,\alpha)\) and \((H,\beta)\) be localy compact goupoids equipped with Haar systems. Then a proper topological correspondence \((X,\lambda)\) from \((G,\alpha)\) to \((H,\beta)\) defines an element in \(\KK(\Cst(G,\alpha),\Cst(H,\beta))\).
\end{corollary}
\begin{proof}
  Theorem~\ref{thm:prop-corr-main-thm} says that \(\Hilm(X)\) is a proper \(\Cst\)\nb-correspondence from \(\Cst(G,\alpha)\) to \(\Cst(H,\beta)\). Hence \([0,\Hilm(X)]\in\KK(\Cst(G,\alpha),\Cst(H,\beta))\) where \(0\in \Bound_{\Cst(H,\beta)}(\Hilm(X))\) is the zero operator.
\end{proof}

Recall Section {2.5.2} in~\cite{mythesis} that discusses isomorphism of topological correspondence.
\begin{proposition}
  \label{prop:prop-iso-corr}
Let \((G,\alpha)\) and \((H,\beta)\) be locally compact groupoids, and let \((X,\lambda)\) and \((Y,\mu)\) be two correspondences from \((G,\alpha)\) to \((H,\beta)\). Let \(\phi\colon X\to Y\) be a \(G\)-\(H\)\nb-equivariant homeomorphism which implements an isomorphism of correspondences. Then, if \((X,\lambda)\) is proper, then so is \((Y,\beta)\).
\end{proposition}
\begin{proof}
 The isomorphism of correspondences \(\phi\) induces a \(H\)\nb-equivariant isomorphism of topological groupoids
\[
\phi'\colon G\ltimes X\to G\ltimes Y,\quad (\gamma,x)\mapsto (\gamma, \phi(x)).
\]
We also get the following commutative square
\begin{figure}[htb]
  \centering
\[
\begin{tikzcd}[scale=2]
 G\times_{\base}X\arrow{r}{\Mlp_{s_X}}\arrow{d}{\phi'}
 & X\times_{s_X,\base[H],s_X}X \arrow{d}{\phi\times_{\base[H]}\phi}\\
 G\times_{\base}Y \arrow{r}{\Mlp_{s_Y}} &  Y\times_{s_Y,\base[H],s_Y}Y&
 \end{tikzcd}
\]
\caption{}
\label{fig:iso-prop-corr-1}
\end{figure}
In this square, the vertical arrows are isomorphisms. We leave it as excercise to reader to prove that if \(\Mlp_{s_X}\) is a local homeomorphism onto its image, then so is \(\Mlp_{s_Y}\).

The map \(\phi\) induces a \(G\)\nb-equivariant isomorphism \([\phi]\colon X/H\to Y/H\). Furthermore, \([r_X]=[r_Y]\circ [\phi]\), that is, \([r_X]\circ [\phi]\inverse=[r_Y]\). Therefore, \([r_X]\) is a proper map if and only if\([r_Y]\)  is so. By hypothesis \([r_X]\) is proper, therefore \([r_Y]\) is proper.

Finally, let \(\Cov\) and \(\Cov'\) be open covers of \( G\times_{\base}X\) and \(X\times_{s_X,\base[H],s_X}\), respectively, which satisfies the last condition in Definition~\ref{def:prop-corr}. Then
\begin{align*}
  \Cov[B]\defeq& \{\phi'(U): U\in \Cov\},\\
  \Cov[B]'\defeq& \{\phi\times_{\base[H]}\phi(V): V\in \Cov'\}
\end{align*}
are open covers of \( G\times_{\base}Y\) and \(Y\times_{s_X,\base[H],s_X}\), respectively, which satisfy the last  condition in Definition~\ref{def:prop-corr}.

For \(u\in \base[H]\), let \(\phi_*(\lambda_u)\) be the measure induced by \(\lambda_u\) on \(Y\) via \(\phi\). Let \(\RN_X\) deonote the function \(\RN\) appearing in Definition~\ref{def:prop-corr}. Define \(\RN_Y(x,y)=\RN_X\circ \phi\inverse(x,y) \,\frac{\dd\phi_*(\lambda)_{s_X((x))}}{\dd\mu_{s_Y(x)}}(y)\).

Let \(f\in \Contc( G\times_{\base}Y,U)\) where \(U\in \mathcal{A}'\), and let \(f'\) be an extension of it in \(\Ext_{\Mlp_Y}(f)\). Observe that \(f'\circ\phi'\) is an extension of \(f\circ \phi'\). Now
\begin{align*}
 & \int_Y f'(x,y)\,\RN_Y(x,y)\;\dd\lambda_{s_Y(x)}(y)\\
 &=\int_Y f'(x,y)\,\RN_X\circ\phi\inverse(x,y)\frac{\dd\phi_*(\lambda)_{s_X((x))}}{\dd\mu_{s_Y(x)}}(y)\;\dd\lambda_{s_Y(x)}(y)\\
&= \int_X f'\circ\phi' (a,b)\,\RN_Y(a,b)\;\dd\mu_{s_X(a)}(b)
\end{align*}
Since \((X,\lambda)\) is proper correspondence, using Equation~\eqref{eq:prop-corr} we see that the last term equals
\[
\int_G f\circ \phi' (\gamma, \gamma\inverse a)\,\dd\alpha^{r_X(a)}(\gamma)=\int_G f(\gamma, \gamma\inverse x)\,\dd\alpha^{r_Y(x)}(\gamma).
\]
The last equailty is due to the fact that \(\phi\) is a \(G\)\nb-equivariant isomorphism. This proves that \((Y,\mu)\) is proper.
\end{proof}

\section{Examples}
\label{sec:exa}

\begin{example}
  \label{exa:Tu}
In~\cite{Tu2004NonHausdorff-gpd-proper-actions-and-K}, Tu defines \emph{locally proper generalised morphism} and \emph{proper} locally proper generalised morphism from a locally compact groupoid \(G\) to another one, say, \(H\). This morphism is a locally compact \(G\)\nb-\(H\)-bispace \(X\) satisfying certain properties. 
We repeat Tu's definition of a locally proper generalised morphism and proper locally proper generalised morphism when \(X\) is Hausdorff; we modify the verbalisation to suit our setting, for the original definitions see~\cite{Tu2004NonHausdorff-gpd-proper-actions-and-K}*{Definitions 7.3 and 7.6}.
\begin{definition}[\cite{Tu2004NonHausdorff-gpd-proper-actions-and-K}*{Definitions 7.3 and 7.6} when \(X\) is Hausdorff]
\label{def:Tu-corr}
    A locally proper generalised morphism from a locally compact groupoid with Haar system $(G,\alpha)$ to a groupoid with Haar system $(H,\beta)$ is a $G$-$H$\nb-bispace $X$ such that
    \begin{enumerate}[label=\roman*)]
    \item the action of $G$ is proper and free,
    \item the right momentum map induces a homeomorphism \(G\7X\to \base[H]\)
    \item the action of $H$ is proper.
    \end{enumerate}
Furthermore, the locally proper generalised morphism \(X\) is called proper if
\begin{enumerate}[label=\roman*), resume]
\item \([r_X]\colon X/H\to\base\) is a proper map.
\end{enumerate}
 \end{definition}
\cite{Tu2004NonHausdorff-gpd-proper-actions-and-K}*{Theorem 7.8} proves that a proper correspondence \(X\) from \((G,\alpha)\) to \((H,\beta)\) produces a proper \(\Cst\)\nb-correspondence from \(\Cred(G,\alpha)\) to \(\Cred(H,\beta)\). The discussion on page 219 in the sixth paragraph in~\cite{Holkar2017Construction-of-Corr} and Example 3.8 in the same article show that a locally proper generalised morphism is a topological correspondence; this example also shows that the \(H\)\nb-invariant family of measures \(\lambda\) on \(X\) is given by
\begin{equation}
  \label{eq:Tu-corr-measures}
  \int_X f\,\dd\lambda_u\defeq \int_G f(\gamma\inverse x)\,\dd\alpha^{r_X(x)}(\gamma)
\end{equation}
where \(f\in \Contc(X)\), \(u\in \base[H]\) and \(x\in s_X\inverse(u)\). The same computation shows that \(\lambda_u\) is well-defined and is independent of the choice of \(y\).

Let \(X\) be a proper locally proper generalised morphism from a locally compact groupoid with Haar system \((G,\alpha)\) to \((H,\beta)\). We show that when \(X\) is Hausdorff and \(X/H\) is paracompact, \((X,\lambda)\) is a proper topological correspondence; here \(\lambda\) is the family of measures discussed above.

In this case, Remark~\ref{rem:relation-between-sX-quo-sX-Mlp} shows that the map \(\Mlp_{s_X}\) is a homeomorphism. Hence the first condition in Definition~\ref{def:prop-corr} is satisfied.

Condition~{(iv)} in Definition~\ref{def:Tu-corr} above agres with Definition~\ref{def:prop-corr}{(ii)}.

 Finally, to see that Definition~\ref{def:prop-corr}{(iii)} is satisfied, let \(\Cov\{G\times_{\base}X\}\) be the open cover of \(G\times_{\base}X\), and \(\Cov'=\{X \times_{s_X,\base,s_X}X\}\) be its extension via \(\Mlp_{s_X}\); recall from Remark~\ref{rem:relation-between-sX-quo-sX-Mlp} that in present situation \(\Mlp_{s_X}\) is a homeomorphism. Let \(f\in \Contc(G\times_{\base}X)\), then \(\{f\circ \Mlp_{S_X}\inverse\}= \Ext_{\Mlp_{s_X}}(f)\). Take \(\RN\) to be the constant function 1 on \(X\times_{s_X,\base[H],s_X}X\). We compute the left and right side terms in Equation~\eqref{eq:prop-corr}:
 \[
   \int_{X} f'(x,y)\;\RN(x,y)\;\dd\lambda_{s_X(x)}(y)=\int_{X} f\circ\Mlp_{s_X}\inverse(x,y)\;\dd\lambda_{s_X(x)}(y).
\]
Since \(\Mlp_{s_X}\) is a homeomorphism, there is unique element \(\gamma_{(x,y)}\in G\) such that \(\gamma_{(x,y)} y=x\), that is \(\gamma_{(x,y)}\inverse x=y\); we remove the suffix of \(\gamma\) for the sake of simplicity. Now we may write the right side term of the previous equation as
\[
\int_{X} f\circ\Mlp_{s_X}\inverse(x,\gamma\inverse x)\;\dd\lambda_{s_X(x)}(\gamma\inverse x).
\]
Using the definition of the measure \(\lambda_{s_X(x)}\) in Equation~\eqref{eq:Tu-corr-measures} we write the right side term in the last equation as
\[
\int_{X} f\circ\Mlp_{s_X}\inverse(x,\eta\inverse \gamma\inverse x)\;\dd\alpha_{s_G(\gamma)}(\eta).
\]
Now use Equation~\ref{eq:loc-sec-inv-x} to see \(f\circ\Mlp_{s_X}\inverse(x,\eta\inverse\gamma\inverse x)=f(\gamma\eta,\eta\inverse \gamma\inverse x)\). Put this value of \(f\circ\Mlp_{s_X}\inverse\) in the last integral, then it becomes
 \[
\int_{X} f(\gamma\eta,\eta\inverse \gamma\inverse x)\;\dd\alpha_{s_G(\gamma)}(\eta).
\]
Now change the variable \(\gamma\eta\mapsto \gamma\) and use the left invariance of the the Haar system \(\alpha\) to see that the above term equals
 \[
\int_{X} f(\gamma,\gamma\inverse x)\;\dd\alpha_{r_X(x)}(\gamma)
\]
which is the right hand side of Equation~\eqref{eq:prop-corr}.
\end{example}

\begin{example}(Groupoid equivalence)
  \label{exa:gpd-equi}
Recall the definition of groupoid equivalence~\cite{Muhly-Renault-Williams1987Gpd-equivalence}*{Definition 2.1}. Let \(X\) be a groupoid equivalence between locally compact groupoids \(G\) and \(H\). Let \(\alpha\) and \(\beta\) be the Haar systems on \(G\) and \(H\), respectively. Then \(X\) is a proper locally proper generalised morphism; a straightforward comparison between~\cite{Muhly-Renault-Williams1987Gpd-equivalence}{Definition 2.1} and Definition~\ref{def:Tu-corr}. Example~{3.9} shows this. Example~\ref{exa:Tu} now shows that when equipped with the family of measures \(\lambda\) in Equation~\ref{eq:Tu-corr-measures}, \((X,\lambda)\) is a proper correspondence. The function~\(\RN\) on \(X\times_{s_X,\base[H],s_X}X\) is the constant function 1.
\end{example}

\begin{example}
  \label{exa:cpt-gp-to-pt}
Let \(G\) be a locally compact group and \(\Pt\) the trivial group(oid). Let \(\alpha\) be the Haar system on \(G\), then \((G,\alpha)\) is a topological correspondence from \(G\) to \(\Pt\): \(G\) acts on itself with the left multiplication and \(\Pt\) acts on \(G\) trivially; \(\alpha\) is \((G,\alpha)\)\nb-quasi-invariant measure on \(G\). This is a special case of Example~\cite{Holkar2017Construction-of-Corr}*{Example 3.5} with the measure on the bisapce inverted; the map is the zero map \(\Pt\to G\). Assume that \(G\) is compact, then \((G,\alpha)\) is a proper correspondence. In this case, the first two conditions in Definition~\ref{def:prop-corr} are clearly satisfied. Finally, observe that the map \(\Mlp_0\colon G\times G\to G\times G\) is an isomorphism. To see that the last condition in Definition~\ref{def:prop-corr} holds, the computation is similar to the one in Example~\ref{exa:Tu} with \(X\) is replaced by \(G\).

To elaborate the algebraic side of this example, the topological correspondence \((G,\alpha)\) from \(G\) to \(\Pt\) gives the left regular representation of  \(\Cst(G)\)  on \(\mathcal{L}^2(G,\alpha)\). Now Theorem~\ref{thm:prop-corr-main-thm} reproduces the basic result that in this  case the the left regaular representation takes values in \(\Comp(\mathcal{L}^2(G,\alpha))\).
\end{example}

\begin{example}
  \label{exa:transversal}
  Let \((G,\alpha)\) be a locally compact groupoid with a Haar system, and \(H\subseteq G\) a closed subgroupoid. Let \(\beta\) be a Haar system on \(H\). Let \(X\defeq \{\gamma\in G: r_G(\gamma)\in \base[H]\}\). Popularly, \(X\) is denoted by \(G^{\base[H]}\). Then \(H\) and \(G\) act on \(X\) from left and right, respectively, by multiplication. The momentum maps for these actions are \(r_X=r_{G}|_{X}\) and \(s_X=s_G|_X\). Both the actions are free, and as~\cite{Holkar2017Construction-of-Corr}*{Example 3.14} shows, the actions are proper. In this case, \(X\) is a topological correspondence from \(H\) to \(G\) where the family of measures is given by Equation~\eqref{eq:Tu-corr-measures}; note that in this case the map \([s_X]\colon H\7X\to \base\) is need not be surjective. This is a correspondence similar to the one in~\cite{Holkar2017Construction-of-Corr}*{Example 3.14} with the direction inverted.

Since the action of \(H\) on \(X\) is proper, the map \(\textup{m}\colon H\times_{\base[H]}X\to X\times X\), \((\eta,x)\mapsto (\eta x,x)\), is proper. As \(\textup{m}\inverse(X \times_{s_X,\base[H],s_X}  X)=H\times_{\base[H]}X\), 
  
Proposition 10.1.3(a) in~\cite{Bourbaki1966Topologoy-Part-1}*{Chapter 1}) implies that \(\Mlp_{s_X}\) is proper. Now the fact that \(\Mlp_{s_X}\) is one-to-one and proper, using~\cite{Bourbaki1966Topologoy-Part-1}*{Chapter 1, Proposition 10.1.2}, we conclude that \(\Mlp_{s_X}\) is a homeomorphism onto a closed subspace of \( X \times_{s_X,\base[H],s_X}  X\). This shows that the first condition in Definition~\ref{def:prop-corr} holds.

  We observe that \([r_X]\colon X/G\to\base[H] \) is the identity map. Thus the second condition in Definition~\ref{def:prop-corr} is  satisfied. The third condition can be checked as in Example~\ref{exa:Tu}.
\end{example}

\begin{example}
  \label{exa:gp-to-sgp}
  Let \(G\) be a locally comapct group and \(H\) a closed subgroup of it. Assume that \(G\) is \(H\)\nb-compact, that is, \(G/H\) is a compact space. Let \(\alpha\)  and \(\beta\) be the Haar measure on \(G\) and \(H\), respectively. Then with the left and right multiplication actions, \((G,\alpha\inverse)\) is a topological correspondence from \(G\) to \(H\). Reader may check that this is an example of correspondences in Definition~\ref{def:Tu-corr}.

  In this case, \(\Mlp_0\colon G\times G\to G\times G\) is the map \((\gamma,\gamma')\mapsto (\gamma\gamma',\gamma')\) which is a homeomorphism; this fulfils the first condition in Definition~\ref{def:prop-corr}. Since \(G/H\) is compact, the constant map \(G/H\to \Pt\) is proper which fulfils the second condition for a proper correspondences. A computation as in Example~\ref{exa:Tu} shows that the third condition in Definition~\ref{def:prop-corr} holds. Thus \((G,\alpha\inverse)\colon G\to H\)  is a proper correspondence. A concrete example of this case is when \(G=\Z\) and \(H=n\Z\) where \(n=1,2,\dots\).
\end{example}

\subsection{The case of spaces}
\label{sec:ex-space-case}

A locally compact space \(X\) is a locally compact groupoid; in this groupoid every arrow is a unit arrow and the source as well as the range maps equal \(\Id_X\) the identity map on \(X\). Equip the groupoid \(X\) with the Haar system \(\delta^X\defeq\{\delta^X_x\}_{x\in X}\) where \(\delta_y^X\) is the point mass at \(x\in X\). Now on, whenever we think of a space as a locally compact groupoid equipped with \emph{the} (or standard or obvious) Haar system, we shall mean this setting. Note that \(X\) is an {\etale} groupoid. Let \(Y\) be another space. Then the only possible action of \(X\) on \(Y\) is the trivial action, that is, there is a momentum map \(b\colon X\to Y\) and for \((x,y)\in X\times_{\Id_X,X,b}Y\), the action is \(x\cdot y=y\). The actions of \(X\) on \(Y\)  are in one-to-one correspondence with the continuous maps from \(X\) to \(Y\).

\begin{lemma}
  \label{lem:prop-quiver-labda-is-atomic}
Let \(Y\xleftarrow{b} X\xrightarrow[\lambda]{f} Z\) be a topological correspondence from a space \(Y\) to a space \(Z\). If \((X,\lambda)\) is proper in the sense of Definition~\ref{def:prop-corr}, then
\begin{enumerate}[i)]
\item for each \(z\in Z\), the measure \(\lambda_z\) is atomic,
\item for each \(z\in Z\), each atom in \(\lambda_z\) is nonzero and the function \(\RN\) appearing in Definition~\ref{def:prop-corr}{(3)} positive. Thus the family of measure \(\lambda\) has full support.
\end{enumerate}
\end{lemma}
\begin{proof}
i): The discussion regarding Figure~\ref{fig:quiver-homeo} in Example~\ref{exa:proper-basic-corr} shows that the image of \(\Mlp_f\) is the diagonal \(\{(x,x)\in X\times_{f,Z,f}X:x\in X\}\) which is closed in \(X\times_{f,Z,f}X\) (since \(X\) is Hausdorff). Fix \(x\in X\). Then using Lemma~\ref{lem:closed-image-of-Mlp-measure-conc} we see that the measure \(\lambda_1^x\) is concentrated on the diagonal in \(X\times_{f,Z,f}X\). Note that the measure \(\lambda_1^x\) is defined on \(\pi_2\inverse(x)\) where \(\pi_2\) is the projection onto the second factor of \(X\times_{f,Z,f}X\). Thus \(\lambda_1^x\) is concentrated on \(\pi_2\inverse(x)\cap\Dia_X(X)\) which equals
\[
\{(x,x')\in X\times_{f,Z,f}X:
  f(x)=f(x')\}\cap \{(x'',x'')\in X\times_{f,Z,f}X: x''\in X\}=\{(x,x)\}.
\]
Thus
\[
\supp(\lambda_1^x)\subseteq \{(x,x)\}=\{x\}\times\{x\}
\]
for each \(x\in X\).
Now \(\lambda_1^x=\delta_x\otimes\lambda_{f(x)}\) hence
 \[
\supp(\lambda_1^x)= \supp(\delta_x)\times \supp(\lambda_{f(x)})=\{x\}\times\supp(\lambda_{f(x)}),
\]
due to~\cite{Bourbaki2004Integration-I-EN}*{Chapter III, \S4.2, Proposition 2, page III.44}. Hence \(\supp(\lambda_{f(x)})\subseteq \{x\}\) for each \(x\in X\).

 \noindent ii): Let \(x_0\in X\) be given. Let \(\tilde g\in \Contc(Y\times_{\Id_Y,Y,b}X)\) be a function which is nonnegative and \(\tilde g(b(x_0),x_0)>0\). Let \(\RN\) be the function in the third condition of the proper correspondence. Then Equation~\ref{eq:quiver-right-eq} says that 
\[
\int_Y \tilde g(y,y\inverse\cdot x)\,\dd\delta_Y^{b(x_0)}= \tilde g(b(x_0),x_0)>0.
\]
Let \(g'\in \Ext_f(\tilde g)\) be a nonnegative extension of \(\tilde g\). Since the third condition in Definition~\ref{def:prop-corr} holds, we get
\[
\int_X g'(x_0,x)\,\RN(x_0,x')\dd\lambda_{f(x_0)}(x)=\int_Y \tilde g(y,y\inverse\cdot x)\,\dd\delta_Y^{b(x_0)}= \tilde g(b(x_0),x_0)>0.
\]
Thus the measure \(\lambda_{f(x_0)}\neq 0\), that is, \(\lambda_z>0\) for each each \(z\in Z\). Furthermore, since the function \(g'\)  and the measure \(\lambda_{z}\) in the left hand side of the above equation are nonnegative, and the whole term is positive, we get \(\RN(x_0,x')>0\).
\end{proof}

Let \(X\) and \(Y\) be spaces and \(f\colon X\to Y\) be a (surjective) local homeomorphism. Then there is a canonical continuous family of measure along \(f\) which we denote by \(\tau^\phi\):
\begin{equation}
\label{eq:etale-measure-family}
\int_X g\,\tau^\phi_y=\sum_{x\in \supp(g)\cap f\inverse(y)}g(x)
\end{equation}
where \(g\in \Contc(X)\) and \(y\in Y\).

\begin{example}
  \label{exa:proper-basic-corr}
Let \(X,Y\) and \(Z\) be locally compact Hausdorff spaces, and  \(Y\xleftarrow{b} X\xrightarrow{f} Z\) continuous maps. Assume that \(f\) is a local homeomorphism. Then \((X.\tau^f)\) is a topological correspondence from the {\etale} groupoid \(X\) to \(Y\). The adjoining function of this correspondence is the constant function 1; this follows directly of one writes the Equation in Definition~{2.1~(iv)} in~\cite{Holkar2017Construction-of-Corr}, and keep in mind that the actions are trivial.

Additionally, assume that \(b\) is a proper map. Then \((X,\lambda)\) is a proper correspondence. Following are the details: in this case, \(Y\times_{Y}X=Y\times_{\Id_Y,Y,b}X\homeo X\). Hence \(\Mlp_f\colon Y\times_{Y}X\to X\times_{f,Z,f}X\) can be identified with the diagonal embedding (see Figure~\ref{fig:quiver-homeo}), thus the first condition in Definition~\ref{def:prop-corr} is satisfied. Since \(b\) is proper, and the action of \(Z\) on \(X\) is the trivial action, the second condition is also fulfilled. To see that the third condition is also satisfied, one identifies \(Y\times_{Y}X\) with \(X\), \(\Mlp_f\) with the diagonal embedding \(X\hookrightarrow X\times_{f,Z,f}X\), and use this embedding to prove the required claim.

 Let \(\theta\colon Y\times_{Y}X\to X\) be the the homeomorphism \(\theta\colon (b(x),x)\mapsto x\); inverse of \(\theta\) is \(x\mapsto (b(x),x)\). Let \(\Dia_X\colon X\to X\times_{f,Z,f}X\) be the diagonal embedding. We draw the commutative triangle in Figure~\ref{fig:quiver-homeo}.
\begin{figure}[htb]
  \centering
\[
\begin{tikzcd}[scale=2]
Y\times_{Y}X\arrow{r}{\Mlp_f}\arrow{d}{\homeo}[swap]{\theta}
 & X\times_{f,Z,f}X\\
 X\arrow{ru}[swap]{\Dia_X} &
 \end{tikzcd}
\]
\caption{}
\label{fig:quiver-homeo}
\end{figure}

Take the cover \(\Cov\) of \(Y\times_{Y}X\) to be 
\(
\{\theta\inverse(U): U\in f^{\HoC}\}.
\)
Let \(\Cov'\defeq\{U\times_{f,Z,f}U: U\in f^{\HoC}\}\). We observe that \(\Cov'\) is an extension of \(f^{\HoC}\) via \(\Dia_X\), and hence it is an extension of \(\Cov\) via \(\Mlp_f\). Note that for \(U',U\subseteq X\), \((U\times_{f,Z,f}U) \cap (X\times_{f,Z,f}X)=U'\) if and only if \(U=U'\). Thus every element in \(f^{\HoC}\) or \(\Cov\) has a unique extension in \(\Cov'\). Let \(\RN\) be the constant function 1. 

Let \(\theta\inverse(U)\in\Cov\) be given where \(U\in f^{\HoC}\). Let \(\tilde{g}\in \Contc(Y\times_{Y}X,\theta\inverse(U))\); then \(g\defeq \tilde{g}\circ \theta\inverse\in \Contc(X;U)\). Let \(g'\in \Contc(X\times_{f,Z,f}X; U\times_{f,Z,f} U)\) be an extension of \(g\) via \(\Dia_X\).Then \(g'\) is an extension of \(\tilde{g}\) also. For \(x\in X\), the right side of Equation~\eqref{eq:prop-corr} for \(\tilde{g}\) is
 \begin{equation}
\label{eq:quiver-right-eq}
\int_{Y} \tilde{g}(y,y\inverse\cdot x)\,\dd\delta_Y^{b(x)}(y)=\int_{Y} \tilde{g}(y,x)\,\dd\delta_Y^{b(x)}(y)= \tilde{g}(b(x), x).
\end{equation}

 For the same \(x\in X\) as above, the left side of Equation~\eqref{eq:prop-corr} for \(g'\) reads
\[
\int_{X} g'(x,x')\,\dd\tau^f_{f(x)}(x').
\]
note that \(g'\) is supported in \(U\times_{f,Z,f}U\), and \(f|_U\) is a homeomorphism. Therefore, the only element in \(U\) which hits \(f(x)\) under \(f\) is \(x\) itself. Now recall that \(g'\) is an extension of \(\tilde{g}\) via \(\Mlp_f\), hence \(g'\circ\Dia_X=\tilde{g}\). Thus
\[
\int_{X} g'(x,x')\,\dd\tau^f_{f(x)}(x')=g'(x,x)=g'\circ\Mlp_f(b(x),x)=\tilde{g}(b(x),x);
\]
the claim follows by comparing the above equation with Equation~\eqref{eq:quiver-right-eq}.
\end{example}

\begin{example}
\label{exa:prop-quiver}
  Let \(Y\) and \(Z\) be spaces considered as locally compact groupoids equipped with the canonical Haar system. 
let \((X,\lambda)\) be a topological correspondence from \(y\) to \(z\), see~\cite{Holkar2017Construction-of-Corr}*{example 3.3}. thus there are momentum maps \(Y\xleftarrow{b} X\xrightarrow{f} Z\) and a continuous family of measures \(\lambda\) along \(f\).  The topological correspondence \((X,\lambda)\) gives a \(\Cst\)\nb-correspondence \(\Hilm(X)\) from \(\Contz(Y)\) to \(\Contz(Z)\). We use the shorthand notation \(Y\xleftarrow{b} X\xrightarrow[\lambda]{f} Z\) to write this correspondence.

When the spaces \(Y\) and \(Z\) above are same and the family of measures \(\lambda\) has full support, Muhly and Tomford call this topological correspondence a \emph{topological quiver}, see~\cite{Muhly-Tomforde-2005-Topological-quivers}*{Definition 3.1}. They call the quiver \((X,\lambda)\) proper if \(b\) is a proper map, and \(f\) is a local homeomorphism. \cite{Muhly-Tomforde-2005-Topological-quivers}*{Theorem 3.11} asserts that the \(\Cst\)\nb-correspondence \(\Hilm(X)\) from \(\Contz(Y)\) to itself is proper if and only if the quiver \((X,\lambda)\) is proper. One can readily see that that  \cite{Muhly-Tomforde-2005-Topological-quivers}*{Theorem 3.11} is valid for a topological correspondence of spaces if the family of measures on the middle space has full support.
\medskip

Returning to our example of the topological correspondences \(Y\xleftarrow{b} X\xrightarrow[\lambda]{f} Z\), assume that \(\lambda\) has full support, \(f\) is a local homeomorphism and \(b\) is proper. Then Example~\ref{exa:proper-basic-corr} shows that \((X,\tau^f)\) is a proper correspondence from \(Y\) to \(Z\). We show that the topological correspondences \((X,\lambda)\) and \((X,\tau^f)\) are isomorphic. Then Proposition~\ref{prop:prop-iso-corr} implies that \((X,\lambda)\) is proper.

 We claim that the identity map, \(\Id_X\colon X\to X\), gives the isomorphism  \((X,\lambda)\to (X,\tau^f)\) of correspondences. Define 
\[
D\colon X \to \R, \quad D(x)= \lambda_{f(x)}(x)\;.
\]
Then \(D\) is positive since \(\lambda\) has full support and continuous since \(\lambda\) is continuous. We note that
\[
D(x')=\frac{\dd\lambda_{f(x)}}{\dd\delta_{f(x)}}(x')
\]
where the latter term is the Radon-Nikodym derivative of \(\lambda_{f(x)}\) with respect to \(\delta_{f(x)}={(\Id_X)}_*(\delta_{f(x)})\). This implies that  \({(\Id_X)}_*(\delta_{z})\sim \lambda_{z}\) for every \(z\in Z\) and the Radon-Nikodym derivative is continuous. 
Thus \(\Id_X\) gives the isomorphism of the correspondences \((X,\lambda)\to (X,\tau^f)\), and the corresponding positive Radon-Nikodym derivative is given by the function \(D\) above.

The isomorphism of correspondences in this example generalises to {\etale} groupoids, see Proposition~\ref{prop:gen-etale-corr-iso}.
\end{example}
\medskip

\cite{Muhly-Tomforde-2005-Topological-quivers}*{Theorem 3.11} implies that if \(Y\xleftarrow{b} X\xrightarrow[\lambda]{f} Z\) is a topological correspondence, then the corresponding \(\Cst\)\nb-correspondence is proper if and only if the map \([r_X]\colon X/H\to \base\) is proper, and the right momentum map \(s_X\) is a local homeomorphism.

\begin{proposition}
  \label{prop:etale-corr-iso-kk-cycles}
Let \(X,Y\) and \(Z\) be spaces, \(Y\xleftarrow{b}X\xrightarrow{f}Z\) maps and \(\lambda\) and \(\lambda'\) two continuous families of measures along \(f\) having full support.  Assume that \(f\) is a local homeomorphism and \(b\) a proper map. Thus \((X,\lambda)\) and \((X,\lambda')\) are topological correspondences from \(Y\to Z\) (Example~\ref{exa:prop-quiver}).  Then in  \(\KK(\Contz(Y),\Contz(Z))\), \([(\Hilm(X)),0]=[(\Hilm(X')),0]\).
\end{proposition}
\begin{proof}
Example~\ref{exa:prop-quiver} shows that \((X,\lambda)\) and \((X',\lambda')\) are isomorphic to \((X,\tau^f)\). Hence the required {\KK}-classes are same.
\end{proof}

 \subsection{The case of {\etale} groupoids}
 \label{sec:ex-egpd-case}

 Let \(G\) be an {\etale} groupoid. Assume that \(\base\) is locally compact Hausdorff subspace of \(G\) and that \(G\) is covered by countably many compact sets the interiors of which form a basis for the topology of \(G\). Then \(G\) is a locally compact groupoid in the sense of~\cite{PatersonA1999Gpd-InverseSemigps-Operator-Alg}*{Definition 2.2.1}. In this section, the term \emph{{\etale} groupoid} shall stand for groupoids which are {\etale} and fulfil the above conditions. In this case, \(G\) is a locally compact groupoid equipped with a Haar system, and it has a nice approximate identity: one writes \(\base\) as an increasing union of compact sets, \(K_n\) where \(n\in\N\), whose interiors cover \(\base\). Then using Urysohn's lemma, one gets functions \(u_n\in \Contc(\base)\) for each \( n\in\N\) such that each \(0\leq f_n\leq 1\) and \(f_n=1\) on \(K_n\). The sequence \(\{f_n\}_{n\in\N}\) is an approximate identity for the {\Star}algebra \(\Contc(G)\), for details see~\cite{PatersonA1999Gpd-InverseSemigps-Operator-Alg}*{pages 47--49}.

For an {\etale} groupoid \(G\), \(\Cst(\base)=\Contz(\base)\) is a subalgebra of \(\Cst(G)\), and the map \(\Contc(G)\to\Contc(\base)\), \(f\mapsto f|_{\base}\), of {\Star}algebras induces an expectation from \(\Cst(G)\to \Contz(\base)\), see~\cite{Renault1980Gpd-Cst-Alg}*{page 61} for details.
 
In this section, during computations the canonical Haar system on the {\etale} groupoid \(G\) shall be denoted by \(\alpha\) and the one on \(H\) by \(\beta\).

 Let \(G\) and \(H\) be {\etale} groupoids and \(X\) a \(G\)-\(H\)-bispace. Assume that the momentum map \(s_X\colon X\to \base[H]\) is a local homeomorphism. Let \(\tau^{s_X}\) be the family of point-mass measures along \(s_X\) (see Equation~\eqref{eq:etale-measure-family}). Assume that \((X,\tau^{s_X})\) is a topological correspondence, that is, the action of \(H\) is proper and each measure in \(\tau^{s_X}\) is \(G\)\nb-quasi-invariant. Then we call \(X\), or to be precise \((X,\tau_{s_X})\), an \emph{{\etale} correspondence} from \(G\) to \(H\). The adjoining function of this correspondence is the constant function 1. This is because for each \(u\in \base[H]\), the measure \(\tau_u^{s_X}\) on the space of units of the transformation groupoid \(G\ltimes X\) is quasi-invariant; here \(G\ltimes X\) is an {\etale} groupoid. And the adjoining function is the modular function of \(G\ltimes X\) for \(\tau^{s_X}_u\) which is constant function 1 due to~\cite{Renault1980Gpd-Cst-Alg}*{1.3.22}.

 \begin{proposition}
   \label{prop:gen-etale-corr-iso}
   Let \(G\) and \(H\) be {\etale} groupoids and \(X\) a \(G\)\nb-\(H\)-bispace. Assume that the right momentum map \(s_X\colon X\to \base[H]\) is a local homeomorphism. Let \(\lambda\) be a continuous family of measures with full support  along \(s_X\). Assume that \((X,\lambda)\) is a topological correspondence from \(G\) to \(H\). Let \(\tau^{s_X}\) be the family of measures consisting of point-masses along \(s_X\). Then \((X,\tau^{s_X})\) is a topological correspondence from \(G\) to \(H\) and the identity map \(\Id_X\colon X\to X\) induces an isomorphism of topological correspondences.
 \end{proposition}
 \begin{proof}
   Since the family of measures \(\lambda\) is continuous and has full support, the function
   \[
D\colon X\to \R^+\quad D\colon x\mapsto \lambda(x)
\]
is continuous. Since \(\lambda\) is \(H\)\nb-invariant, so is \(D\). Then
\[
  D\colon \Contc(X)\to \Contc(X) \quad \colon f\mapsto fD
\]
is an isomorphism of complex vector spaces.

Now we note that for \(f\in\Contc(X)\) and every \(u\in \base[H]\),

\begin{equation}
  \label{eq:gen-etale-corr-iso}
\lambda(f)=\tau^{s_X}(f\lambda(x))(u)=\tau^{s_X}(fD)(u).
\end{equation}

Thus \(\lambda_u\sim\tau^{s_X}_u\) for every \(u\in\base[H]\), and the Radon-Nikodym derivative \(\frac{\lambda_u}{\tau^{s_X}_u}(x)=\lambda(x)\) is continuous. Now it follows from~\cite{Renault1980Gpd-Cst-Alg}*{Proposition 1.3.3(ii)} that \(\tau^{s_X}_u\) is \(G\ltimes X\)-quasi-invariant for each \(u\in\base[H]\). This shows that \((X,\tau^{s_X})\) is a topological correspondence from \(G\) to \(H\). Since \(\lambda_u\sim\tau^{s_X}_u\), which follows from Equation~\eqref{eq:gen-etale-corr-iso}, we have also proved that the identity map \(X\to X\) is an isomorphism of the topological correspondences \((X,\lambda)\) and \((X,\tau^{s_X})\).
\end{proof}
Converse of Proposition~\ref{prop:gen-etale-corr-iso} clearly holds. Let \(G\)  and \(H\) be {\etale} groupoids, \(X\) a \(G\)\nb-\(H\)-bispace and \(s_X\colon X\to \base[H]\) a local homeomorphism such that \((X,\tau^{s_X})\colon G\to H\) is a  topological correspondences. If \(D\colon X\to \R\) is an \(H\)\nb-invariant positive continuous function, then \(D\tau^{s_X}\defeq\{D\tau^{s_X}_u\}_{u\in\base[H]}\) is an \(H\)\nb-invariant family of measures, and since each measure \(D\tau^{s_X}_u\) in this family is equivalent to \(\tau^{s_X}_u\),~\cite{Renault1980Gpd-Cst-Alg}*{Proposition 1.3.3 (ii)} implies that \(D\tau^{s_X}_u\) is \(G\ltimes X\)\nb-quasi-invariant measure on its space of units.
\smallskip

Let \(X\) be a \(G\)\nb-\(H\)-bispace and let \(s_X\colon X\to \base[H]\) be a local homeomorphism. Assume that \((X,\tau^{s_X})\colon G\to H\) is a topological correspondence. Then, Proposition~\ref{prop:etale-gpd-loc-free-act} implies that, this correspondence is proper if \(\Mlp_{s_X}\colon G\times_{\base} X\to X\times_{s_X,\base[H],s_X} X\) is open and the rest two conditions in Definition~\ref{def:prop-corr} hold. Due to Proposition~\ref{prop:prop-iso-corr}, the correspondence \((X,\lambda)\) in Proposition~\ref{prop:gen-etale-corr-iso} above is proper if and only if \((X,\tau^{s_X})\) is proper. Thus for a correspondence \((X,\lambda)\) as in Proposition~\ref{prop:gen-etale-corr-iso}, the space \(X\) determines the correspondence upto isomorphism of correspondences.

Clearly, the correspondences of spaces in  Examples~\ref{exa:proper-basic-corr} and~\ref{exa:prop-quiver} are {\etale} correspondences. Following are a few more examples of {\etale} correspondences.

\begin{example}
  \label{exa:space-to-gpd-etale-corr}
  Let \(Y\) be a space considered as groupoid, \(H\) an {\etale} groupoid and \(X\) a \(Y\)\nb-\(H\)-bispace. Let \(s_X\colon X\to\base[H]\) be a local homeomorphism. Then \((X,\tau_{s_X})\colon Y\to H\) is a topological correspondence. This correspondence is proper if the right momentum map \(r_X\colon X\to Y\) is proper.

To confirm that \((X,\tau_{s_X})\) is a topological correspondence, one needs to check that \(\tau_{s_X}^u\) is \(Y\)\nb-quasi-invariant. This follows directly if one substitutes appropriate values in the required equation in~\cite{Holkar2017Construction-of-Corr}*{Definition 2.1}.
  
Now assume that \(r_X\) is proper. To see that this correspondence is proper, we need to check that Definition~\ref{def:prop-corr}{(iii)} holds as other conditions obviously hold. The map \(\Mlp_{s_X}\colon Y\times_{\Id_Y,Y,r_X} X\colon X\times_{s_X,\base[H],s_X}X\) is \((f(x),x)\mapsto (x,x)\), and checking the necessary condition is same as in Example~\ref{exa:proper-basic-corr}. Thus the proof follows from the same computation as in Example~\ref{exa:proper-basic-corr}. This examples ends here.
\end{example}

Let \(G\) and \(H\) be {\etale} groupoids and \(X\) an {\etale} correspondence from \(G\) to \(H\). Then \(\Hilm(X)\) is the Hilbert \(\Cst(H)\)\nb-module; let \(\sigma_G\colon \Cst(G)\to\Cst(H)\) be the representation which gives the \(\Cst\)\nb-correspondence. One may restrict the codomain of \(r_X\) to \(\base\) and consider \(X\) as an {\etale} correspondence from \(\base\) to \(H\); we call this correspondence as the one obtained from \(X\colon G\to H\) by restricting to \(\base\).  Note that in this  case, \(\Hilm(X)\) is the Hilbert \(\Cst(H)\)\nb-module. And since \(\Contz(\base)\subseteq \Cst(G)\) is a \(\Cst\)\nb-subalgebra,  the representation \(\sigma_{\base[G]}\colon \Cst(G)\to \Hilm(X)\) is obtained by restricting \(\sigma_G\) to \(\Contz(\base[G])\). Thus 
\begin{equation}
  \label{eq:space-to-e-gpd-corr}
\left\{\begin{aligned}
     X&\colon G\to H &\text{gives } \quad (\Hilm(X),\sigma_G)&\colon \Cst(G)\to \Cst(H),\\
     X&\colon \base\to H &\text{gives }\quad (\Hilm(X), \sigma_{\base})&\colon \Contz(\base)\to \Cst(H).
   \end{aligned}\right.
    \end{equation}

\begin{proposition}\label{prop:prop-etale-cor-char}
  Let \(G\) and \(H\) be {\etale} groupoids and \(X\) an {\etale} correspondence from \(G\) to \(H\). Then the following hold:
  \begin{enumerate}[i)]
  \item The \(\Cst\)\nb-correspondence \((\Hilm(X),\sigma_G)\colon \Cst(G)\to \Cst(H)\) is proper if and only if \(X\colon \base\to H\) obtained by restricting \(X\colon G\to H\) to \(\base\) is proper in the sense of Definition~\ref{def:prop-corr}.
  \item If \(r_X\colon X\to \base\) is a proper map, then  the \(\Cst\)\nb-correspondence \((\Hilm(X),\sigma_G)\) from \(\Cst(G)\) to \(\Cst(H)\) is proper.
  \end{enumerate}
\end{proposition}
\begin{proof}
  (i): Let \(\iota\colon \Contz(\base)\hookrightarrow \Cst(G)\) be the inclusion. Then the representations in Equation~\eqref{eq:space-to-e-gpd-corr} are related as \(\sigma_{\base}=\sigma_G\circ \iota\); this can be checked by writing formulae for the representations of  the pre-\(\Cst\)\nb-algebras \(\Contc(\base)\) and  \(\Contc(G)\) on the pre-Hilbert \(\Cst(H)\)\nb-module \(\Contc(X)\). Therefore, if \(\sigma_G(\Cst(G))\subseteq \Comp(\Hilm(X))\), then \(\sigma_{\base}(\Contz(\base))\subseteq \Comp(\Hilm(X))\).

  Conversely, assume that \(X\colon \base\to H\) is proper. Then Theorem~\ref{thm:prop-corr-main-thm} says that \(\sigma_{\base}(\Contz(\base))\subseteq \Comp(\Hilm(X))\). Let \(\{u_i\}\) be the approximate unit of the pre-\(\Cst\)-algebra \(\Contc(G)\) where each \(u_i\in\Contc(\base)\); such an approximate unit is discussed at the beginning of this section. Note that for \(u\in \Contc(\base)\), \(\sigma_G(u)(\xi)=\sigma_{\base}(u)(\xi)\) for every \(\xi\in\Contc(X)\). Let \(f\in \Contc(G)\). Then for every \(\xi\in\Contc(X)\),  we have
  \[
\sigma_G(f)(\xi)=\lim_i\sigma_G(f u_i)(\xi)=\lim_i\sigma_G(f)\sigma_{G}(u_i)(\xi)
\]
where each \(\sigma_G(f)\sigma_{G}(u_i)\in \Comp(\Hilm(X))\). Therefore, \(\sigma_G(f)=\lim_i\sigma_G(f)\sigma_{G}(u_i)\) is also in \(\Comp(\Hilm(X))\).

\noindent (ii): If \(r_X\colon X\to \base\) is a proper map, then Examples~\ref{exa:space-to-gpd-etale-corr} says that \(X\colon \base\to H\) is a proper correspondence in the sense of Definition~\ref{def:prop-corr}. Theorem~\ref{thm:prop-corr-main-thm} implies that
\((\Hilm(X),\sigma_G)\colon \Cst(G)\to \Cst(H)\) is a proper \(\Cst\)\nb-correspondence.
\end{proof}

\begin{lemma}
  \label{lem:etale-corr-full-supp}
Let \(G\) and \(H\) be {\etale} groupoids. Let \((X,\lambda)\colon Y\to H\) be a proper topological correspondence in the sense of Definition~\ref{def:prop-corr}, then
\begin{enumerate}[i)]
\item for each \(z\in Z\), the measure \(\lambda_z\) is atomic,
\item for each \(z\in Z\), each atom in \(\lambda_z\) is nonzero and the function \(\RN\) appearing in Definition~\ref{def:prop-corr}{(3)} positive. Thus the family of measure \(\lambda\) has full support.
\end{enumerate}
\end{lemma}
\begin{proof}
In this case, since \(r_X\colon X\to\base\) is a proper map, Examples~\ref{exa:space-to-gpd-etale-corr}  says that the correspondence \(X\colon\base\to H\) obtained by restricting \(X\colon G\to H\) is proper. We prove the claims of present lemma for the restricted correspondence. If one writes the map \(\Mlp_{s_X}\) for the restricted correspondence, then one observes that the proof of both the claims run on the same lines as the one of Lemma~\ref{lem:prop-quiver-labda-is-atomic}.
\end{proof}

\paragraph{\bfseries Acknowledgement:} We are grateful to Ralf Meyer, Suliman Albandik and Alcides Buss for many fruitful discussions. Special thanks to S.\,Albandik for many motivating discussions. We thank Ralf Meyer and Jean Renault for checking the manuscript, pointing out mistakes and suggesting improvements.
This work is done during the author's stay in the Federal University of Santa Caterina, Florianop{\'o}lis, Brazil and Indian Institute of Science Education and Research, Pune, India. The author is thankful to both the institutes for their hospitality. We are thankful to CNPq, Brazil and SERB, India whose funding made this work possible.


\begin{bibdiv}
\begin{biblist}

\bib{Bourbaki1966Topologoy-Part-1}{book}{
      author={Bourbaki, Nicolas},
       title={Elements of mathematics. {G}eneral topology. {P}art 1},
   publisher={Hermann, Paris; Addison-Wesley Publishing Co., Reading,
  Mass.-London-Don Mills, Ont.},
        date={1966},
      review={\MR{0205210 (34 \#5044a)}},
}

\bib{Bourbaki2004Integration-I-EN}{book}{
      author={Bourbaki, Nicolas},
       title={Integration. {I}. {C}hapters 1--6},
      series={Elements of Mathematics (Berlin)},
   publisher={Springer-Verlag, Berlin},
        date={2004},
        ISBN={3-540-41129-1},
        note={Translated from the 1959, 1965 and 1967 French originals by
  Sterling K. Berberian},
      review={\MR{2018901 (2004i:28001)}},
}

\bib{Buss-Holkar-Meyer2016Gpd-Uni-I}{article}{
      author={Buss, Alcides},
      author={Holkar, Rohit~Dilip},
      author={Meyer, Ralf},
       title={A universal property for groupoid {\(c^*\)}-algebras},
        date={2016},
        note={Preprint- arxiv:1612.04963v1},
}

\bib{Candel-Conlon2000Foliations-I}{book}{
      author={Candel, Alberto},
      author={Conlon, Lawrence},
       title={Foliations. {I}},
      series={Graduate Studies in Mathematics},
   publisher={American Mathematical Society, Providence, RI},
        date={2000},
      volume={23},
        ISBN={0-8218-0809-5},
      review={\MR{1732868}},
}

\bib{Folland1995Harmonic-analysis-book}{book}{
      author={Folland, Gerald~B.},
       title={A course in abstract harmonic analysis},
      series={Studies in Advanced Mathematics},
   publisher={CRC Press, Boca Raton, FL},
        date={1995},
        ISBN={0-8493-8490-7},
      review={\MR{1397028 (98c:43001)}},
}

\bib{Hilsum-Skandalis1987Morphismes-K-orientes-deSpaces-deFeuilles-etFuncto-KK}{article}{
      author={Hilsum, Michel},
      author={Skandalis, Georges},
       title={Morphismes {$K$}-orient\'es d'espaces de feuilles et
  fonctorialit\'e en th\'eorie de {K}asparov (d'apr\`es une conjecture d'{A}.
  {C}onnes)},
        date={1987},
        ISSN={0012-9593},
     journal={Ann. Sci. \'Ecole Norm. Sup. (4)},
      volume={20},
      number={3},
       pages={325\ndash 390},
         url={http://www.numdam.org/item?id=ASENS_1987_4_20_3_325_0},
      review={\MR{925720 (90a:58169)}},
}

\bib{mythesis}{thesis}{
      author={Holkar, Rohit~Dilip},
       title={Topological construction of $\textup{C}^*$-correspondences for
  groupoid ${C}^*$-algebras},
        type={Ph.D. Thesis},
        date={2014},
}

\bib{Holkar2017Composition-of-Corr}{article}{
      author={Holkar, Rohit~Dilip},
       title={Composition of topological correspondences},
        date={2017},
     journal={Journal of {O}perator {T}heory},
      volume={78.1},
       pages={89\ndash 117},
}

\bib{Holkar2017Construction-of-Corr}{article}{
      author={Holkar, Rohit~Dilip},
       title={Topological construction of {$C^*$}-correspondences for groupoid
  {$C^*$}-algebras},
        date={2017},
     journal={Journal of {O}perator {T}heory},
      volume={77:1},
      number={23-24},
       pages={217\ndash 241},
}

\bib{Holkar-Renault2013Hypergpd-Cst-Alg}{article}{
      author={Holkar, Rohit~Dilip},
      author={Renault, Jean},
       title={Hypergroupoids and {$C^*$}-algebras},
        date={2013},
        ISSN={1631-073X},
     journal={C. R. Math. Acad. Sci. Paris},
      volume={351},
      number={23-24},
       pages={911\ndash 914},
         url={http://dx.doi.org/10.1016/j.crma.2013.11.003},
      review={\MR{3133603}},
}

\bib{Khoshkam-Skandalis2002Reg-Rep-of-Gpd-Cst-Alg-and-Applications}{article}{
      author={Khoshkam, Mahmood},
      author={Skandalis, Georges},
       title={Regular representation of groupoid {$C^*$}-algebras and
  applications to inverse semigroups},
        date={2002},
        ISSN={0075-4102},
     journal={J. Reine Angew. Math.},
      volume={546},
       pages={47\ndash 72},
         url={http://dx.doi.org/10.1515/crll.2002.045},
      review={\MR{1900993}},
}

\bib{Stadler-Ouchi1999Gpd-correspondences}{article}{
      author={Macho~Stadler, Marta},
      author={O'uchi, Moto},
       title={Correspondence of groupoid {$C^\ast$}-algebras},
        date={1999},
        ISSN={0379-4024},
     journal={J. Operator Theory},
      volume={42},
      number={1},
       pages={103\ndash 119},
      review={\MR{1694789 (2000f:46077)}},
}

\bib{Meyer-Zhu2014Groupoids-in-categories-with-pretopology}{article}{
      author={{Meyer}, R.},
      author={{Zhu}, C.},
       title={{Groupoids in categories with pretopology}},
        date={2014-08},
     journal={ArXiv e-print},
      eprint={1408.5220},
}

\bib{Muhly-Renault-Williams1987Gpd-equivalence}{article}{
      author={Muhly, Paul~S.},
      author={Renault, Jean~N.},
      author={Williams, Dana~P.},
       title={Equivalence and isomorphism for groupoid {$C^\ast$}-algebras},
        date={1987},
        ISSN={0379-4024},
     journal={J. Operator Theory},
      volume={17},
      number={1},
       pages={3\ndash 22},
      review={\MR{873460 (88h:46123)}},
}

\bib{Muhly-Tomforde-2005-Topological-quivers}{article}{
      author={Muhly, Paul~S.},
      author={Tomforde, Mark},
       title={Topological quivers},
        date={2005},
        ISSN={0129-167X},
     journal={Internat. J. Math.},
      volume={16},
      number={7},
       pages={693\ndash 755},
         url={http://dx.doi.org/10.1142/S0129167X05003077},
      review={\MR{2158956 (2006i:46099)}},
}

\bib{PatersonA1999Gpd-InverseSemigps-Operator-Alg}{book}{
      author={Paterson, Alan L.~T.},
       title={Groupoids, inverse semigroups, and their operator algebras},
      series={Progress in Mathematics},
   publisher={Birkh\"auser Boston, Inc., Boston, MA},
        date={1999},
      volume={170},
        ISBN={0-8176-4051-7},
         url={http://dx.doi.org/10.1007/978-1-4612-1774-9},
      review={\MR{1724106 (2001a:22003)}},
}

\bib{Renault1980Gpd-Cst-Alg}{book}{
      author={Renault, Jean},
       title={A groupoid approach to {$C^{\ast} $}-algebras},
      series={Lecture Notes in Mathematics},
   publisher={Springer, Berlin},
        date={1980},
      volume={793},
        ISBN={3-540-09977-8},
      review={\MR{584266 (82h:46075)}},
}

\bib{Renault1985Representations-of-crossed-product-of-gpd-Cst-Alg}{article}{
      author={Renault, Jean},
       title={Repr\'esentation des produits crois\'es d'alg\`ebres de
  groupo\"\i des},
        date={1987},
        ISSN={0379-4024},
     journal={J. Operator Theory},
      volume={18},
      number={1},
       pages={67\ndash 97},
      review={\MR{912813 (89g:46108)}},
}

\bib{Renault2014Induced-rep-and-Hpgpd}{article}{
      author={Renault, Jean},
       title={Induced representations and hypergroupoids},
        date={2014},
        ISSN={1815-0659},
     journal={SIGMA Symmetry Integrability Geom. Methods Appl.},
      volume={10},
       pages={Paper 057, 18},
      review={\MR{3226993}},
}

\bib{Tu2004NonHausdorff-gpd-proper-actions-and-K}{article}{
      author={Tu, Jean-Louis},
       title={Non-{H}ausdorff groupoids, proper actions and {$K$}-theory},
        date={2004},
        ISSN={1431-0635},
     journal={Doc. Math.},
      volume={9},
       pages={565\ndash 597 (electronic)},
      review={\MR{2117427 (2005h:22004)}},
}

\end{biblist}
\end{bibdiv}
\end{document}